\documentclass[10pt,oneside]{article}
\usepackage[T1]{fontenc}
\usepackage[english]{babel}

\usepackage{graphicx,accents}

\newcommand{\cev}[1]{\reflectbox{\ensuremath{\vec{\reflectbox{\ensuremath{#1}}}}}}

\usepackage[cal=cm]{mathalfa}

\usepackage[T1]{fontenc}
\usepackage{titlesec}

\usepackage{hhline}

\titleformat
{\chapter} % command
[display] % shape
{\bfseries\Large\itshape} % format
{Story No. \ \thechapter} % label
{0.5ex} % sep
{
    \rule{\textwidth}{1pt}
    \vspace{1ex}
    \centering
} % before-code
[
\vspace{-0.5ex}%
\rule{\textwidth}{0.3pt}
] % after-code

\titleformat{\section}[wrap]
{\normalfont\bfseries}
{\thesection.}{0.5em}{}

\titlespacing{\section}{12pc}{1.5ex plus .1ex minus .2ex}{1pc}

\usepackage{tocloft} 

\usepackage{titlesec}
\titleformat{\subsection}[runin]% runin puts it in the same paragraph
       {\normalfont\bfseries}% formatting commands to apply to the whole heading
       {\thesubsection}% the label and number
       {0.5em}% space between label/number and subsection title
       {}% formatting commands applied just to subsection title
       [.]% punctuation or other commands following subsection title

\makeatletter
\@ifundefined{dddot}{}{}
\usepackage{mathdots}  % or yhmath
\makeatother

\makeatletter
\@ifundefined{ddddot}{}{}
\usepackage{yhmath}
\makeatother

\usepackage{geometry} % page geometry
\usepackage{xcolor} % colors
\usepackage{tikz} % drawing
\usepackage[draft=false]{hyperref} % hyperlinks
\usepackage{mathtools}
\usepackage{amsmath,amsthm,amssymb}
\usepackage[all,cmtip]{xy}
\usepackage{fullpage}

\usepackage[capitalise]{cleveref}  % load after hyperref

\numberwithin{equation}{section} %% Comment out for sequentially-numbered
\numberwithin{figure}{section} %% Comment out for sequentially-numbered

\theoremstyle{plain}

\newtheorem{theorem}{Theorem}[subsection]

\newtheorem{lemma}[theorem]{Lemma}
\newtheorem{proposition}[theorem]{Proposition}
\newtheorem{corollary}[theorem]{Corollary}

\theoremstyle{definition}

\newtheorem{definition}[theorem]{Definition}

\newtheorem{notation}[theorem]{Notation}
\newtheorem{remark}[theorem]{Remark}
\newtheorem{example}[theorem]{Example}

\newtheorem*{thm*}{Theorem}
\newtheorem*{prop*}{Proposition}

\crefname{theorem}{Theorem}{Theorems}
\crefname{lemma}{Lemma}{Lemmas}
\crefname{proposition}{Proposition}{Propositions}
\crefname{corollary}{Corollary}{Corollaries}
\crefname{conjecture}{Conjecture}{Conjectures}
\crefname{convention}{Convention}{Conventions}
\crefname{construction}{Construction}{Constructions}
\crefname{observation}{Observation}{Observations}
\crefname{warning}{Warning}{Warnings}
\crefname{definition}{Definition}{Definitions}
\crefname{assumption}{Assumption}{Assumptions}
\crefname{notation}{Notation}{Notations}
\crefname{remark}{Remark}{Remarks}
\crefname{example}{Example}{Examples}
\crefname{summary}{Summary}{Summaries}

\usepackage{amsmath,amsthm,amssymb,amsfonts,extarrows}
\usepackage{enumerate,amscd,amsxtra,MnSymbol}
\usepackage{mathrsfs}
\usepackage{color}
\usepackage[cmtip,all]{xy}
\usepackage{eucal}
\usepackage{bbold}
\usepackage{upgreek}
\usepackage{paralist}
\usepackage{tikz-cd}
\usepackage{dsfont}
\usepackage{bbm}
\usepackage[bbgreekl]{mathbbol}

\DeclareMathSymbol\bbDelta \mathord{bbold}{"01}
\DeclareMathSymbol\bDelta \mathord{bbold}{"01}

\newcommand{\bD}{{\mathbb D}}

\renewcommand{\P}{{\mathbb P}}

\newcommand{\mA}{{\mathcal A}}
\newcommand{\mB}{{\mathcal B}}
\newcommand{\mC}{{\mathcal C}}
\newcommand{\mD}{{\mathcal D}}
\newcommand{\mE}{{\mathcal E}}
\newcommand{\mF}{{\mathcal F}}

\newcommand{\mJ}{{\mathcal J}}

\newcommand{\mM}{{\mathcal M}}
\newcommand{\mN}{{\mathcal N}}
\newcommand{\mO}{{\mathcal O}}
\newcommand{\mP}{{\mathcal P}}

\newcommand{\mS}{{\mathcal S}}

\newcommand{\mU}{{\mathcal U}}
\newcommand{\mV}{{\mathcal V}}
\newcommand{\mW}{{\mathcal W}}

\newcommand{\A}{A}
\newcommand{\B}{B}
\newcommand{\C}{C}

\newcommand{\E}{{E}}
\newcommand{\F}{{F}}
\newcommand{\G}{{G}}

\newcommand{\J}{{J}}
\newcommand{\K}{{K}}
\renewcommand{\L}{{\mathrm L}}

\newcommand{\N}{{\mathrm N}}
\renewcommand{\P}{{P}}
\newcommand{\Q}{{Q}}
\newcommand{\R}{{\mathrm R}}
\newcommand{\rS}{{S}}
\newcommand{\T}{{T}}

\newcommand{\X}{X}
\newcommand{\Y}{Y}
\newcommand{\Z}{Z}

\newcommand{\bj}{{j}}

\newcommand{\m}{{m}}
\newcommand{\bk}{{k}}

\newcommand{\n}{{n}}

\newcommand{\op}{\mathrm{op}}

\newcommand{\colim}{\mathrm{colim}}

\newcommand{\rev}{{\mathrm{rev}}}

\newcommand{\Env}{{\mathrm{Env}}}  
\newcommand{\ot}{\otimes}

\newcommand{\co}{\mathrm{co}}

\newcommand{\univ}{\mathrm{univ}}
\newcommand{\strict}{\mathrm{strict}}

\renewcommand{\S}{\mathcal{S}}

\newcommand{\id}{\mathrm{id}}
\newcommand{\Cat}{\mathrm{Cat}}

\newcommand{\Set}{\mathrm{Set}}

\newcommand{\Fun}{\mathrm{Fun}}

\newcommand{\cocart}{{\mathrm{cocart}}}
\newcommand{\coCart}{{\mathrm{coCart}}}
\newcommand{\cart}{{\mathrm{cart}}}
\newcommand{\Cart}{{\mathrm{Cart}}}
\newcommand{\bicart}{{\mathrm{bicart}}}

\newcommand{\lax}{{\mathrm{lax}}}
\newcommand{\oplax}{{\mathrm{oplax}}}

\newcommand{\tu}{{\mathbb 1}}
\newcommand{\coop}{\mathrm{coop}}

\newcommand{\ev}{{\mathrm{ev}}}

\newcommand{\Map}{{\mathrm{Map}}}

\newcommand{\Mor}{{\mathrm{Mor}}}

\newcommand{\PrL}{\mathrm{Pr^L}}
\newcommand{\PrR}{\mathrm{Pr^R}}

\newcommand{\cop}{\mathrm{cop}}

\newcommand{\LMor}{\mathrm{LMor}}

\newcommand{\cube}{{\,\vline\negmedspace\square}}

\newcommand{\scat}{\mathcal{C}\mathit{at}}

\newcommand{\fcat}{\mathfrak{Cat}}

\begin{document}

\title{\textsc{Fibrations in Oriented Category Theory}}

\author{David Gepner and Hadrian Heine}

\maketitle

\begin{abstract}
We study fibrations of higher categories from the perspective of oriented category theory, a framework which accounts for lax phenomena in higher category theory via systematic enrichment in the Gray tensor product.
We give several equivalent characterizations of fibrations
of $(\infty,\infty)$-categories and oriented categories,
and show that categories of fibrations naturally
organize to form oriented categories.
We study the interaction between fibrations and oriented pullbacks
and construct higher-categorical versions of free fibrations and universal fibrations.
The latter give rise to Grothendieck constructions for fibrations of $(\infty,\infty)$-categories and oriented categories.
\end{abstract}

\tableofcontents

\vspace{5mm}

\section{Introduction}

\subsection{Fibrations}
The notion of fibration plays an important role in geometry and topology.
The concept itself goes back at least to the study of covering spaces and the universal cover in the mid-to-late 1800s, notably in Riemann's work on complex functions and Poincar\'e's development of algebraic topology.
By the early-to-mid 1900s, the theory of covering spaces had been subsumed into a general theory of fiber bundles
and characterizations of fibrations in terms of lifting properties were obtained by Hurewicz, later generalized by Serre to a notion better suited to homotopy theory, and finally adapted by Grothendieck to the abstract categorical setting.

In this paper, as well as its predecessors and sequels, we will adapt the viewpoint that $\infty$-category theory\footnote{Henceforth we shall refer to $(\infty,n)$-categories simply as $n$-categories, for all $0\leq n\leq\infty$. In particular, the category of $0$-categories is equivalent to the category of $\infty$-groupoids.} is fundamentally geometric in nature and might usefully be though of an oriented, or directed, generalization of topology, especially algebraic topology and homotopy theory.
From this perspective, it should not surprising that fibrations play a fundamental and important role in $\infty$-category theory.\footnote{Note that in $\infty$-groupoid theory, any map is equivalent to a fibration, but this fails to be the case in the presence of noninvertible morphisms.}
Specifically, we regard $\infty$-categories as oriented spaces \cite{oriented}, which play the same role in oriented category theory as ordinary spaces (also known as anima, homotopy types, or $\infty$-groupoids, or sets if one is working in the strict non-homotopical version of the theory) do in homotopical (or ordinary) category theory.

One of the most crucial features of a fibration in algebraic topology is that pulling back along such a map results in a space which is homotopy equivalent to the homotopy pullback.
In oriented algebraic topology, one obtains an analogous statement: the pullback along a fibration results in an oriented space which is lax homotopy equivalent, in a suitable sense, to the oriented pullback.
The orientations, or the directions of the higher dimensional cells which comprise an oriented space, are also inherent in all the usual operations, including the corrected version of the cartesian product, called the Gray tensor product.
Unlike the cartesian product, it is compatible with the intrinsic notion of dimension in higher category theory and hence results in a geometrically meaningful operation, though at the cost of introducing the antisymmetries inherent in the oriented theory.

\subsection{Oriented category theory}
One of the chief subtleties of higher category theory is that the cartesian product fails to be the correct operation for many purposes.
Indeed, as explained in \cite{GepnerHeine2026}, \cite{oriented}, and \cite{gepner2026homotopy}, the cartesian product is not compatible with the natural notion of dimension in higher category theory, and must be replaced by the Gray tensor product in order to remedy this defect.
This difficulty propagates through the entire theory and necessitates the development of a theory of higher categories in which morphism $\infty$-categories compose via the Gray tensor product operation.
Since the Gray tensor comes in two antisymmetric variants, the lax and oplax versions, which are related by various (anti)involutions, it is perhaps useful to think of the resulting structure as a choice of orientation which dictates the precise nature of the composition of the morphism $\infty$-categories.

Specifically, writing $(\infty\Cat,\boxtimes)$ for the category of $\infty$-categories equipped with the oplax Gray tensor monoidal structure, the category of oriented categories
\[
\Cat\boxtimes =\Cat_{(\infty\Cat,\boxtimes)}
\]
is by definition the category of categories enriched in the monoidal category $(\infty\Cat,\boxtimes)$.
There is also a negatively, or anti-, oriented version
\[
\boxtimes\Cat = _{(\infty\Cat,\boxtimes)}\Cat,
\]
as well as the bioriented version
\[
\boxtimes\Cat\boxtimes = {_{(\infty\Cat,\boxtimes)}\Cat_{(\infty\Cat,\boxtimes)}}.
\]
With these natural enrichments, there are oriented, antioriented, and bioriented versions of $\infty\Cat$, denoted $\infty\fcat$, etc.

We refer to the resulting theory as oriented category theory, since it is both descriptive as well as compatible with certain pre-existing notions, such as the oriented pullback and oriented pushout familiar from low-dimensional higher category theory.
In fact, these oriented construction are special cases of the theory of oriented (co)limits, which we develop in this paper, along with the basic theory of fibrations and an oriented version of the Grothendieck construction.

In geometry, it is typically cartesian fibrations which play the most central role.
But even if one is only interested in cartesian fibrations of $\infty$-categories, one is already led to consider cocartesian fibrations, since cartesian fibrations become cocartesian fibrations after applying the morphism object functor, which can be regarded as an oriented version of the path object, a construction which is already ubiquitous in topology.
On the other hand, the conjugation, or ``co'', duality, which reverses the even dimensional cells, converts a cartesian fibration into an anticartesian fibration, and restricting to morphism $\infty$-categories yields an anticocartesian fibration.

Since cocartesian fibrations are stable under basechange, pulling back the universal cocartesian fibration determines a functor 
\[
\infty{\Cat}_{//^{\oplax}\infty\Cat}\to\co\mathcal{C}\mathit{art}\subset\Fun(\bD^1,\infty\Cat)
\]
from the oplax slice to a certain subcategory of $\Fun(\bD^1,\infty\Cat)$ spanned by the cocartesian fibrations and suitable maps of cocartesian fibrations (these maps require preservation of (co)cartesian morphisms in every positive dimension).
In fact, this functor is itself a morphism of cartesian fibrations over $\infty\Cat$.
However, pullback along the universal cocartesian fibration is actually computed as the oriented fiber over the tautological basepoint $\bD^0\to\infty\Cat$ selecting the object $\bD^0\in\infty\Cat$.

From this description, it is natural to expect the functor to refine to an oriented functor of oriented categories. 
We show that this is indeed the case, although establishing this is not completely formal.
Dually, pulling back the universal cocartesian fibration along functor turns out to refine to an antioriented functor of antioriented categories.
For a fixed base $\infty$-category $X$, we denotes these (anti)oriented categories
\[
\infty\fcat^{\cart}_{/X}\qquad\textrm{and}\qquad\infty\fcat^{\cocart}_{/X},
\]
respectively.
Furthermore, there is a bioriented structure on the category of bicartesian finbrations
\[
\infty\fcat^{\bicart}_{/X}
\]
over $X$.
This is also a fundamental construction, as it provides a natural notion of adjunction of $\infty$-categories.

\subsection{Cartesian fibrations and dual fibrations}
The notion of cartesian fibration goes back at least to Grothendieck and his school in their study of fibered categories and descent.
The generalization to homotopical $n$-categories is due to Lurie for $n=1$ and to Nuiten \cite{nuiten2023straightening} for $n>1$.
Although Nuiten does not directly treat the case in which $n=\infty$, one can simply take the limit to obtain the correct definition.

In the $1$-categorical setting, the difference between cartesian and cocartesian fibrations is purely in the direction of the fiber transport, and applying the opposite involution exchanges these notions.
Hence statements about one formally imply analogous results about the other.
However $n\Cat$ supports a total of $2^n$ involutions, resulting in an infinite number variants of the notion of cartesian fibration when $n=\infty$.
A subtle point is that passing to morphism $\infty$-categories of a cartesian fibration results in a cocartesian fibration, and conversely, so applying even or odd dimensional duality results in the non-dimensionally-alternating notions of anticartesian and anticocartesian fibration.
It is also worth isolating the pure degree one opposite of a cartesian fibration, which we refer to as an opcartesian fibration.

$$\begin{tabular}{| c | c | c | c| c |}
\hline
Cartesian & Anticartesian & Cocartesian & Anticocartesian & Opcartesian\\
\hhline{|=|=|=|=|=|}	
$p:X\to S$ & $p^{\co}:X^{\co}\to S^{\co}$ & $p^{\coop}:X^{\coop}\to S^{\coop}$ & $p^{\op}:X^{\op}\to S^{\op}$ & $p^\circ:X^\circ\to S^\circ$\\
\hline
$p:X\to S$ & $\overline{p}:\overline{X}\to S^{\coop}$ & $p^\vee:X^\vee\to S^{\circ}$ & $ \overline{p}^\vee:\overline{X}^\vee\to S^{\cop} $ & $p^\tau:X^\tau\to S^{\op}$\\
\hline
\end{tabular}$$
Here $\tau:\infty\Cat\to\infty\Cat$ is the operation which reverses cells in dimensions one and two.

However, the cartesian fibrations and their dual variants are only one type of fibration of $\infty$-categories.
Already in the case of $1$-categories, there are robust notions of bicartesian fibrations, locally (co)cartesian fibrations, and exponentiable fibrations.
These also define analogous families in the setting of $\infty$-categories, and determine the following commutative diagram of inclusions of (not generally full) subcategories of $\Fun(\bD^1,\infty\Cat)$.

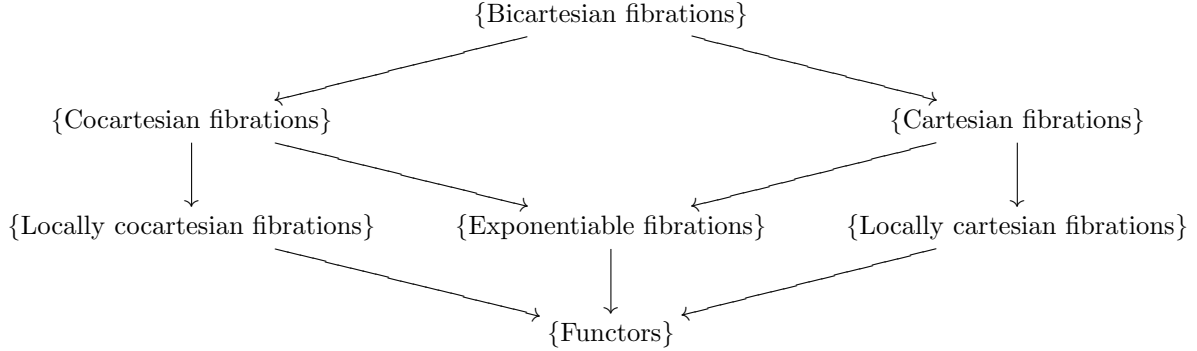
$$
\begin{xy}
\xymatrix{
 &   \{ {\mathrm{Bicartesian} \ \mathrm{fibrations}} \} \ar[ld]\ar[rd] &
\\
\{ \mathrm{Cocartesian} \ \mathrm{fibrations} \}  \ar[rd] \ar[d] & 
& \{ \mathrm{Cartesian} \ \mathrm{fibrations}  \} \ar[ld] \ar[d]\\
\{ {\mathrm{Locally} \ \mathrm{cocartesian} \ \mathrm{fibrations} \} } \ar[rd] & \{ \mathrm{Exponentiable} \ \mathrm{fibrations} \} \ar[d] & {\{\mathrm{Locally} \ \mathrm{cartesian} \ \mathrm{fibrations} \} }  \ar[ld] \\
& \{ \mathrm{Functors}\} &
}
\end{xy} $$

There are a number of ways to characterize cartesian fibrations $p:Y\to X$.
The first and most historically motivated would be to require the existence of cartesian lifts for all $n$-morphisms $x:\bD^n\to X$, where $n$ is an arbitrary positive integer.
In the 1-categorical case, as opposed to the higher category case, it is enough to consider $n=1$, but this simplification obscures much of the subtlety of the theory of fibrations of $\infty$-categories.

Let $1 \leq \n \leq \infty$ and $\phi: \mC \to \mD$ a functor of $\infty$-categories.
A 1-morphism $f: X \to Y$ in $\mC$ is $\phi$-cocartesian if for every $Z\in \mC$ the induced functor
$$ \Mor_\mC(Y,Z) \to \Mor_\mD(\phi(Y), \phi(Z))\times_{\Mor_\mD(\phi(X), \phi(Z)} \Mor_\mC(X,Z) $$
is an equivalence.
Inductively, an $n$-morphism $\alpha:\bD^n\to\mC$ is $\phi$-cocartesian if
for every pair of morphisms $(X \to \alpha(0), \alpha(1) \to Y)$ in $\mC$ the composite $n-1$-morphism
\[
\bD^{n-1}\to\Mor_\mC(\alpha(0),\alpha(1)) \to \Mor_\mC(X,Y)
\]
is $\phi_{X,Y}$-cocartesian.
With this in place, we can define a cartesian fibration as a functor
$p:Y\to X$ such that
for each positive integer $n$, $n$-cell $x:\bD^n\to X$, and lift $t(y):\bD^{n-1}\to Y$ of the target $n-1$-cell $\bD^{n-1}\to\bD^n$, there exists a $p$-cartesian lift $y:\bD^n\to Y$ of $x$ with target $t(y)$.

There are also more inductive formulations of this notion which rely on checking that the induced functors on the morphism $\infty$-categories which appear in the functor $p:Y\to X$ are {\em cocartesian} fibrations.
This approach is especially useful if $X$ and $Y$ are $n$-categories for some $n<\infty$, since passing to morphism $\infty$-categories reduces categorical dimension and any functor of $0$-categories is a (co)cartesian fibration.

Observe the alternation between cartesian and cocartesian which occurs by passing to morphism $\infty$-categories.
It is for this reason that it is sometimes simpler to work with the conjugate or opposite of a cartesian fibration, which results in an anticartesian or anticocartesian fibration, respectively.

For other purposes, it is useful to know that one can lift along more complicated geometric shapes than just the disks.
Specifically, we formulate lifting for (anti)oriented simplices, which has the advantage of being a forming a dense subcategory of $\infty\Cat$.
Namely, a functor $p:Y\to X$ is a cartesian fibration if, for all positive integers $n$, every commutative square
\begin{equation}
\begin{xy}
\xymatrix{
\bDelta^0 \star \bDelta^{n-2} \ar[d]_{\{0\}\star \bDelta^{n-2}} \ar[r]
& Y^{\op} \ar[d]^\phi
\\ 
\bDelta^1 \star \bDelta^{n-2} \ar[r] & X^{\op}
}
\end{xy}\end{equation}
admits a filler by a $p$-cocartesian oriented $n$-simplex, and every commutative square
\begin{equation}
\begin{xy}
\xymatrix{(\bDelta^{n-2})^\co \bar{\star} (\bDelta^0)^\co \ar[d]_{(\bDelta^{n-2})^\co \bar{\star} \{0\}} \ar[r]
& Y^{\op} \ar[d]^\phi
\\ 
(\bDelta^{n-2})^\co \bar{\star} (\bDelta^1)^\co \ar[r] & X^{\op}
}
\end{xy}\end{equation}
admits a filler by a $p$-cocartesian antioriented $n$-simplex.
The $(-)^\op$ which appears in these conditions does not convert the cartesian fibration $p$ into a cocartesian fibration, but rather into an anticocartesian fibration, and the result is easier to state for anticocartesian fibrations since it does not require further dualization, though one could easily apply whichever involutions are necessary to the (anti)orientals and to obtain analogous statements.

We offer another yet another characterization of cartesian fibrations, as follows, making full use the oplax functor categories in order to reduce to having to test only on the source inclusion $\{0\}\subset\bD^1$.
For a functor of $\infty$-categories $\phi: Y\to X$, one can isolate
a subcategory $$ \Fun^{\lax, \cocart}(\bD^1,Y) \subset \Fun^{\lax}(\bD^1,Y)$$ whose $n$-morphisms for $n \geq 0$
correspond to functors $\bD^1 \boxtimes \bD^n \to Y$
such that the composition $$ \bD^{n+1} \to \bD^1 \boxtimes \bD^n \to Y $$
with the functor assigning the unique non-invertible $n+1$-morphism of $\bD^1 \boxtimes \bD^n$
is a $\phi$-cocartesian $n+1$-morphism.
Then $\phi:Y\to X$ is a cartesian fibration if the induced functor
$$
\Fun^{\oplax, \cart}(\bD^1,Y) \to \Fun^\oplax(\bD^1,X) \times_{\Fun^\oplax(\{0\},X)} Y
$$
is an equivalence.

\subsection{Variants}
There are a number of other types of fibrations which are closely related to the notion of cartesian fibration.
To motivate some of these notions, we show that a functor of $\infty$-categories $p:Y\to X$ is a cartesian fibration if any only if its basechange $Y\times_{X}\bDelta^n\to\bDelta^n$ along any map of the form $\bDelta^n\to X$ is a cartesian fibration.
The same holds for various other dense subcategories of $\infty\Cat$; that is, we could also test on cubes or objects of $\Theta$.
Density, however, is crucial here, so it does not apply to the disks $\bD^n$.
Rather, a slight generalization of the notion of cartesian fibration is that of locally cartesian fibration.
While they have an intrinsic definition, we show they are precisely those functors $p:Y\to X$ whose restriction to each cell $\bD^n\to X$ of $X$ is a cartesian fibration $Y\times_X\bD^n\to\bD^n$.

Another basic type of fibration which arises in practice are the bicartesian fibrations, which are functors of $\infty$-categories $p:Y\to X$ which are simultaneously cartesian and cocartesian fibrations.
As might be expected, bicartesian fibrations are related to adjunctions in the $\infty$-category of $\infty$-categories.
Since we will need a bit more technology to assert anything nontrivial about bicartesian fibrations, we treat them instead in the forthcoming paper \cite{OrientedGrothendieck2026}.

In the setting of enriched category theory, there is also a notion of fibration.
If $\mV$ is a monoidal category, we say that an $\mV$-enriched functor $\phi: \mC \to \mD$ is a cartesian fibration if for every $Y \in \mC$ and morphism $\alpha: Z \to \phi(Y)$ in $\mD$, there is a $\phi$-cartesian morphism $X \to Y$ in $\mC$ lying over $\alpha$.
Here the notion of $\phi$-cartesian is the expected generalization from that of the $1$-categorical context.

This formal notion of enriched fibration can be used to define the notion of cartesian fibration of oriented categories, and dually a notion of cocartesian fibration of oriented categories.
However, these are not simply fibrations of categories enriched in $\infty\Cat$ under the Gray tensor monoidal structure; rather, it makes explicit reference to the notion of (co)cartesian fibration of $\infty$-categories.
Namely, we say that an oriented functor $\phi: \mC \to \mD$ is a cocartesian fibration if it is a $(\infty\Cat, \boxtimes^\rev)$-enriched cartesian fibration and, for every $X,Y \in \mC$
the induced functor
$$\Mor_\mC(X,Y) \to \Mor_\mD(\phi(X),\phi(Y)) $$ is a cartesian fibration
and for every $X \in \mC$ and morphism $Y \to Z $ in $\mC$ the induced commutative square
$$\begin{xy}
\xymatrix{
\Mor_\mC(X,Y) \ar[d]^{} \ar[r]
& \Mor_\mC(X,Z) \ar[d]
\\ 
\Mor_\mD(\phi(X),\phi(Y)) \ar[r] & \Mor_\mD(\phi(X),\phi(Z))}
\end{xy}$$
is a map of cartesian fibrations.

An example of a fibration of oriented categories is the tautological cocartesian fibration
\[
\infty\fcat_{\ast//}\to\infty\fcat.
\]
This might also be referred to as the universal cocartesian fibration, especially as its restriction along the inclusion $\infty\scat\to$$\,\infty\fcat$ is the universal cocartesian fibration $\infty\fcat_{\ast//}\to\infty\fcat$ of $\infty$-categories.
More generally, given any object $S$ of an $\infty$-category or oriented category $\mC$, we can construct a slice fibration
\[
\mC_{//S}\to\mC
\]
given by forgetting the structure maps to $S$.
This is a cartesian fibration in which the cartesian lifts of morphisms in $\mC$ with specified lift of the target cells are essentially obtained by composition with the structure maps to $S$.

\subsection{Bifibrations and free fibrations}
A bifibration is a functor of $\infty$-categories $p:X\to S\times T$ such that the projection $X\to S$ is cartesian and $X\to T$ is cocartesian.
The canonical example of a bifibration is the arrow bifibration
\[
p:\Fun^{\oplax}(\bD^1,X)\to\Fun^{\oplax}(\partial\bD^1,X)\simeq X\times X,
\]
which is the source fibration in first factor and the target fibration in the second factor.
We show that the source functor is cartesian via precomposition while the target functor might be expected to be cocartesian via postcomposition.
This is evidently the case $1$-categorically, but it is significantly more technical to establish this higher categorically.

Bifibrations are often called two-sided fibrations in the literature, though we warn the reader that various authors reserve the term bifibations for a two-sided fibration with groupoidal fibers.
Since this is not a reasonable condition to impose when studying fibrations of higher categories, we use will use the terms bifibrations and two-sided fibrations interchangeably.
Applying even or odd dimensional duality to a bifibration results in an antibifibration and conversely, as evidenced in the following table:

$$\begin{tabular}{| c | c | c | c| c |}
\hline
Bifibration & Antibifibration & Bifibration & Antibifibration\\
\hhline{|=|=|=|=|}	
$p:X\to S\times T$ & $p^{\co}:X^{\co}\to S^{\co}\times T^{\co}$ & $p^{\coop}:X^{\coop}\to T^{\coop}\times S^{\coop}$ & $p^{\op}:X^{\op}\to T^{\op}\times S^{\op}$ \\
\hline
$p:X\to S\times T$ & $\overline{p}:\overline{X}\to S^{\coop}\times T^{\coop}$ & $p^\vee:X^\vee\to T^{\circ}\times S^{\circ}$ & $ \overline{p}^\vee:\overline{X}^\vee\to T^{\cop}\times S^{\cop} $ \\
\hline
\end{tabular}$$

The oriented pullback construction, applied to the cospans consisting of identity functors $X\to X\leftarrow X$ and the unique functors $X\to\bD^0\leftarrow X$, results in arrow bifibration
\[
X\underset{X}{\vec{\times}}X\simeq\Fun^{\oplax}(\bD^1,X)\to X\times X\simeq X\underset{\bD^0}{\vec{\times}}X.
\]
More generally, we show that oriented pullbacks are bifibrations in general; that is, the oriented pullback of the span $X\to Z\leftarrow Y$ results in a bifibration
\[
X\underset{Z}{\vec{\times}}Y\to X\times Y.
\]

Given an arbitrary functor of $\infty$-categories $p:Y\to X$, we may form the free cartesian or cocartesian fibration, also called the enveloping fibration.
By definition, the enveloping (co)cartesian fibration construction forms a left adjoint to the forgetful functor from the $\infty$-category of (co)cartesian fibrations over $S$ to the $\infty$-category of $\infty$-categories over $S$.
We show that these left adjoints, applied to a functor $X\to S$, are given by the formulae
\[
X\underset{S}{\vec{\times}}S\qquad\textrm{and}\qquad S\underset{S}{\vec{\times}}X,
\]
respectively.
In particular, they determine bifibrations
\[
X\underset{S}{\vec{\times}}S\to X\times S\qquad\textrm{and}\qquad S\underset{S}{\vec{\times}}X\to S\times X.
\]

\subsection{Universal fibrations}
In a follow-up paper \cite{OrientedGrothendieck2026} we will show a third type of equivalent formulation of cartesian fibration.
These formulations are much more technically involved as they require a delicate analysis of the Grothendieck construction in the higher categorical context.
We show that a functor of $\infty$-categories
$p:Y\to X$ is a cartesian fibration if and only if $p^{\coop}$ is the lax colimit of functor $f:X^{\coop}\to\infty\scat$, or equivalently, if and only if
$p^{\coop}:Y^{\coop}\to X^{\coop}$ is the oriented left fiber of a functor $X^{\coop}\to\infty\scat$ at the terminal object inclusion $*\to\infty\scat$.

Pulling back the universal cocartesian fibration results in a functor
\[
\int_S:\Fun(S,\infty\scat)\to\infty\scat_{/S}.
\]
usually referred to as the Grothendieck construction.
The primary significance of the Grothendieck construction, which we study in the higher and oriented categorical context in a follow-up paper, is that its image is precisely the subcategory
\[
\infty\scat^{\cocart}_{/S}\subset\infty\scat_{/S}
\]
spanned by the cocartesian fibrations and (higher) morphisms of cocartesian fibrations.

Consequently, we may unstraighten cartesian fibrations $p:X\to S$ to presheaves $f:S^{\op}\to\infty\Cat$, providing a means of overcoming one of the major technical challenges involved in higher category theory.
This is tremendously useful in practice in homotopy-coherent mathematics, fibrations can usually be specified even when there is only naturally a contractible space of choices for something, whereas a functor requires specific choices for everything which are compatable with the infinity hierarchy of higher compositions.

There is a refinement of the universal cocartesian fibration of $\infty$-categories to a universal cocartesian fibration of oriented categories.
It is defined as the target projection
\[
\infty\fcat_{\ast//}\to\infty\fcat
\]
and can be shown to be a cocartesian fibration of oriented categories which restriction along the standard inclusion $\infty\scat\to$ $\infty\fcat$ recovers the target fibration $\infty\scat_{\ast//}\to\infty\scat$.

\subsection{Main results}
The theory of fibrations in the setting of $\infty$-categories is considerable more complex than in the setting of $1$-categories.
Thus it is should not be surprising that there are a significant number of equivalent characterizations of cartesian fibrations in terms of lifting properties.

\begin{theorem}[{\cref{T1}}]
For a morphism of $\infty$-categories $p:Y\to X$, the following conditions are equivalent:
\begin{enumerate}[\normalfont(1)]\setlength{\itemsep}{-2pt}
\item[\em{(1)}]
For each positive integer $n$, $n$-cell $x:\bD^n\to X$, and lift $t(y):\bD^{n-1}\to Y$ of the target $n-1$-cell $\bD^{n-1}\to\bD^n$, there exists a $p$-cartesian lift $y:\bD^n\to Y$ of $x$ with target $t(y)$.
\item[\em{(2)}]
For all positive integers $n$, every commutative square
\begin{equation}
\begin{xy}
\xymatrix{
\bDelta^0 \star \bDelta^{n-2} \ar[d]_{\{0\}\star \bDelta^{n-2}} \ar[r]
& Y^{\op} \ar[d]^\phi
\\ 
\bDelta^1 \star \bDelta^{n-2} \ar[r] & X^{\op}
}
\end{xy}\end{equation}
admits a filler by a $p$-cocartesian oriented $n$-simplex, and every commutative square
\begin{equation}\label{filler1}
\begin{xy}
\xymatrix{(\bDelta^{n-2})^\co \bar{\star} (\bDelta^0)^\co \ar[d]_{(\bDelta^{n-2})^\co \bar{\star} \{0\}} \ar[r]
& Y^{\op} \ar[d]^\phi
\\ 
(\bDelta^{n-2})^\co \bar{\star} (\bDelta^1)^\co \ar[r] & X^{\op}
}
\end{xy}\end{equation}
admits a filler by a $p$-cocartesian antioriented $n$-simplex.
\item[\em{(3)}]
The functor $p$ admits cartesian lifts of 1-dimensional arrows with specified target and, for every ordered pair of objects $s, t \in Y$, $p_{s,t}:\Mor_Y(s,t)\to\Mor_X(p(s),p(t))$ is a cocartesian fibration, such that for any morphisms $s'\to s$ and $t\to t'$ in $X$, the induced map $p_{s,t}\to p_{s',t'}$ is a morphism of cocartesian fibrations.
\end{enumerate}
\end{theorem}
A cartesian fibration of $\infty$-categories is any functor of $\infty$-categories $p:Y\to X$ which satisfies one of the equivalent conditions above.
We typically regard condition (1) as the actual definition, since it is probably the most checkable in a generic situation.
However, especially when working with $n$-categories, it is often the case that condition (3) is the easiest to work with, at least if one has good control over the morphism $\infty$-categories in $p:Y\to X$.
This is the approach taken by Nuiten in \cite{nuiten2023straightening}, following Street's work in the bicategorical case \cite{street1980fibrations}.
However, Nuiten's inductive approach has the drawback that it does not explicitly identify the cartesian lifts of higher morphisms, as in the standard $1$-categorical definition of cartesian fibration.
In many situations, especially when the $\infty$-categories $X$ and $Y$ are not defined inductively, it is easier to show that $p:Y\to X$ is a cartesian fibration by explicitely constructing the cartesian lifts.

In a sequel paper \cite{OrientedGrothendieck2026}, we will provide a number of additional equivalent characterizations of cartesian fibrations, related to the fact that they are classified by the universal cartesian fibration.
This required detailed knowledge of the Grothendieck construction, in particular determining its essential image.
In particular, this results in higher categorical analogues of Lurie's straightening and unstraightening equivalence, and is not a formal consequence of the 1-categorical theory, but instead involves a delicate analysis of cartesian fibrations over some dense subcategory of $\infty\Cat$ (we work with $\Theta$) and descent arguments.

\begin{theorem}(\cref{orienfib})
For any $\infty$-category $X$ the categories of cartesian and cocartesian fibrations over $X$ refine to antioriented categories
\[
\infty\fcat^{\cart}_{/X}\qquad\textrm{and}\qquad\infty\fcat^{\cocart}_{/X},
\]
respectively,
and the categories of anticartesian and anticocartesian fibrations over $X$ refine to oriented categories
\[
\infty\fcat^{\overline{\cart}}_{/X}\qquad\textrm{and}\qquad\infty\fcat^{\overline{\cocart}}_{/X},
\]
respectively.

% Furthermore, the category of bicartesian fibrations over $X$ refines to a bioriented category
% \[
% \infty\fcat^{\bicart}_{/X}.
% \]
\end{theorem}

\begin{theorem}[{\cref{leftadjointtheorem}}]
For any $\infty$-category $X$, the inclusion of the subcategory
\[
\infty\fcat^{\cart}_{/X}\subset\infty\fcat_{/X}
\]
admits a left adjoint (compatibly with the oriented structure).
We refer to the value of the left adjoint
on a functor $f : X \to Y$ as the free, or enveloping, cartesian fibration on $f$, and it is computed as the oriented pullback $$ X \vec{\times}_Y Y \to X $$ of $f$ along the identity of $Y.$
\end{theorem}

\begin{theorem}[{\cref{precocart}}]
    A functor of $\infty$-categories $p:Y\to X$ is a cartesian fibration if and only if the induced functor
    $$
    \Fun^{\oplax, \cart}(\bD^1,Y) \to \Fun^\oplax(\bD^1,X) \times_{\Fun^\oplax(\{0\},X)} Y
    $$
    is an equivalence.
\end{theorem}

\begin{theorem}[{\cref{targetfi}}]
    Oriented pullbacks are bifibrations.
    More precisely, the projections induced from the oriented pullback of the diagram $X\to Z\leftarrow Y$ result in a bifibration
\[
X\underset{Z}{\vec{\times}}Y\to X\times Y.
\]
\end{theorem}

Evidently, cartesian fibrations are stable under pullback.
It turns out that they are also stable under oriented pullback, although this statement is not entirely formal.
\begin{theorem}[{\cref{laxfib}}]
For every pair of maps of cocartesiam fibrations
\begin{equation*}
\xymatrix{
\mA \ar[r]^F \ar[d]^\rho & \mC \ar[d]^\phi \\
\mE \ar[r] & \mD,
}\qquad
\xymatrix{
\mB \ar[r]^G \ar[d]^\kappa & \mC \ar[d]^\phi \\
\mF \ar[r] & \mD,
}
\end{equation*}
with common target, the induced functor 
\begin{equation*}
\mA {\bar{\vec{\times}}}_{\mC} \mB \to \mE {\bar{\vec{\times}}}_{\mD} \mF
\end{equation*}
is a cocartesian fibration.
A 1-morphism in $\mA {\bar{\vec{\times}}}_{\mC} \mB $ is cocartesian over $\mE {\bar{\vec{\times}}}_{\mD} \mF $ if and only if its image in $\mA$ is cocartesian over $\mE$, its image in $\mB$ is cocartesian over $\mF$ and its image in $\Fun^\lax(\bD^1,\mC)$
corresponds to an antioriented square in $\mC$ whose non-invertible 2-morphism is cartesian over $\mD.$

Moreover, for every (anti)oriented commutative squares of $n+1$-cocartesian fibrations 

\qquad\qquad\qquad\begin{tikzcd}[row sep=scriptsize, column sep=scriptsize]
& \mA \arrow{dl}{\alpha} \arrow{rr}{F} \arrow[dd] & & \mC \arrow{dl}{\gamma} \arrow[dd] \\ \mA' \ar[double]{rrru}{} \arrow[rr, crossing over] \arrow[dd] & & \mC' \\
& \mE \arrow[dl] \arrow[rr] & & \mD \arrow[dl] \\
\mE' \ar[double]{rrru}{} \arrow[rr] & & \mD' \arrow[from=uu, crossing over]\end{tikzcd}\qquad\qquad
\begin{tikzcd}[row sep=scriptsize, column sep=scriptsize]
& \mB \arrow{dl}{\beta} \arrow{rr}{G} \arrow[dd] & & \mC \ar[double]{llld}{} \arrow{dl}{\gamma} \arrow[dd] \\ \mB' \arrow[rr, crossing over] \arrow[dd] & & \mC' \\
& \mF \arrow[dl] \arrow[rr] & & \mD \ar[double]{llld}{} \arrow[dl] \\
\mF' \arrow[rr] & & \mD' \arrow[from=uu, crossing over]
\end{tikzcd}\\
the following induced commutative square is a map of $n$-cocartesian fibrations:
\begin{equation*}
\xymatrix{
\mA {\bar{\vec{\times}}}_{\mC} \mB \ar[r] \ar[d] & \mA' {\bar{\vec{\times}}}_{\mC'} \mB' \ar[d] \\
\mE {\bar{\vec{\times}}}_{\mD} \mF \ar[r] & \mE' {\bar{\vec{\times}}}_{\mD'} \mF'.
}\end{equation*}
\end{theorem}

\begin{theorem}[{\cref{T7}}]
    The Grothendieck construction for oriented cocartesian fibrations is the unique map 
$$\int: {\boxtimes\widehat{\Cat}}_{//^\oplax\infty\fcat} \to {\boxtimes\co\widehat{\Cart}}$$
of cartesian fibrations over $\boxtimes\widehat{\Cat}$ sending $\infty\fcat$ to the universal oriented cocartesian fibration $$ \infty\fcat_{*//^\oplax} \to \infty\fcat.$$
The map $\int: {\boxtimes\widehat{\Cat}}_{//^\oplax\infty\fcat} \to {\boxtimes\co\widehat{\Cart}}$ of cartesian fibrations over
$\boxtimes\widehat{\Cat}$
restricts to a map 
\begin{equation*}
\int: {\boxtimes\Cat}_{//^\oplax\infty\fcat} \to {\boxtimes\coCart}
\end{equation*}
of cartesian fibrations over $\boxtimes\Cat$ which restricts on the fiber over an antoriented category $ \mC $
the functor $$ \boxtimes\Fun(\mC, \infty\fcat) \to {\boxtimes\coCart}_{/\mC}.$$

\end{theorem}

\subsection{Relation to other work}
The theory of fibered $1$-categories emerged from the work of Grothendieck and his school on descent in algebraic geometry, notably  \cite{grothendieck1959technique} and \cite{grothendieck2002rev}.
This was later generalized by Lurie in \cite{lurie.HTT} to the homotopical $1$-categorical setting, who refers to (co)fibered categories as (co)cartesian fibrations.
There is also work of Gepner--Haugseng--Nikolaus on free (co)cartesian fibrations \cite{articles} and applications to the theory of lax (co)limits.

There has also been substantial development of this subject in the setting of bicategories, including Gray's very early paper \cite{gray1966fibred} and Street's study \cite{street1980fibrations}.
In addition to Lurie's work \cite{lurie2009infinity}, fibrations of  homotopical bicategories have also been treated by Gagna--Harpaz--Lanari \cite{gagna2020fibrations}, Abell\'an--Stern \cite{abellan20262}, Abell\'an--Gagna--Haugseng \cite{abellán2024straighteninglaxtransformationsadjunctions}, Abell\'an--Haugseng--Martini \cite{abellan2026free}, and others.

Higher categorically, there is work of Nuiten on cartesian fibrations of $n$-categories \cite{nuiten2023straightening} as well as work of Moser--Rasekh--Rovelli on lax (co)limits and the Grothendieck construction \cite{moser2023inftyncategoricalstraighteningunstraighteningconstruction}.
Rasekh also studies fibrations in the setting of Segal spaces in \cite{rasekh2021cartesian} and \cite{rasekh2022cartesian}, and Abell\'an has work on local fibrations of bicategories \cite{abellan2026local}.
Finally, for $n=\infty$, the basic theory of cartesian fibrations has also been considered by Loubaton \cite{loubaton2024categorical}.

A detailed study of exponentiable fibrations appears in the Ayala--Francis paper \cite{MR4074276}.
Even more generally, cocartesian fibrations over the simplex category provide a way of discussing lax normal functors in the context of $\infty$-categories.
This is related to recent work of Blom \cite{blom2024straightening} and Heine \cite{heine2026local}.

\subsection{Notation and terminology}

We fix a hierarchy of set-theoretic universes whose objects we call small, large, very large, etc.
We call a space (equivalently, $\infty$-groupoid) $X$ small, large, etc. if for any choice of basepoint and natural number $n$ its homotopy sets $\pi_n X$ are small, large, etc.
We call an $\infty$-category small, large, etc. if its maximal subspace and all its mapping spaces are small, large, etc.

We refer to (not necessarily univalent) weak $(\infty,\n)$-categories for $0 \leq \n \leq \infty$ simply as $\n$-categories, and we refer to (not necessarily univalent) weak $(\n,\n)$-categories as $(\n,\n)$-categories.
In particular, we refer to (not necessarily univalent) $(\infty,1)$-categories as 1-categories, or simply categories.
We will sometimes want to work strictly, which can be viewed as a basechange along the colimit-preserving symmetric monoidal functor $\mS\to\Set$.
In this case, we will refer to strict $(\n,\n)$-categories simply as strict $\n$-categories.\footnote{Note that an $(n,n)$-category need not be a strict $(n,n)$-category if $n>2$.
For instance, the fundamental $\infty$-groupoid of the $2$-sphere $S^2\simeq \mathrm{B}^2\Omega^2 S^2$ is an $(\infty,0)$-category which is not strict, nor are its $n$-truncations for any $n>2$.}

\begin{notation}
We will make use of the following notation and terminology when discussing categories, in the sense of categories enriched in the monoidal category of $\infty$-groupoids under the cartesian product.
\begin{enumerate}[\normalfont(1)]\setlength{\itemsep}{-2pt}
\item We write $\mS$ for the category of spaces, by which we mean small $\infty$-groupoids, homotopy types, or anima, and $\Set$ for the category of small sets.
\item We write $\infty\Cat$ for the large category of small $\infty$-categories.

\item We write $\Delta$ for (a skeleton of) the category of finite, non-empty, partially ordered sets and order preserving maps, whose objects we denote by $[\n] = \{0 < ... < \n\}$ for $\n \geq 0$.\footnote{This should not be confused with the category $\bDelta$ of oriented simplices.}
\item We write $\Map_{\mC}(A,B)$ for the space of maps (equivalently, $1$-morphisms) from $A$ to $B$ in $\mC$, for any category $\mC$ containing an ordered pair of objects $(A,B)\in\mC$.

\item We write $\ast$ for the final object and
$\emptyset$ for the initial object in any category.

\item We call a map of spaces $X \to Y$ an embedding if its fibers are empty or contractible, or likewise if the induced map $X \to X \times_Y X$ is an equivalence, or likewise if it is of the form $X \to X \coprod X'$ for spaces $X,X'.$ 

\item We call a morphism $X \to Y$ in any category $\mC$ a monomorphism if for every $Z \in \mC$ the induced map of spaces $\Map_\mC(Z,X) \to \Map_\mC(Z,Y) $ is an embedding. So embeddings of spaces are precisely monomorphisms in $\mS.$ 

\item We call a fully faithful functor $\mC \to \mD$ an embedding generalizing the notation for spaces.

\item We call a functor an inclusion if it is a monomorphism in $\Cat$. A functor $\mC \to \mD$ is an inclusion if and only if it induces an embedding on maximal subspaces and on all mapping spaces. 

\item For a diagram $X\to Z\leftarrow Y$ in a category $\mC$ we write $X\underset{Z}{\prod} Y$ or $X\underset{Z}{\times} Y$ for the pullback, and given a diagram $X\leftarrow W\to Y$ in a category $\mC$, we write $X\underset{W}{\coprod} Y$ or $X\underset{Z}{+} Y$ for the pushout.
\item If $\mC$ and $\mD$ are categories and $\mC\to\mD$ is a left adjoint functor with right adjoint $\mD\to\mC$, we often write $\mC\rightleftarrows\mD$ for this adjunction, where the left adjoint is understood to be the functor going from left to right.

\item 
We write $\mC_*$ or $\mC_{\ast/}$ for the category of pointed objects in a category $\mC$, i.e. the full subcategory of $\Fun([1],\mC)$ of arrows in $\mC$ whose source is a final object.

\item We write $\PrL$ and $\PrR$ for the subcategories of $\widehat{\Cat}$ spanned by the presentable categories and the left and right adjoint functors, respectively.
There is a canonical equivalence $\PrL \simeq (\PrR)^\op$
sending left to right adjoints.
\item We write $\otimes$ for the symmetric monoidal structure on $\PrL$ and $\PrR$.
More precisely, by \cite{lurie.higheralgebra}, $\PrL$ carries a closed symmetric monoidal structure such that the subcategory inclusion $\PrL \subset \widehat{\Cat}$ is a lax symmetric monoidal with respect to the the cartesian structure on $\widehat{\Cat}$.
\item We write $\infty\scat$ for the large $\infty$-category of small $\infty$-categories.\footnote{The morphism $\infty$-categories are formed via enrichment in the cartesian monoidal structure.}
\item We write $\Fun(\mD,\mC)$ for the $\infty$-category of functors from an $\infty$-category $\mD$ to an $\infty$-category $\mC$, the value at $\mC$ of the right adjoint to the functor $(-)\times\mC:\infty\Cat\to\infty\Cat$ for the cartesian product.

\item We write $\bD^1$ for the walking arrow, the category with two objects and a unique non-identity arrow.
\item We write $\partial\bD^1$ and $S^0$ for the maximal subspace in $\bD^1$, the set with two elements.

\item We write $\iota_{n}\mC$ for the $n$-category arising from an $\infty$-category $\mC$ by discarding all noninvertible morphisms above dimension $n.$
\item We write $\tau_{n}\mC$ for the $n$-category arising from an $\infty$-category $\mC$ by inverting all morphisms above dimension $n.$
\end{enumerate}
\end{notation} 

\subsection*{Acknowledgements}
We thank Tim Campion, Rune Haugseng, Felix Loubaton, Naruki Masuda, Thomas Nikolaus, Markus Spitzweck, and Germ\'an Stefanich for interesting conversations related to the subject of this paper.
We thank the MPIM for their hospitality while much of this work was carried out.

\vspace{.25cm}

\section{\mbox{Higher categories}}

\subsection{Enriched categories}

We first recall the notion of homotopy coherent enrichment, as defined and studied in, for instance, \cite{MR3345192}, \cite{HEINE2023108941}, \cite{heine2024higher}, \cite{heine2025equivalence}, \cite{HINICH2020107129}.

For every presentably monoidal category $\mV$ there is a presentable 2-category $${\mV\mathrm{-}\Cat}$$ of (not necessarily univalent) $\mV$-enriched categories and $\mV$-enriched functors and a forgetful functor $$ \iota:{\mV\mathrm{-}\Cat} \to \Cat$$ to the presentable 2-category $\Cat$ of (not necessarily univalent) categories,
which is an equivalence for $\mV= \mS$ the category of homotopy types \cite[Corollary 3.23.]{heine2024bienriched}. 

Let $${\mV\mathrm{-}\Cat}^\univ \subset {\mV\mathrm{-}\Cat}$$ be the reflexive full subcategory of univalent $\mV$-enriched categories.

\begin{notation}Let $\mV$ be a presentably monoidal category.
A $\mV$-enriched category $\mC$ has an underlying category $\iota(\mC)$, for every objects $X,Y \in \iota(\mC)$ a morphism object 
$$\Mor_\mC(X, Y) \in \mV$$
and for every objects $X,Y,Z \in \iota(\mC)$ a composition morphism in $\mV:$
$$\Mor_\mC(Y, Z) \ot \Mor_\mC(X, Y) \to \Mor_\mC(X, Z).$$

We write $\X \in \mC$ for $\X \in \iota(\mC)$ and usually notationally identify $\mC$ with $\iota(\mC).$

\end{notation}

The following is \cite[Example 2.134.]{heine2024bienriched}:

\begin{example}\label{linenr}
Let $\mV$ be a presentably monoidal category. Every presentably left $\mV$-tensored category is a $\mV$-enriched category.
Every $\mV$-linear functor between presentably left $\mV$-tensored categories is a $\mV$-enriched functor.
In particular, $\mV$, which is presentably left tensored over itself, is a $\mV$-enriched category.

\end{example}

\begin{notation}Let $\mV$ be a presentably monoidal category.
Let $B\tu_\mV \subset \mV$ be the full $\mV$-enriched subcategory spanned by the tensor unit. Then $\Mor_{B\tu_\mV}(\tu_\mV, \tu_\mV) \simeq \tu_\mV.$
    
\end{notation}

For every enriched category there is an opposite one:

\begin{notation}
There is an involution $$(-)^\circ: {\mV\mathrm{-}\Cat} \simeq {\mV^\rev\mathrm{-}\Cat}$$ forming the opposite enriched category.
For every $\mC \in {\mV\mathrm{-}\Cat}$ and $X,Y \in \mC$
there are canonical equivalences
$ \iota(\mC^\circ) \simeq \iota(\mC)^\op$
and $$ \Mor_{\mC^\circ}(X,Y) \simeq \Mor_{\mC}(Y,X).$$
\end{notation}

The following is \cite[Proposition 3.72.]{heine2024bienriched}:

\begin{proposition}

Let $\mV,\mW$ be presentably monoidal categories and 
$\phi:\mV \to \mW$ a lax monoidal functor.

\begin{enumerate}[\normalfont(1)]\setlength{\itemsep}{-2pt}
\item There is an induced functor
$\phi_!: \mV \mathrm{-}\Cat \to \mW \mathrm{-}\Cat$
that transfers the enrichment, which descends to a functor
$$\mV \mathrm{-}\Cat^\univ \to \mW \mathrm{-}\Cat^\univ.$$

\item For every $\mV$-enriched category $\mC$ and $X,Y \in \iota(\mC)$
there is a canonical equivalence
$$ \Mor_{\phi_!(\mC)}(X,Y) \simeq \phi(\Mor_\mC(X,Y)).$$

\item If $\phi$ is monoidal and admits a right adjoint $\gamma$,
there is an induced adjunction
$\phi_!: \mV \mathrm{-}\Cat \to \mW \mathrm{-}\Cat: \phi^*:=\gamma_!$
that descends to an adjunction 
$\mV \mathrm{-}\Cat^\univ \to \mW \mathrm{-}\Cat^\univ.$

\end{enumerate}

\end{proposition}

The following is \cite[Proposition 2.1.5.]{GepnerHeine2026}:

\begin{proposition}\label{monochar}
Let $\mV$ be a presentably monoidal category.
A $\mV$-enriched functor $\phi: \mC \to \mD$ is a monomorphism in $\mV\mathrm{-}\Cat$ if and only it it induces an embedding $\iota(\mC) \to \iota(\mD)$ on underlying spaces and for every $A,B \in \mC$ the induced morphism 
$\Mor_\mC(A, B) \to \Mor_\mD(\phi(A), \phi(B))$ in $\mV$ is a monomorphism.
    
\end{proposition}

The following is \cite[Example 3.1.13.]{heine2024higher}:

\begin{lemma}\label{enrsub}

Let $\mV$ be a presentably monoidal category and $\mC$ a $\mV$-enriched category. For every $X,Y \in \mC$ let
$\mE_{X,Y} \to \Mor_\mC(X,Y)$ be a monomorphism in $\mV$.
The following are equivalent:
\begin{enumerate}[\normalfont(1)]\setlength{\itemsep}{-2pt}
\item For every $X \in \mC$ the unit $\tu_\mV \to \Mor_\mC(X,X)$ in $\mV$ factors through $\mE_{X,X} \to \Mor_\mC(X,X)$
and for every $X,Y,Z \in \mC$ the composition morphism $\Mor_\mC(Y,Z) \ot \Mor_\mC(X,Y) \to \Mor_\mC(X,Z) $ in $\mV$ 
restricts to a morphism $\mE_{Y,Z} \ot \mE_{X,Y} \to \mE_{X,Z} $.

\item There is a monomorphism of $\mV$-enriched categories $\theta: \mB \to \mC$
such that for every $X,Y \in \mC$ the induced monomorphism 
$\Mor_\mB(X,Y) \to \Mor_\mC(\theta(X),\theta(Y))$ in $\mV$ identifies with 
$\mE_{\theta(X),\theta(Y)} \to \Mor_\mC(\theta(X),\theta(Y))$.

\end{enumerate}

\end{lemma}

There is a close relationship between enriched categories
and tensored and cotensored categories:

\begin{definition}Let $\mV$ be a presentably monoidal category, $\mC$ a $\mV$-enriched category and $X \in \mC, V \in \mV.$	 
\begin{enumerate}[\normalfont(1)]\setlength{\itemsep}{-2pt}
\item The tensor of $V$ and $X$ in $\mC$ is the object $V \ot X \in \mC $ such that there is a morphism
$V \to \Mor_\mC(X, V \ot X) $ in $\mV$ that induces for every $Y \in \mC$ an equivalence
$$ \Mor_\mC(V \ot X,Y) \to \Mor_\mV(V, \Mor_\mC(X,Y)). $$ 
\item The cotensor of $V$ and $X$ in $\mC$ is the object ${^V X} \in \mC $ that is the tensor of $V $ and $X$ in the opposite $\mV^\rev$-enriched category $\mC^\circ.$
\end{enumerate}	
\end{definition}

Since the category of enriched categories forms a 2-category, there is a natural intrinsic notion of adjunction between enriched categories:

\begin{definition}Let $\mV$ be a presentably monoidal category.
A $\mV$-enriched functor $\mC \to \mD$ admits a left (right) adjoint if
it admits a left (right) adjoint in the 2-category $\mV \mathrm{-}\Cat.$

\end{definition}

\begin{notation}Let $\mV$ be a presentably monoidal category and $\mM, \mN$ be $\mV$-enriched categories.

Let $$\mV\mathrm{-}\Fun(\mM,\mN)$$ be the category of $\mV$-enriched functors
$\mM \to \mN.$

Let $$\mV\mathrm{-}\Fun^\L(\mM,\mN) \subset {\mV\mathrm{-}\Fun(\mM,\mN)}$$ be the full subcategory of $\mV$-enriched functors
$\mM \to \mN$ that admit a $\mV$-enriched right adjoint.
    
\end{notation}

The following is \cite[Remark 2.75.]{heine2024bienriched}:

\begin{proposition}
Let $\mV$ be a presentably monoidal category.
\begin{enumerate}[\normalfont(1)]\setlength{\itemsep}{-2pt}
\item A $\mV$-enriched functor $\phi: \mC \to \mD$ admits a $\mV$-enriched right adjoint if and only if for every $\Y \in \mD$ the $\mV^\rev$-enriched functor
$\Mor_\mD(\phi(-),\X): \mC^\circ \to \mV$ is representable.
\item A $\mV$-enriched functor $\phi: \mC \to \mD$ admits a $\mV$-enriched left adjoint if and only if the opposite $\mV^\rev$-enriched functor $\phi^\circ: \mC^\circ \to \mD^\circ$ admits a $\mV^\rev$-enriched right adjoint. By (1) this holds if and only if for every $\Y \in \mD$ the $\mV$-enriched functor
$\Mor_\mD(\Y,\phi(-)): \mC \to \mV$ is representable.
\end{enumerate}
\end{proposition}

The following is \cite[Lemma 2.77.]{heine2024bienriched}:

\begin{proposition}\label{adj}
Let $\mV$ be a presentably monoidal category.
\begin{enumerate}[\normalfont(1)]\setlength{\itemsep}{-2pt}
\item A $\mV$-enriched functor $\mC \to \mD$ admits a right adjoint if and only if it preserves tensors and the underlying functor admits a right adjoint.
\item A $\mV$-enriched functor $\mC \to \mD$ admits a left adjoint if and only if it preserves cotensors and the underlying functor admits a left adjoint.
\end{enumerate}

\end{proposition}

Next we introduce the enriched category of enriched presheaves \cite{heine2025equivalence}.

\begin{notation}

Let $\mV, \mW$ be presentably monoidal categories and $\mM$ a $\mV$-enriched category and $\mO$ a $\mW$-enriched category.
Let $\langle \mM, \mN \rangle $ be the $\mV \ot \mW$-enriched category that is the transfer of enrichment of the $\mV \times \mW$-enriched category $\mM \times \mN$ along the universal
monoidal functor $\mV \times \mW \to \mV \ot \mW$ preserving small colimits componentwise.
    
\end{notation}

The next theorem follows from \cite[Proposition 4.11, Theorem 4.86]{heine2024bienriched}:

\begin{theorem}\label{psinho} Let $\mV, \mW$ be presentably monoidal categories and $\mN$ a univalent $(\mV, \mW)$-bienriched category. 
\begin{enumerate}[\normalfont(1)]\setlength{\itemsep}{-2pt}
\item Let $\mM$ be a small univalent $\mV$-enriched category. The
category ${\mV\mathrm{-}\Fun}(\mM, \mN)$ refines to an univalent $\mW$-enriched category characterized by an 
equivalence
$$ \mW\mathrm{-}\Fun(\mO,{\mV\mathrm{-}\Fun}(\mM, \mN)) \to {\mV \ot \mW\mathrm{-}\Fun}(\langle\mM,\mO\rangle,\mN)$$
natural in any univalent $\mW$-enriched category $\mO$.
\item Let $\mO$ be a small univalent $\mW$-enriched category. 
The category ${\mW\mathrm{-}\Fun}(\mO, \mN)$ refines to an univalent $\mV$-enriched category characterized by an equivalence
$$ \mV\mathrm{-}\Fun(\mM,{\mW\mathrm{-}\Fun}(\mO, \mN)) \to {\mV \ot \mW\mathrm{-}\Fun}(\langle\mM,\mO\rangle,\mN) $$
natural in any univalent $\mV$-enriched category $\mM$.

\item If $\mN$ is a presentably $\mV, \mW$-bitensored category, then ${\mV\mathrm{-}\Fun}(\mM, {\mN}) $ is a presentably right $\mW$-tensored category and ${\mW\mathrm{-}\Fun}(\mO, {\mN}) $ is a presentably left $\mV$-tensored category.

\end{enumerate}

\end{theorem}
    
\begin{definition}

Let $\mV$ be a presentably monoidal category and $\mC$ a small univalent $\mV$-enriched category.
The presentably left $\mV$-tensored category of $\mV$-enriched presheaves on $\mC$ is
$$ \mP_\mV(\mC):= \mV\mathrm{-}\Fun(\mC^\op,\mV). $$
\end{definition}

The next theorem, which follows from \cite[Theorem 3.41, Theorem 4.70]{heine2024bienriched}, is an enriched version of the universal property of the category of presheaves as the free cocompletion under small colimits:
%\cite[Theorem 5.1.5.6]{lurie.HTT}:

\begin{theorem}
\label{Yonedaext}
Let $\mV$ be a presentably monoidal category, $\mC$ a small univalent $\mV$-enriched category and $\mD$ a presentably left $\mV$-tensored category.
\begin{enumerate}[\normalfont(1)]\setlength{\itemsep}{-2pt}
\item There is a $\mV$-enriched embedding $\iota_\mC : \mC \to \mP_\mV(\mC)$ that sends $X$ to $\L\Mor_\mC(-,X)$ and induces
for every presentably left $\mV$-tensored category $\mD$ an equivalence
$$ {\mV\mathrm{-}\Fun^\L}(\mP_\mV(\mC),\mD)\to \mV\mathrm{-}\Fun(\mC,\mD).$$
\item Let $F: \mC \to \mD$ be a $\mV$-enriched functor and
$\bar{F}: \mP_\mV(\mC) \to \mD $ the unique $\mV$-enriched left adjoint extension of $F$.
For every $\mV$-enriched functor $G: \mP_\mV(\mC) \to \mD $
the induced morphism $$\Map_{{\mV\mathrm{-}\Fun^\L}(\mP_\mV(\mC),\mD)}(\bar{F},G) \to \Map_{\mV\mathrm{-}\Fun(\mC,\mD)}(F,G \circ \iota_\mC) $$ is an equivalence.
\item Let $F: \mC \to \mD$ be a $\mV$-enriched functor and
$\bar{F}: \mP_\mV(\mC) \to \mD $ the unique $\mV$-enriched left adjoint extension of $F$.
The $\mV$-enriched right adjoint $\mD \to \mP_\mV(\mC)$ of $\bar{F}$ sends $Y$ to $\Mor_\mD(-,Y) \circ F. $
\end{enumerate}
\end{theorem}

The following is the enriched Yoneda-lemma proven in 
\cite[Corollary 4.44]{heine2024bienriched}:

\begin{lemma}
Let $\mV$ be a presentably monoidal category and $\mC$ a small univalent $\mV$-enriched $\infty$-category. For every object $X \in \mC$ and $F \in \mP_\mV(\mC)$ the induced morphism $$\Mor_{\mP_\mV(\mC)}(\Mor_\mC(-,X),F) \to F(X) $$ is an equivalence.

\end{lemma}

We will use the following terminology to describe compatible enrichments in two monoidal categories:

\begin{definition}\label{bienr}
Let $\mV$ and $\mW$ be presentably monoidal categories.	 
\begin{enumerate}[\normalfont(1)]\setlength{\itemsep}{-2pt}
\item A left $\mV$-enriched category is a $\mV$-enriched category.
\item A left $\mV$-enriched functor is a $\mV$-enriched functor
\item A right $\mV$-enriched category is a $\mV^\rev$-enriched category.
\item A right $\mV$-enriched functor is a $\mV^\rev$-enriched functor.
\item A $\mV,\mW$-bienriched category is a $\mV \ot \mW^\rev$-enriched category.
\item A $\mV,\mW$-enriched functor is a $\mV \ot \mW^\rev$-enriched functor.
\end{enumerate}
\end{definition}

We often refer to a $(\mV,\mV)$-bienriched category simply as $\mV$-bienriched category.

\begin{notation}
Let $\mV, \mW$ be presentably monoidal categories.
\begin{enumerate}[\normalfont(1)]\setlength{\itemsep}{-2pt}
\item Let $${_\mV \Cat}:= {\mV}\mathrm{-}\Cat$$
be the 2-category of left $\mV$-enriched categories and left $\mV$-enriched functors.
\item Let $${\Cat_\mW}:= {\mW^\rev}\mathrm{-}\Cat$$
be the 2-category of right $\mW$-enriched categories and right $\mW$-enriched functors.
\item Let $${_\mV \Cat_\mW}:= {\mV \ot \mW^\rev}\mathrm{-}\Cat$$
be the 2-category of $\mV,\mW$-bienriched categories and $\mV,\mW$-enriched functors.
\end{enumerate}
\end{notation}

\begin{notation}

We will write $\L\Mor, \R\Mor, \mathrm{B}\Mor$ for the morphism objects in a left, right and bienriched category, respectively.    
\end{notation}

\begin{notation}
Let $\mV, \mW$ be presentably monoidal categories and $\mM, \mN$ be $(\mV,\mW)$-bienriched categories.

Let $$_\mV\Fun_\mW(\mM,\mN):= {\mV \ot \mW^\rev \mathrm{-}\Fun(\mM,\mN)}$$ be the category of $\mV,\mW$-enriched functors $\mM \to \mN.$

\end{notation}

The following is \cite[Example 2.134.]{heine2024bienriched}:

\begin{example}
Let $\mV,\mW$ be presentably monoidal categories. Every presentably $\mV,\mW$-bitensored category is a $\mV,\mW$-enriched category.
Every presentably right $\mW$-tensored category is a right $\mW$-enriched category.
This follows also from \cref{linenr} identifying
presentably $\mV,\mW$-bitensored categories with presentably left $\mV \ot \mW^\rev$-tensored categories, and identifying
presentably right $\mW$-tensored categories with presentably left $\mW^\rev$-tensored categories.
In particular, $\mV$, which is presentably bitensored over itself, is a $\mV,\mV$-bienriched category.

Similarly, every $\mV,\mW$-linear functor between presentably $\mV,\mW$-bitensored categories is a $\mV,\mW$-enriched functor.

\end{example}

\begin{definition}Let $\mV$ and $\mW$ be presentably monoidal categories, let $\mC$ be a  
$\mV,\mW$-enriched category, and let $X \in \mC, V \in \mV, W \in \mW.$	 
\begin{enumerate}[\normalfont(1)]\setlength{\itemsep}{-2pt}
\item The left (co)tensor of $V$ and $X$ in $\mC$ is the (co)tensor of
$V \ot \tu_\mW \in \mV \ot \mW^\rev$ and $X$ in $\mC.$

\item The right (co)tensor of $W$ and $X$ in $\mC$ is the (co)tensor of
$\tu_\mV \ot W \in \mV \ot \mW^\rev$ and $X$ in $\mC.$

\end{enumerate}

\end{definition}

In the following we consider enriched slice categories.
To define bienriched slice categories we use the following lemma, which is \cite[Lemma 2.1.23.]{oriented}:

\begin{lemma}\label{tensorunit}
Let $\mV, \mW$ be presentably monoidal categories whose tensor unit is final.
The tensor unit of $\mV \ot \mW$ is final.
	
\end{lemma}

% \begin{proof}

% By \cite[Proposition 7.15]{FreeAlgebras} there are 
% small regular cardinals $\kappa, \tau$ such that
% the monoidal structures on $\mV,\mW$ restrict to the full subcategories $\mV^\kappa \subset \mV , \mW^\tau \subset \mW$ of $\kappa, \tau$-compact objects, respectively.
% The latter monoidal embeddings induce monoidal equivalences
% $\Ind_\kappa(\mV^\kappa) \simeq \mV, \Ind_\kappa(\mW^\kappa) \simeq \mW$.
% The monoidal localizations $\mP(\mV^\kappa) \rightleftarrows \Ind_\kappa(\mV^\kappa) \simeq \mV,$
% $\mP(\mW^\tau) \rightleftarrows \Ind_\tau(\mW^\kappa) \simeq \mW $
% induce a monoidal localization 
% $\mP(\mV^\kappa \times \mW^\tau) \simeq \mP(\mV^\kappa) \ot \mP(\mW^\tau) \rightleftarrows \mV \ot \mW . $
% By assumption the tensor unit of $ \mV^\kappa \times \mW^\tau $ is final.
% Since the Yoneda embedding is monoidal and preserves the final object, the tensor unit of $\mP(\mV^\kappa \times \mW^\tau) $ is final and so local for any localization.
% So the tensor unit of $\mP(\mV^\kappa \times \mW^\tau) $ belongs to $\mV \ot \mW.$
% Since there is a monoidal localization functor $\mP(\mV^\kappa \times \mW^\tau) \to \mV \ot \mW$,
% the tensor unit of $\mV \ot \mW$ is equivalent to the tensor unit of $\mP(\mV^\kappa \times \mW^\tau) $, which is final in $\mP(\mV^\kappa \times \mW^\tau) $ and so in $\mV \ot \mW.$ 
% \end{proof}

\begin{corollary}\label{finality}
	
Let $\mV, \mW$ be presentably monoidal categories whose tensor unit is final.
Then ${_\mV \Cat_\mW}$ admits a final object $*$.
	
\end{corollary}

\begin{corollary}Let $\mV, \mW$ be presentably monoidal categories whose tensor unit is final.
For every $(\mV,\mW)$-bienriched category $\mC$
the induced functor ${\mV\mathrm{-}\Fun_\mW}(*,\mC) \to \mC$ is an equivalence.
    
\end{corollary}

For the next notation we that the 2-category $\mV\mathrm{-}\Cat$ admits cotensors.

\begin{notation}
Let $\mV$ be a presentably monoidal category whose tensor unit is final, $\mC$ a $\mV$-enriched category and $X $ an object of $\mC.$
Let $$ \mC_{\X/}:= \{X\} \times_{\mC^{\{0\}}} \mC^{\bD^1} $$ be the pullback of the $\mV$-enriched functor $\mC^{\bD^1} \to \mC^{\{0\}}$ evaluating at the source along the $\mV$-enriched functor $* \to \mC$ classifying $X$. 
By definition there is a $\mV$-enriched functor $\mC_{\X/}\to \mC$.
    
\end{notation}

\begin{remark}\label{init}
Let $\mV$ be a presentably monoidal category whose tensor unit is final, $\mC$ a $\mV$-enriched category and $\X \to \Y, \X \to \Z$ morphisms in $\mC.$
The induced morphism $$\tu \to \Mor_{\mC_{\X/}}(\X,\X) \to \Mor_{\mC_{\X/}}(\X,\Z)$$ is an equivalence and the resulting commutative square
$$\begin{xy}
\xymatrix{
\Mor_{\mC_{\X/}}(\Y,\Z) \ar[d]^{} \ar[r]
& \Mor_{\mC}(\Y,\Z) \ar[d] \ar[d]^{}
\\ 
\tu \simeq \Mor_{\mC_{\X/}}(\X,\Z) \ar[r] & \Mor_{\mC}(\X,\Z)
}
\end{xy}$$
is a pullback square.
    
\end{remark}

The following is \cite[Lemma 2.2.3.]{oriented}:

\begin{lemma}

Let $\mV$ be a presentably monoidal category whose tensor unit is final, $\mC$ a $\mV$-enriched category and $X $ an object of $\mC.$
If $\mC$ admits small weakly contractible conical colimits, then $\mC_{\X/}$ admits small colimits.

\end{lemma}

\begin{remark}
Let $\mV$ be a presentably monoidal category whose tensor unit is final, $\mC, \mD$ be $\mV$-enriched categories, $F: \mC \to \mD$ a $\mV$-enriched functor 
and $Y \to X $ a morphism in $\mC.$
There is an induced $\mV$-enriched functor 
$\mC_{\X/} \to \mD_{F(\Y)/}$ that fits into a commutative square 
$$\begin{xy}
\xymatrix{
\mC_{\X/} \ar[d]^{} \ar[r]
&  \mD_{F(\Y)/} \ar[d] \ar[d]^{}
\\ 
\mC
\ar[r]^F & \mD
}
\end{xy}$$	
of $\mV$-enriched categories.
    
\end{remark}

The following is \cite[Lemma 2.2.5.]{oriented}:

\begin{lemma}\label{bien}
Let $\mV$ be a presentably monoidal category whose tensor unit is final, $\mC$ a $\mV$-enriched category and $X \to Y$ a morphism in $\mC.$
If $\mC$ admits conical pushouts, the $\mV$-enriched functor $\mC_{\Y/} \to \mC_{\X/}$ admits a $\mV$-enriched left adjoint.
	
\end{lemma}

The following is \cite[Corollary 2.2.2.]{GepnerHeine2026}:

\begin{corollary}\label{susp}

Let $\mV$ be a presentably monoidal category and $\A \in \mV$.
There is a $\mV$-enriched category $S(\A) $, which we call the suspension of $A$, satisfying the following properties:
\begin{enumerate}[\normalfont(1)]\setlength{\itemsep}{-2pt}
\item The space of objects of $S(\A) $ is the set $\{0,1\}.$

\item For every $0 \leq \ell \leq 1$ the unit $\tu \to \Mor_{S(\A)}(\ell,\ell)$ is an equivalence.

\item The morphism object $\Mor_{S(\A)}(1,0)$ is initial.

\item There is an equivalence $ \Mor_{S(\A)}(0,1) \simeq \A$. 

\item For every $\mV$-enriched category $\mC $ and objects $\X,\Y$ of $\mC$
the induced map is an equivalence $$ \Map_{\mV\mathrm{-}\Cat_{B(\tu) \coprod B(\tu)/}}(S(\A), (\mC; \X,\Y)) \to \Map_\mV(\A, \Mor_\mC(\X,\Y)).$$
\end{enumerate}
    
\end{corollary}

\begin{definition}
Let $\mV$ be a presentably monoidal category, $\mC, \mD$ be $\mV$-enriched categories and $X,Y \in \mC, Y,Z \in \mD$
objects. The bipointed wedge $(\mC; X,Y) \vee (\mD; Y,Z) $
is the pushout $ (\mC \coprod_{\{Y\}} \mD; X,Z).$
    
\end{definition}

\begin{corollary}

Let $\mV$ be a presentably monoidal category, $n \geq 1$ and $\A_1, ..., \A_\n \in \mV$.
\begin{enumerate}[\normalfont(1)]\setlength{\itemsep}{-2pt}
\item The space of objects of the $\mV$-enriched category $S(A_1) \vee ...\vee S(A_n)$ is the set $\{0,...,\n\}.$
\item For every $0 \leq \ell \leq \n$ the unit $\tu \to \Mor_{S(A_1) \vee ...\vee S(A_n)}(\ell,\ell)$ is an equivalence.
\item For every $0 \leq \bk < \ell \leq \n$ the morphism object $\Mor_{S(A_1) \vee ...\vee S(A_n)}(\ell,\bk)$ is initial.
\item For every $0 \leq \ell < \n$ there is an equivalence $ \Mor_{S(A_1) \vee ...\vee S(A_n)}(\ell,\ell+1) \simeq \A_{\ell+1}$. 
\item For every $0 \leq \bk < \m \leq \n$ the following induced morphism is an equivalence $$\bigotimes_{\bk \leq \ell < \m}  \A_{\ell+1} \simeq \bigotimes_{\bk \leq \ell < \m} \Mor_{S(A_1) \vee ...\vee S(A_n)}(\ell,\ell+1) \to \Mor_{S(A_1) \vee ...\vee S(A_n)}(\bk,\m).$$
\end{enumerate}
\end{corollary}

\begin{notation}\label{Thetasgen}Let $\mV$ be a small monoidal category. Let $\Theta(\mV) \subset \mV\mathrm{-}\Cat$ be the full subcategory spanned by the wedges of suspensions of objects of $\mV$ and $B\tu_\mV$.

\end{notation}

\begin{definition}

Let $\mV$ be a small monoidal category.
A presheaf on $\Theta(\mV) $ satisfies the Segal condition 
if for every $n \geq 2$ and $X_1,...,X_n \in \mV$
the following canonical map is an equivalence:
$$ F(S(X_1) \vee\cdots\vee S(X_n)) \to F(S(X_1)) \times_{F(B\tu_\mV)}\times\cdots \times_{F(B\tu_\mV)} F(S(X_n)).$$
    
\end{definition}

\begin{definition}Let $\mV$ be a small monoidal category, $F$ a presheaf on $\Theta(\mV)$ and $A,B \in F(B\tu_\mV). $ 
The presheaf of morphisms $A$ to $B$ in $F$ is the following presheaf on $\mV$, where $S: \mV \to \mV\mathrm{-}\Cat_{B\tu_\mV \coprod B\tu_\mV/}$ is the suspension:
$$\Mor_F(A,B):= (F \circ S) \times_{F(B\tu_\mV) \times F(B\tu_\mV)} \{(A,B)\}.$$
    
\end{definition}

The following is \cite[Theorem 2.4.8.]{oriented}:

\begin{theorem}\label{denseinherited} Let $\mW$ be a presentably monoidal category and $\mV$ a small dense full monoidal subcategory of $\mW$.
The full subcategory $\Theta(\mV)$ of $\mV\mathrm{-}\Cat$ spanned by the wedges of suspensions of objects of $\mV$ and $B\tu_\mV$ is dense in $\mW\mathrm{-}\Cat.$
A presheaf on $\Theta(\mV) $ belongs to the essential image of the $\Theta(\mV)$-nerve
if and only if it satisfies the Segal condition and 
all morphism presheaves belong to the essential image of the $\mW$-nerve.

\end{theorem}

\subsection{$\infty$-Categories}
Our main application of theory of enriched categories will be to $\infty$-categories and oriented categories.
We recall these basic notions in the following two subsections.

\begin{definition}
For every $\n \geq 0$ we inductively define the presentable cartesian closed category $\n\Cat$ of small (not necessarily univalent) $\n$-categories by setting
$$(\n+1)\Cat:= {{\n\Cat}\mathrm{-}\Cat} $$
starting with $$ 0\Cat :=\mS.$$
\end{definition}

\begin{notation}
For every $\n \geq 0$ we inductively define colocalizations
$$\n\Cat \rightleftarrows (\n+1)\Cat: \iota_\n,$$
where both adjoints  preserve finite products and filtered colimits.
Let
$$0\Cat= \mS \rightleftarrows 1\Cat = {\mS\mathrm{-}\Cat} : \iota_0 $$ be the canonical colocalization whose right adjoint assigns the space of objects. Let $$(\n+1)\Cat= {{\n\Cat}\mathrm{-}\Cat} \rightleftarrows (\n+2)\Cat= {{(\n+1)\Cat}\mathrm{-}\Cat}:\iota_{\n+1}:= (\iota_\n)_! $$
be the induced adjunction.	
	
\end{notation}

\begin{definition}The presentable category $\infty\Cat$ of small (non-univalent) $\infty$-categories is the limit
$$\infty\Cat:= \lim(\cdots\xrightarrow{\iota_{\n}} \n \Cat \xrightarrow{\iota_{\n-1}}\cdots \xrightarrow{\iota_0} 0 \Cat) $$
of presentable categories and right adjoint functors.

\end{definition}

The next proposition follows from the fact that the functor ${(-)\mathrm{-}\Cat}$ preserves small limits and is \cite[Proposition 2.4.4.]{GepnerHeine2026}:

\begin{proposition}\label{fix}
	
There is a canonical equivalence
$$ \infty\Cat \simeq {\infty\Cat}\mathrm{-}\Cat. $$
	
\end{proposition}

\begin{notation}
Let $\partial\bD^1$ denote the two element set $\{0,1\}$.
\end{notation}

\begin{notation}
Let $\Mor: \infty\Cat_{\partial\bD^1/} \to \infty\Cat$
be the canonical functor $$\infty\Cat_{\partial\bD^1/} \simeq \mS_{\partial\bD^1/}\times_\mS {\infty\Cat}\mathrm{-}\Cat \to \infty\Cat $$
sending $(\mC,\X,\Y)$ to $\Mor_\mC(\X,\Y).$
\end{notation}

\begin{remark}\label{homfil}
The functor $\Mor: \infty\Cat_{\partial\bD^1/} \to \infty\Cat$
preserves small filtered colimits and limits.

\end{remark}

\begin{definition}

A functor $X \to Y$ of $\infty$-categories is an inclusion -- or subcategory inclusion-- if is a monomorphism in the category $\infty\Cat,$ i.e. for every $\infty$-category $Z$ the induced map $$\Map_{\infty\Cat}(Z,X) \to \Map_{\infty\Cat}(Z,Y)$$ is an embedding.
In this case we also say that $X$ is a subcategory of $Y$.
 
\end{definition}

\cref{monochar} implies the following:

\begin{corollary}

A functor $\phi: X \to Y$ of $\infty$-categories is an inclusion if and only if it induces an embedding $\iota_0(X) \to \iota_0(Y)$ on underlying spaces and for every $A,B \in X$ the induced functor
$$\Mor_X(A, B) \to \Mor_Y(\phi(A), \phi(B))$$ is an inclusion.

\end{corollary}

The following is \cite[Proposition 2.5.9.]{GepnerHeine2026}:

\begin{proposition}\label{subchar} Let $X$ be an $\infty$-category and $\mE$ a collection of cells of $X.$
The following are equivalent:

\begin{enumerate}[\normalfont(1)]\setlength{\itemsep}{-2pt}

\item There is a subcategory inclusion $Y \to X$ such that for every
$n \geq 0$ an $n$-morphism of $X$ belongs to $Y$ if and only if 
it belongs to $\mE.$

\item The collection $\mE$ contains all identity 1-morphism between objects of $X$ in $\mE$, for every objects $A, B $ of $ X$ in $\mE$ there is a subcategory inclusion $Y_{A,B} \to \Mor_X(A,B)$
such that an $n$-morphism of $\Mor_X(A,B)$ belongs to $Y_{A,B}$ if and only if the corresponding $n+1$-morphism of $X$ belongs to $\mE,$
and for every morphisms $A' \to A, B \to B' $ of $X$ in $\mE$ the functor $\Mor_X(A,B) \to \Mor_X(A',B')$ sends $Y_{A,B} $ to $Y_{A',B'} $.

\end{enumerate}

\end{proposition}

\begin{definition}\label{ruik} Let $\n \geq 0.$
We inductively define involutions $$(-)^\op_\n, (-)^\co_\n: \n\Cat \to \n\Cat$$ by setting
$(-)^\co_0, (-)^\op_0:  0\Cat \to 0\Cat$ are the identities, $$(-)^\op_{\n+1}: {(\n+1)}\Cat \xrightarrow{(-)^\circ} {(\n+1)}\Cat \xrightarrow{((-)^\co_\n)_!}{(\n+1)}\Cat, $$ $$(-)^\co_{\n+1}:=((-)^\op_\n)_!: {(\n+1)}\Cat \xrightarrow{ }{(\n+1)}\Cat.$$
There are commutative squares:
$$\begin{xy}
\xymatrix{
{(\n+1)}\Cat \ar[d]^{\iota_\n}  \ar[r]^{(-)_{\n+1}^\op} & {(\n+1)}\Cat  \ar[d]^{\iota_\n}
\\ 
{\n}\Cat \ar[r]^{(-)_{\n}^\op} & {\n}\Cat
}
\end{xy}
\qquad
\begin{xy}
\xymatrix{
{(\n+1)}\Cat \ar[d]^{\iota_\n}  \ar[r]^{(-)_{\n+1}^\co} & {(\n+1)}\Cat  \ar[d]^{\iota_\n}
\\ 
{\n}\Cat \ar[r]^{(-)_{\n}^\co} & {\n}\Cat
}
\end{xy}
$$
and so induced involutions on the limit $$(-)^\op, (-)^\co: \infty\Cat \to \infty\Cat. $$

\end{definition}

\begin{remark}\label{oppo} By \cref{ruik} there are commutative squares, where $\sigma$ permutes the distinguished objects:

$$\begin{xy}
\xymatrix{
\infty\Cat_{\partial\bD^1/} \ar[d]^{\Mor}  \ar[r]^{(-)^\co} & \infty\Cat_{\partial\bD^1/}  \ar[d]^{\Mor}
\\ 
\infty\Cat \ar[r]^{(-)^\op} & \infty\Cat
}
\end{xy}\qquad
\begin{xy}
\xymatrix{
\infty\Cat_{\partial\bD^1/} \ar[d]^{\Mor}  \ar[r]^{\sigma \circ (-)^\op} & \infty\Cat_{\partial\bD^1/}  \ar[d]^{\Mor}
\\ 
\infty\Cat \ar[r]^{(-)^\co} & \infty\Cat.
}
\end{xy}
$$

\end{remark}

The following follows from \cref{susp}:

\begin{proposition}\label{suspi}

The functor $\Mor: \infty\Cat_{\partial\bD^1/} \to \infty\Cat$ admits a left adjoint
$S: \infty\Cat \to \infty\Cat_{\partial\bD^1/}$ such that for every $\infty$-category $\mC$ the $\infty$-category $S(\mC),$ called categorical suspension of $\mC$, has two objects 0,1, and morphism $\infty$-categories:
$$\Mor_{S(\mC)}(1,0)\simeq \emptyset, \ \Mor_{S(\mC)}(0,1)\simeq \mC, \ \Mor_{S(\mC)}(0,0)\simeq \Mor_{S(\mC)}(1,1) \simeq *.$$
    
\end{proposition}

\begin{notation}
For every $0 \leq \n \leq \m$ the left adjoint embeddings $\n\Cat \leftrightarrows \m\Cat$ preserve small limits and thus induce a left adjoint embedding $\n\Cat \leftrightarrows \infty\Cat: \iota_\n$ that preserves small limits
and so admits a left adjoint $\tau_\n: \infty\Cat \to \n\Cat$ by presentability.

\end{notation}

\begin{remark}
The $\iota_n$ collectively filter every $\infty$-category $X$ as a colimit $X\simeq\colim_{n}\iota_n(X)$.
\end{remark}

The next is \cite[Lemma 2.4.18.]{GepnerHeine2026}:

\begin{lemma}\label{carclo}
The presentable category $\infty\Cat$ is cartesian closed.
\end{lemma}

\begin{notation}

For any $\infty$-category $\mC$ let 
$\Fun(\mC,-): \infty\Cat \to \infty\Cat$ be the right adjoint of the functor $(-) \times \mC:\infty\Cat \to \infty\Cat.$
    
\end{notation}

\begin{remark}
Since $\infty\Cat$ is cartesian closed, it refines to an $\infty$-category $\infty\scat$ such that for every two objects $X$ and $Y$,
\[
\Mor_{\infty\scat}(X,Y)=\Fun(X,Y).
\]
\end{remark}

\begin{definition}
    We write $\infty\scat$ for the $\infty$-category of small $\infty$-categories.
\end{definition}

\begin{definition}Let $\n \geq 0$.
The $\n$-disk is the $n$-fold iterated suspension $\bD^\n:= S^{\n}(*)$ of the terminal $\infty$-category $\ast$.
\end{definition}

\begin{definition}Let $\n \geq 0$.
The boundary of the $\n$-disk is the $n$-fold iterated suspension $\partial\bD^\n:= S^{\n}(\emptyset)$ of the initial $\infty$-category $\emptyset$.
	
\end{definition}

\begin{remark}Let $\n \geq 0.$
The functor $\emptyset \subset *$ induces an inclusion $\partial\bD^\n \subset \bD^\n$.
	
\end{remark}

\begin{example}
Then $\partial\bD^0=\ast, \partial\bD^1=S(\emptyset)=*\coprod*$ is the set with two elements.
Viewed as a subject of $\bD^1$, these elements acquire a natural ordering.
\end{example}

\begin{definition}
The bipointed wedge of an ordered pair $(\mC,\mD)$ of bipointed $\infty$-categories  $\infty\Cat_{/\partial\bD^1}$ and $\mC'\in\infty\Cat_{/\partial\bD^1}$ is the bipointed $\infty$-category obtained by taking the pushout along the common basepoint.
\end{definition}

\begin{definition}\label{Theta}
	
Let $\Theta' \subset \infty\Cat_{\partial\bD^1/}$ be the full subcategory generated by $\bD^0$ under suspensions and bipointed wedges and $\Theta \subset \infty\Cat$ the essential image of $\Theta'$ under the forgetful functor.	
	
\end{definition}

\begin{remark}

An $\infty$-category belongs to $\Theta$ if and only if it is of the form $$ \bD^{i_0} \coprod_{\bD^{j_1}} \bD^{i_1} \coprod_{\bD^{j_2}} \cdots \coprod_{\bD^{j_n}} \bD^{i_n} $$
for any sequence of natural numbers $n, i_0,\ldots, i_n, j_1,\ldots, j_n$ and monomorphisms 
$ \bD^{j_\ell} \rightarrowtail \bD^{i_\ell},\bD^{j_\ell} \rightarrowtail \bD^{i_{\ell-1}}$, $ 1 \leq \ell \leq n$.

\end{remark}

The next is \cite[Theorem 2.5.8.]{GepnerHeine2026}:

\begin{theorem}\label{theta}

The restricted Yoneda embedding $\N:
\infty\Cat \to \Fun(\Theta^\op,\mS)$ is fully faithful and admit left adjoints that preserves finite products.
A $\Theta$-space is in the essential image of the restricted Yoneda embedding if and only if it satisfies the Segal condition, i.e. it is local with respect to the map 
$$ \N(\bD^{i_0}) \coprod_{\N(\bD^{j_1})} \N(\bD^{i_1})\coprod_{\N(\bD^{j_2})} \cdots \N(\coprod_{\bD^{j_n}}) \N(\bD^{i_n}) \to \N(\bD^{i_0} \coprod_{\bD^{j_1}} \bD^{i_1} \coprod_{\bD^{j_2}} \cdots \coprod_{\bD^{j_n}} \bD^{i_n}) $$
for any sequence of natural numbers $n, i_0,\ldots, i_n, j_1,\ldots, j_n$ and monomorphisms 
$ \bD^{j_\ell} \rightarrowtail \bD^{i_\ell},\bD^{j_\ell} \rightarrowtail \bD^{i_{\ell-1}}$, $ 1 \leq \ell \leq n$. 
For every $ 0 \leq n \leq \infty$ the full subcategory $n\Cat \subset \infty\Cat$ is generated under small colimits by the disks of dimension smaller $n+1.$

\end{theorem}

By \cite[Definition 2.6.25.]{GepnerHeine2026} for every $n \geq 0$ there is an oriented $n$-cube $\cube^n$
whose 1-truncation is $(\bD^1)^{\times n}.$

\vspace{1mm}

\begin{notation}\label{sqGray}

Let $\cube \subset \infty\Cat$ be the full subcategory of oriented cubes.
\end{notation}

The following is \cite[Theorem 2.8.12.]{GepnerHeine2026}:

\begin{theorem}\label{cubedense}

The full subcategory $\cube \subset \infty\Cat$ is dense.
    
\end{theorem}

By \cite[Definition 2.6.32.]{GepnerHeine2026} for every $n \geq 0$ there is an oriented $n$-simplex $\bDelta^n$
whose 1-truncation is the totally ordered set $[n]= \{0 < ...< n \}.$

\begin{notation}

Let $\bDelta \subset\infty\Cat$ be the full subcategory spanned by the oriented simplices.
Let $\bDelta^+ \subset\infty\Cat$ be the full subcategory spanned by the oriented simplices and the initial object.
  
\end{notation}

The next is \cite[Theorem 2.7.8.]{GepnerHeine2026}:

\begin{theorem}\label{orientdense}

The full subcategory $\bDelta\subset\infty\Cat$ is dense.
    
\end{theorem}

In the following we recall the Gray tensor product.
By \cite[Definition 2.6.23.]{GepnerHeine2026} there is a canonical monoidal structure on $\cube$ whose tensor unit is
$\bD^0$ and such that $\cube^n \boxtimes \cube^m= \cube^{n+m}$ for every $n,m \geq 0.$

We have the following result of Campion \cite{campion2022cubesdenseinftyinftycategories}, which is also proven in \cite[Corollary 3.4.2.]{GepnerHeine2026}:

\begin{corollary}\label{locmon2}
Let $\cube \subset \infty\Cat$ denote the full subcategory of oriented cubes, equipped with the Gray monoidal structure.
The convolution monoidal structure descends along 
the localization $$\mP(\cube) \rightleftarrows \infty\Cat.$$

\end{corollary}

\begin{notation}
	
Since the Gray tensor product defines a presentably monoidal structure on $\infty\Cat$, it is closed: for every $\infty$-category $\mC$ the functor $$ \mC \boxtimes (-): \infty\Cat \to \infty\Cat$$ admits a right adjoint $\Fun^\lax(\mC,-)$, and the functor $$ (-) \boxtimes \mC : \infty\Cat \to \infty\Cat$$ admits a right adjoint $\Fun^\oplax(\mC,-)$.
\end{notation}

\begin{definition}Let $\F,\G:\mC \to \mD $ be functors of $\infty$-categories.
\begin{enumerate}[\normalfont(1)]\setlength{\itemsep}{-2pt}
\item A lax natural transformation $\F \to \G$ is a morphism in $\Fun^\lax(\mC,\mD)$.

\item An oplax natural transformation $\F \to \G$ is a morphism in $\Fun^\oplax(\mC,\mD)$.
\end{enumerate}
\end{definition}

\begin{remark}
	
Let $\mC,\mD$ be $\infty$-categories. The canonical functors $\mC \to *, \mD \to *$
give rise to functors $\mC \boxtimes \mD \to \mC \boxtimes * \simeq \mC, \mC \boxtimes \mD \to * \boxtimes \mD \simeq \mD$ and so to a functor $\mC \boxtimes \mD \to \mC \times \mD.$
By adjointess the latter functor induces functors
$ \Fun(\mC,\mD)\to \Fun^\lax(\mC,\mD)$ and $\Fun(\mC,\mD)\to \Fun^\oplax(\mC,\mD).$
\end{remark}

\begin{remark}\label{lao}\label{grayspace}
If $\mC$ is an $\n$-category and $\mD$ an $\m$-category for $\n,\m \geq 0$,
then $\mC \boxtimes \mD$ is an $\n+\m$-category.
This holds since $\n\Cat$ is closed under small colimits in $\infty\Cat$ and $\n\Cat$ is generated under small colimits by the oriented $\ell$-cubes for $1 \leq \ell \leq \n$.
For $n=m=0$ one finds that the Gray-monoidal structure restricts to the 
full subcategory of spaces $\mS \subset \infty\Cat$.
The restricted Gray-monoidal structure on $\mS$ is the cartesian structure since $\mS$ is generated in $\infty\Cat$ under small colimits by the final category, the tensor unit for the Gray tensor product and the cartesian product.
Moreover the left adjoint $\tau_0: \infty\Cat \to \mS$ of the monoidal embedding $\mS \subset \infty\Cat$ is monoidal since
$\infty\Cat$ is generated by the oriented cubes under small colimits and the image of any oriented cube under $\tau_0$ is contractible.

\end{remark}

The following is \cite[Proposition 3.5.9.]{GepnerHeine2026}: 
\begin{proposition}\label{dua}
	
There are canonical monoidal involutions
$$(-)^\op, (-)^\co: (\infty\Cat, \boxtimes) \simeq (\infty\Cat, \boxtimes)^\rev $$
refining the involutions of \cref{ruik} and restricting to 
monoidal involutions
$$(-)^\op, (-)^\op: (\infty\Cat^{\mathrm{strict}}, \boxtimes) \simeq (\infty\Cat^{\mathrm{strict}}, \boxtimes)^\rev.$$
\end{proposition}

\begin{corollary}\label{grayhoms}
Let $\mC,\mD \in \infty\Cat$.
There are canonical equivalences $$ \Fun^\oplax(\mC,\mD)^\op \simeq \Fun^\lax(\mC^\op,\mD^\op), $$$$\Fun^\oplax(\mC,\mD)^\co \simeq \Fun^\lax(\mC^\co,\mD^\co).$$	
\end{corollary}

Next we define join and slice.

\begin{notation} For every small category $\mC$ that admits an initial object let $ \mP_{\mathrm{red}}(\mC) \subset \mP(\mC) $  
be the full subcategory of presheaves on $\mC$
that are reduced, i.e. send the initial object to the final space.
\end{notation}

\begin{notation}

Let $\bDelta^+\subset\infty\Cat$ denote the full subcategory consisting of the oriented simplices and the initial object. 

\end{notation}

By \cite[Definition 2.6.30.]{GepnerHeine2026} the category $\bDelta^+$ carries a monoidal structure, the join,
whose tensor unit is the empty $\infty$-category and such that $$\bDelta^n \star \bDelta^m = \bDelta^{n+m+1}$$ for every $n,m \geq -1.$
The following is \cite[Corollary 3.5.8.]{GepnerHeine2026}:

\begin{theorem}\label{locmon4}

The join monoidal structure descends along the localization $\mP_{\mathrm{red}}(\bDelta^+) \rightleftarrows \infty\Cat. $

\end{theorem} 

We also define the antijoin:

\begin{definition}
Let $X, Y \in \infty\Cat.$ The antijoin of $X, Y$ is
$$X \bar{\star} Y := (X^{\co} \star Y^{\co})^{\co} .$$

\end{definition}

\begin{definition}Let $X$ be a small $\infty$-category.
\begin{enumerate}[\normalfont(1)]\setlength{\itemsep}{-2pt}
\item The oplax slice, or oplax over $\infty$-category, functor is the right adjoint $$ \infty\Cat_{X/ } \to \infty\Cat, \qquad (F:X \to Y) \mapsto Y_{//^\oplax F}$$ of the functor $ (-) \star X: \infty\Cat \to \infty\Cat_{X/ }.$

\item The lax coslice, or lax under $\infty$-category, functor is the right adjoint $$ \infty\Cat_{X/ } \to \infty\Cat,\qquad (F:X \to Y) \mapsto Y_{F//^\lax }$$ of the functor $ X \star (-): \infty\Cat \to \infty\Cat_{\X / }.$

\item The lax slice, or lax over $\infty$-category, functor is the right adjoint $$ \infty\Cat_{\X / } \to \infty\Cat,\qquad (F:X \to Y) \mapsto Y_{//^\lax F} := (Y^\co_{//^\oplax F^\co})^\co $$
of the functor $ (-) \bar{\star} X: \infty\Cat \to \infty\Cat_{X/ }.$

\item The oplax coslice, or oplax under $\infty$-category, functor is the right adjoint $$ \infty\Cat_{X/ } \to \infty\Cat,\qquad (F:X \to Y) \mapsto Y_{F//^\oplax}:=(Y^\co_{F^\co//^\lax })^\co $$
of the functor $ X \bar{\star} (-): \infty\Cat \to \infty\Cat_{\X / }.$

\end{enumerate}

\end{definition}

\begin{notation}Let $F: X \to Y $ be a functor.
If unspecified, $Y_{//F}$ will refer to the oplax slice, and $Y_{F//}$ will refer to the oplax coslice.
\end{notation}

The following is \cite[Lemma 3.7.13.]{GepnerHeine2026}:

\begin{lemma}Let $F: X \to Y $ be a functor.
There is a canonical equivalence of $\infty$-categories
$$ (Y_{//^\oplax F})^\op \simeq (Y^\op)_{F^\op//^\lax }.$$
\end{lemma}

\subsection{Oriented categories}

In the following we apply the theory of bienriched $\infty$-categories to the Gray monoidal structure on $\infty\Cat$.
See \cite{heine2024bienriched} for a detailed discussion of bienriched category theory.

\begin{definition}
\begin{enumerate}[\normalfont(1)]\setlength{\itemsep}{-2pt}
\item An oriented category is a category right enriched in $(\infty\Cat,\boxtimes).$
An oriented functor is a functor right enriched in $(\infty\Cat,\boxtimes).$

\item An antioriented category is a category left enriched in $(\infty\Cat,\boxtimes).$
An antioriented functor is a functor left enriched in $(\infty\Cat,\boxtimes).$

\item A bioriented category is a category bienriched in  $((\infty\Cat,\boxtimes), (\infty\Cat,\boxtimes))$.
A bioriented functor is a functor bienriched in  $((\infty\Cat,\boxtimes), (\infty\Cat,\boxtimes))$.

\end{enumerate}

\end{definition}

\begin{notation}
\begin{enumerate}[\normalfont(1)]\setlength{\itemsep}{-2pt}
\item Let $\mC$ be an oriented category and $\X,\Y \in \mC.$	
By the structure of a right $(\infty\Cat,\boxtimes)$-enriched category there is a right morphism $\infty$-category $\R\Mor_\mC(\X,\Y)$.

\item Let $\mC$ be an antioriented category and $\X,\Y \in \mC.$	
By the structure of a left $(\infty\Cat,\boxtimes)$-enriched category there
is a left morphism $\infty$-category $\L\Mor_\mC(\X,\Y)$.

\item Let $\mC$ be a bioriented category and $\X,\Y \in \mC.$	
By the structure of a $(\infty\Cat,\boxtimes)$-bienriched category there
is a morphism $\infty$-category $\Mor_\mC(\X,\Y) \in \infty\Cat \otimes \infty\Cat.$

\end{enumerate}
	
\end{notation}

\begin{example}
The Gray monoidal structure on $\infty\Cat$ is closed and so endows $\infty\Cat$ with a bienrichment in $(\infty\Cat,\boxtimes).$ This way we see
$\infty\Cat$ as a large bioriented category, which we denote by $ \infty\fcat.$
Every full subcategory of $\infty\Cat$ inherits the structure of a bioriented category.

\end{example}	

\begin{definition}
We refer to morphims in $\boxtimes\Cat$ as {\em antioriented functors}, to morphims in $\Cat\boxtimes$ as {\em oriented functors}, and to morphims in $\boxtimes\Cat\boxtimes$ as {\em bioriented functors}.
\end{definition}

\begin{notation}

Let $$\Cat\boxtimes,\qquad \boxtimes\Cat, \qquad \boxtimes\Cat\boxtimes$$
denote the categories of oriented categories, antioriented categories, and bioriented categories, respectively.
\end{notation}

\begin{remark}
An oriented, antioriented, or bioriented category is {\em presentable} if the respective left, right, or bienriched category is presentable in the sense of enriched $\infty$-category theory, which is likewise a presentable category endowed with a closed left action, closed right action, or closed biaction of the enriching monoidal category or categories, respectively.
\end{remark}

\begin{notation}\emph{}
\begin{enumerate}[\normalfont(1)]\setlength{\itemsep}{-2pt}
\item Let $\mC,\mD \in \boxtimes\Cat$. Let ${\boxtimes\Fun}(\mC,\mD)$ be the category of antioriented functors $\mC \to \mD.$	
		
\item Let $\mC,\mD \in \Cat\boxtimes$. Let ${\Fun\boxtimes}(\mC,\mD)$ be the category of oriented functors $\mC \to \mD.$	
		
\item Let $\mC,\mD \in \boxtimes\Cat\boxtimes $. Let ${\boxtimes\Fun\boxtimes}(\mC,\mD)$ be the category of bioriented functors $\mC \to \mD.$		

\end{enumerate}
\end{notation}

\begin{remark}
We refer to adjunctions of oriented, antioriented, and bioriented categories as oriented, antioriented, or bioriented adjunctions, or simply adjunctions if the orientation is clear from context.
\end{remark}

\begin{remark}\label{adj2}
\begin{enumerate}[\normalfont(1)]\setlength{\itemsep}{-2pt}
\item An antioriented (oriented) functor $\mC \to \mD$ admits a right adjoint if and only if it preserves left (right) tensors and the underlying functor admits a right adjoint.

\item A bioriented functor $\mC \to \mD$ admits a right adjoint if and only if it preserves left and right tensors and the underlying functor admits a right adjoint.

\item An antioriented (oriented) functor $\mC \to \mD$ admits a left adjoint if and only if it preserves left (right) cotensors and the underlying functor admits a left adjoint.
	
\item A bioriented functor $\mC \to \mD$ admits a left adjoint if and only if it preserves left and right cotensors and the underlying functor admits a left adjoint.

\end{enumerate}

\end{remark}

Next we define the appropriate notions of opposite oriented category.

\begin{definition}
    
Let
\begin{align*}
&(-)^\circ: {\Cat\boxtimes} \simeq {\boxtimes\Cat}\\
&(-)^\circ: {\boxtimes\Cat} \simeq {\Cat\boxtimes}\\
&(-)^\circ:  {\boxtimes\Cat\boxtimes} \simeq {\boxtimes\Cat\boxtimes}
\end{align*}
be the opposite enriched category involutions.	
	
\end{definition}

\begin{definition}
We define the following involutions, which reverse the even dimensional cells.
Note that we are also forced to reverse the orientation.
The equivalences of \cref{dua} give rise to the equivalences
\begin{align*}
&(-)^\co:= (-)^\op_!: {\boxtimes\Cat} \simeq {\Cat\boxtimes}\\
&(-)^\co:= (-)^\op_!: {\Cat\boxtimes} \simeq {\boxtimes\Cat}\\
&(-)^\co:= ((-)^\op, (-)^\op)_!: {\boxtimes\Cat\boxtimes} \simeq {\boxtimes\Cat\boxtimes}.
\end{align*}
\end{definition}

\begin{definition}
We define the following involutions, which reverse the odd dimensional cells.
Note that we are also forced to keep the orientation.
\begin{align*}
&(-)^\op:= (-)^\circ\circ (-)^\co_!: {\boxtimes\Cat} \simeq {\boxtimes\Cat}\\
&(-)^\op:= (-)^\circ\circ (-)^\co_!: {\Cat\boxtimes} \simeq {\Cat\boxtimes}\\
&(-)^\op :=(-)^\circ \circ ((-)^\co, (-)^\co)_!: {\boxtimes\Cat\boxtimes} \simeq {\boxtimes\Cat\boxtimes}
\end{align*}
\end{definition}

\begin{definition}
Combining the latter two types of involutions gives rise to the following sort of involution, which reverses all cells:
\begin{align*}
&{(-)^{\co\op}}:= {(-)^{\co}} \circ {(-)^{\op}} \simeq {(-)^{\op}} \circ{(-)^{\co}}:{\boxtimes\Cat} \simeq {\Cat\boxtimes}\\
&{(-)^{\co\op}}:= {(-)^{\co}} \circ {(-)^{\op}} \simeq {(-)^{\op}} \circ {(-)^{\co}}:{\Cat\boxtimes} \simeq {\boxtimes\Cat}\\
&{(-)^{\co\op}}:= {(-)^{\co}} \circ {(-)^{\op}} \simeq {(-)^{\op}} \circ {(-)^{\co}}: {\boxtimes\Cat\boxtimes} \simeq {\boxtimes\Cat\boxtimes},
\end{align*}
where all equivalences are involutions.
\end{definition}

Next we define oriented pushouts and oriented pullbacks in oriented categories, which were also studied in \cite[3.8.]{oriented}, \cite[3.3.]{heine2025categorification}, \cite[4.3.]{heine2026stable}.

\begin{definition}
Let $\mC$ be an oriented category.
\begin{enumerate}[\normalfont(1)]\setlength{\itemsep}{-2pt}
\item 
The oriented pullback of the diagram $X \to Z\leftarrow Y$ in $\mC$ is a diagram
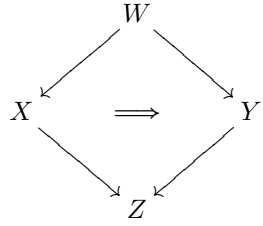
\[
\xymatrix{& W\ar[rd]\ar[ld] &\\
X\ar[rd] & \Longrightarrow & Y\ar[ld]\\
& Z &}
\]
in $\mC$ such that for all objects $T$ of $\mC$ the induced functor
\[
\R\Mor_{\mC}(T,W)\to \R\Mor_{\mC}(T,X) \underset{\R\Mor_{\mC}(T,Z)}{\times}
\Fun^\oplax(\bD^1,\R\Mor_{\mC}(T,Z)) \underset{\R\Mor_{\mC}(T,Z)}{\times} \R\Mor_{\mC}(T,Y)
\]
is an equivalence in $\infty\Cat$.
In this case, we will write $X\underset{Z}{\vec{\times}} Y$ or $ Y\underset{Z}{{\cev\times}} X$ for $W.$

\item 
The oriented pushout of the oriented diagram $X \leftarrow Z\to Y$ in $\mC$ is the oriented pullback of the corresponding diagram in the oriented category $\mC^\op,$ which we denote by $ X\underset{Z}{{\vec{+}}} Y $ or $Y\underset{Z}{\cev{+}} X$.

\end{enumerate}

\end{definition}

\begin{remark}
An oriented pullback square in an oriented category $\mC$ is a diagram $\cube^2\to\mC$ which satisfies the universal property above.
\end{remark}

\begin{definition}
Let $\mC$ be an antioriented category.
\begin{enumerate}[\normalfont(1)]\setlength{\itemsep}{-2pt}
\item 
The antioriented pullback of the diagram $X \to Z\leftarrow Y$ in $\mC$ is the oriented pushout of the corresponding diagram in the oriented category $\mC^\circ$, which we denote by $X\underset{Z}{\bar{\vec{\times}}} Y$ or $ Y\underset{Z}{{\bar{\cev\times}}} X$.

\item 
The antioriented pushout of the diagram $X \leftarrow Z\to Y$ in $\mC$
is the oriented pullback of the corresponding diagram in the oriented category $\mC^\circ$, which we denote by $ X\underset{Z}{{\bar{\vec{+}}}} Y $ or $Y\underset{Z}{\bar{\cev{+}}} X.$

\end{enumerate}

\end{definition}

\begin{definition}Let $\mC$ be a bioriented category.
\begin{enumerate}[\normalfont(1)]\setlength{\itemsep}{-2pt}

\item The (anti)oriented pullback of a diagram $A \to C\leftarrow B$ in $\mC$ is the (anti)oriented pullback in the underlying (anti)oriented category of $\mC$.

\item The (anti)oriented pushout of a diagram $A \leftarrow C \to B$ in $\mC$ is the (anti)oriented pushout in the underlying (anti)oriented category of $\mC$.

\end{enumerate}

\end{definition}

\begin{remark}
Note that the oriented pullback is not symmetric since it makes use of a specified ordering of the sources of the maps $A\to C$ and $B\to C$.
\end{remark}

The following is \cite[Lemma 3.8.6.]{oriented}:

\begin{lemma}\label{0desc}\label{adesc}
Let $\mC$ be an oriented category.
\begin{enumerate}[\normalfont(1)]\setlength{\itemsep}{-2pt}
\item Let $A \leftarrow C \to \B$ be morphisms in $\mC$. If $\mC$ admits pushouts and right tensors with $\bD^1$, 
there is a canonical equivalence $$\A\,\underset{C}{\vec{+}}\, \B \simeq \A\!\!\underset{C \otimes \{0\}}{+} \!\!(\C \ot \bD^1) \!\!\underset{C \otimes \{1\}}{+}\!\!\B.$$	

\item Let $A \to C\leftarrow B$ be morphisms in $\mC$. If $\mC$ admits pullbacks and right cotensors with $\bD^1$, 
there is a canonical equivalence $${\A \,\underset{\C}{\vec{\times}}}\, \B \simeq \A \!\!\underset{\C^{\{0\}}}{\times} \!\!{\C^{\bD^1}}\!\!\underset{{\C^{\{1\}}}}{\times}  \!\!\B.$$	
\end{enumerate}

\end{lemma}

Dually, we obtain the following:

\begin{corollary}\label{bdesc}
Let $\mC$ be an antioriented category.
\begin{enumerate}[\normalfont(1)]\setlength{\itemsep}{-2pt}
\item 
Let $A \leftarrow C \to \B$ be morphisms in $\mC$. If $\mC$ admits pushouts and left tensors with $\bD^1$, 
there is a canonical equivalence $$\A\,\underset{C}{\bar{\vec{+}}}\, \B \simeq \A \underset{\{0\}\otimes\C}{+} ( \bD^1\ot \C) \underset{\{1\}\otimes\C}{+} \B.$$	

\item Let $A \to C\leftarrow B$ be morphisms in $\mC$. If $\mC$ admits pullbacks and left cotensors with $\bD^1$, 
there is a canonical equivalence $$\A\underset{C}{\bar{\vec{\times}}} \B \simeq \A \underset{^{\{0\}}\C}{\times}{^{\bD^1}\C} \underset{^{\{1\}}\C}{\times} \B.$$	
\end{enumerate}

\end{corollary}

We have the following pasting law, which is \cite[Lemma 3.8.11.]{oriented}:

\begin{lemma}\label{pasting}

Consider the following diagram in any oriented category $\mC$, where the left hand square is a commutative square:
\[
\begin{tikzcd}
\Q \ar{d} \ar{r} & \P \ar{r}{} \ar{d}[swap]{} & \B  \ar{d}{} \\
\E \ar{r} & \A \ar[double]{ur}{}  \ar{r}[swap]{} & \C
\end{tikzcd}
\]
If the right hand square is an oriented pullback square, the left hand square is a pullback square if and only if the outer square is an oriented pullback square.

\end{lemma}

The following is \cite[Proposition 3.8.12.]{oriented}:

\begin{proposition}\label{homs} 
Let $\F: \mA \to \mC,\G: \mB \to \mC$ be functors and $\A,\A'\in \mA, \B,\B' \in \mB$ and $\sigma: \F(\A)\to \G(\B), \sigma': \F(\A')\to \G(\B')$ morphisms. There is a canonical equivalence $$\Mor_{{\mA \,\underset{\mC}{\vec{\times}}}\, \mB}((\A,\B, \sigma),(\A',\B', \sigma')) \simeq {\Mor_\mB(\B,\B') \,\underset{\Mor_\mC(\F(\A),\G(\B'))}{\vec{\times}}}\, \Mor_{\mA}(\A,\A').$$

\end{proposition}

The following is \cite[Corollary 3.8.13.]{oriented}:

\begin{corollary}\label{homso}

Let $\F: \mA \to \mC,\G: \mB \to \mC$ be functors and $\A,\A'\in \mA, \B,\B' \in \mB$ and $\sigma: \F(\A)\to \G(\B), \sigma': \F(\A')\to \G(\B')$ morphisms. There is a canonical equivalence $$\Mor_{\mA\underset{\mC}{\bar{\vec{\times}}} \mB}((\A,\B, \sigma),(\A',\B', \sigma')) \simeq \Mor_\mA(\A,\A')\underset{\Mor_\mC(\F(\A),\G(\B'))}{\bar{\vec{\times}}} \Mor_\mB(\B,\B').$$

\end{corollary}

\begin{corollary}\label{homs2}\label{homso2} Let $\mB$ be an $\infty$-category and $\sigma: \A \to \B, \sigma': \A' \to \B'$ morphisms in $\mB$. There are canonical equivalences $$\Mor_{\Fun^\lax(\bD^1,\mB)}(\sigma, \sigma') \simeq \Mor_\mA(\A,\A')\underset{\Mor_\mC(\A,\B')}{\bar{\vec{\times}}} \Mor_\mB(\B,\B'), $$
$$\Mor_{\Fun^\oplax(\bD^1,\mB)}(\sigma, \sigma') \simeq {\Mor_\mB(\B,\B') \,\underset{\Mor_\mC(\A,\B')}{\vec{\times}}}\, \Mor_{\mA}(\A,\A'). $$\end{corollary}

\begin{corollary}\label{dimensio}

Let $ n \geq 0$ and $\mA,\mB,\mC$ be $n$-categories.
Let $\F: \mA \to \mC,\G: \mB \to \mC$ be functors.
The oriented pullback $\mA\underset{\mC}{\bar{\vec{\times}}} \mB$
is an $n$-category.
    
\end{corollary}

\begin{proof}

We proceed by induction on $n \geq 0.$
For $n=0$ the statement follows from \cref{adesc}:
the oriented pullback $\mA\underset{\mC}{\bar{\vec{\times}}} \mB$ identifies with the pullback $\mA \times_\mC \Fun^\oplax(\bD^1,\mC) \times_\mC \mB$.
Hence it suffices to see that $\Fun^\oplax(\bD^1,\mC)$ is a space if
$\mC$ is a space. But if $\mC$ is a space, the diagonal functor
$\mC \to \Fun^\oplax(\bD^1,\mC)$ is an equivalence since the classifying space of any oriented cube is contractible.

We prove the induction step. Let $ n \geq 0.$
If the statement holds for $n$, it also holds for $n+1$ by \cref{homso}.
\end{proof}

% \begin{theorem}\label{thm:oplaxinclusion}
% Let $\mA$ and $ \mC$ be $\infty$-categories.
% The canonical functors $$\Fun(\mA,\mC) \to \Fun^\lax(\mA,\mC)$$
% and
% $$ \Fun(\mA,\mC) \to \Fun^\oplax(\mA,\mC)$$ are inclusions.
% \end{theorem}

The following is \cite[Theorem 4.2.7.]{oriented}.

\begin{theorem}\label{interchange}
\begin{enumerate}[\normalfont(1)]\setlength{\itemsep}{-2pt}
\item The induced functor $$ \infty\Cat\simeq\, _{\infty\Cat}\Cat \to {\boxtimes\Cat} $$ of 2-categories is fully faithful. The essential image precisely consists of the antioriented categories $\mC$ such that the antioriented interchange law holds strictly:
for every $X,Y,Z \in \mC$ and functors $\bD^n \to \L\Mor_\mC(Y,Z)$ and $\bD^m \to \L\Mor_\mC(X,Y)$
the functor $$\bD^n \boxtimes \bD^m \to \L\Mor_\mC(Y,Z) \boxtimes \L\Mor_\mC(X,Y) \to \L\Mor_\mC(X,Z)$$ factors through the functor
$\bD^n \boxtimes \bD^m \to \bD^n \times \bD^m.$ 
\item The induced functor $$\infty\Cat\simeq\Cat_{\infty\Cat} \to {\Cat\boxtimes} $$ of 2-categoriesis fully faithful. The essential image precisely consists of the oriented categories $\mC$ such that the oriented interchange law holds strictly:
for every $X,Y,Z \in \mC$ and functors $\bD^n \to \R\Mor_\mC(Y,Z)$ and $\bD^m \to \R\Mor_\mC(X,Y)$
the functor $$\bD^n \boxtimes \bD^m \to \R\Mor_\mC(Y,Z) \boxtimes \R\Mor_\mC(X,Y) \to \R\Mor_\mC(X,Z)$$ factors through the functor
$\bD^n \boxtimes \bD^m \to \bD^n \times \bD^m.$ 

%\item The induced functor $$_{\infty\Cat}\Cat_{\infty\Cat} \to {\boxtimes \Cat\boxtimes} $$ of 2-categories is fully faithful.
\end{enumerate}   
\end{theorem}

% \begin{theorem}\label{presheaves}
% Let $\mD$ be an $\infty$-category.
% The canonical functors 
% $$\Fun(\mD,\infty\scat) \to {\boxtimes\Fun}(\mD,\infty\fcat) $$
% and 
% $$\Fun(\mD,\infty\scat) \to {\Fun\boxtimes}(\mD,\infty\fcat)$$
% induce fully faithful functors on underlying 1-categories.
% \end{theorem}

\newpage

\section{\mbox{Fibrations of $\infty$-categories}}

\subsection{Enriched fibrations}

\begin{definition}

Let $\mV$ be a monoidal category
and $\phi: \mC \to \mD$ a $\mV$-enriched functor.
A morphism $X \to Y$ in $\mC$ is $\phi$-cartesian if for every $Y\in \mC$ the following commutative square in $\mV$ is a pullback square:
$$\begin{xy}
\xymatrix{
\Mor_\mC(Z,X) \ar[d]^{} \ar[r]
& \Mor_\mC(Z,Y) \ar[d]
\\ 
\Mor_\mD(\phi(Z),\phi(X)) \ar[r] & \Mor_\mD(\phi(Z),\phi(Y)).}
\end{xy}$$
   
\end{definition}

\begin{definition}\label{enrfibr}

Let $\mV$ be a monoidal category.

\begin{enumerate}[\normalfont(1)]\setlength{\itemsep}{-2pt}
\item A $\mV$-enriched functor $\phi: \mC \to \mD$ is a cartesian fibration if for every $Y \in \mC$ and morphism $\alpha: Z \to \phi(Y)$ in $\mD$ there is a $\phi$-cartesian morphism $X \to Y$ in $\mC$ lying over $\alpha.$

\item A $\mV$-enriched functor $\phi: \mC \to \mD$ is a cocartesian fibration if the opposite $\mV$-enriched functor $\phi^\circ: \mC^\circ \to \mD^\circ $ is a cartesian fibration.

\end{enumerate}

\end{definition}

\begin{proposition}\label{enrtargetfibr}
Let $\mV$ be a presentably monoidal category and $\mC$ a $\mV$-enriched category.
\begin{enumerate}[\normalfont(1)]\setlength{\itemsep}{-2pt}
\item Evaluation at the target $\mC^{\bD^1} \to \mC$ is a cocartesian fibration of $\mV$-enriched categories whose cocartesian morphisms are inverted by evaluation at the source.

\item If $\mC$ admits pullbacks, evaluation at the target $\mC^{\bD^1} \to \mC$ is a cartesian fibration of $\mV$-enriched categories.
\end{enumerate}
    
\end{proposition}

\begin{proof}

(1): Let $\alpha: X \to Y, \beta: Y \to Z$ be morphisms in $\mC.$
The latter determine a morphism $\sigma$ in $\mC^{\bD^1}$
from the morphism $\alpha: X \to Y$ to the morphism $\beta \circ \alpha X \to Y \to Z$
that is sent to the identity by evaluation at the source and is sent to $\beta$ by evaluation at the target.
We prove that $\sigma$ is a cocartesian lift of $\beta.$
We need to see that for every morphism $\gamma: A \to B$ in $\mC$
the following commutative square is a pullback square:
$$\begin{xy}
\xymatrix{
\Mor_{\mC^{\bD^1}}(\beta \circ \alpha,\gamma) \ar[d]^{} \ar[r]
& \Mor_{\mC^{\bD^1}}(\alpha,\gamma) \ar[d] 
\\ 
\Mor_{\mC}(Z,B) \ar[r] & \Mor_{\mC}(Y,B).}
\end{xy}$$
The latter identifies with the pullback square:
$$\begin{xy}
\xymatrix{
\Mor_{\mC}(X,A) \times_{\Mor_{\mC}(X,B)} \Mor_{\mC}(Z,B) \ar[d]^{} \ar[r]
& \Mor_{\mC}(X,A) \times_{\Mor_{\mC}(X,B)} \Mor_{\mC}(Y,B) \ar[d] 
\\ 
\Mor_{\mC}(Z,B) \ar[r] & \Mor_{\mC}(Y,B).}
\end{xy}$$

(2): Let $\alpha: X \to Y, \beta: Z \to Y $ be morphisms in $\mC.$
The pullback square for $\alpha, \beta$ determines a morphism $\kappa$ in $\mC^{\bD^1}$
from the projection $\alpha': X \times_Y Z \to Z$ to $ \alpha$
that is sent to $\beta$ by evaluation at the target.
We prove that $\kappa$ is a cartesian lift of $\beta.$
We need to see that for every morphism $\gamma: A \to B$ in $\mC$
the following commutative square is a pullback square:
$$\begin{xy}
\xymatrix{
\Mor_{\mC^{\bD^1}}(\gamma, \alpha') \ar[d]^{} \ar[r]
& \Mor_{\mC^{\bD^1}}(\gamma, \alpha) \ar[d] 
\\ 
\Mor_{\mC}(B,Z) \ar[r] & \Mor_{\mC}(B,Y).}
\end{xy}$$
The latter identifies with the pullback square:
$$\begin{xy}
\xymatrix{
\Mor_{\mC}(A,X \times_Y Z) \times_{\Mor_{\mC}(A,Z)} \Mor_{\mC}(B,Z) \ar[d]^{} \ar[r]
& \Mor_{\mC}(A,X) \times_{\Mor_{\mC}(A,Y)} \Mor_{\mC}(B,Y) \ar[d] 
\\ 
\Mor_{\mC}(B,Z) \ar[r] & \Mor_{\mC}(B,Y).}
\end{xy}$$
\end{proof}

\begin{proposition}\label{enrcharo}
Let $\mV$ be a presentably monoidal category, $\phi: \mC \to \mD$
a $\mV$-enriched functor and $f: X \to Y$ a morphism in $\mC.$
The following are equivalent:

\begin{enumerate}[\normalfont(1)]\setlength{\itemsep}{-2pt}
\item The morphism $f$ is $\phi$-cocartesian.

\item For every morphism $g: A \to B$ in $\mC$ the following induced morphism in $\mV$ is an equivalence:
$$\Mor_{\mC^{\bD^1}}(f,g) \to \Mor_\mC(X,A) \times_{\Mor_\mD(\phi(X),\phi(A))} \Mor_{\mD^{\bD^1}}(\phi(f),\phi(g)).$$ 

\end{enumerate}
    
\end{proposition}

\begin{proof}

The morphism of (2) identifies with the following morphism:
$$\Mor_{\mC^{\bD^1}}(f,g) \simeq \Mor_\mC(X,A) \times_{\Mor_\mC(X,B)} \Mor_\mC(Y, B) \to $$$$ \Mor_\mC(X,A) \times_{\Mor_\mD(\phi(X),\phi(A))} \Mor_{\mD^{\bD^1}}(\phi(f),\phi(g)) $$$$ \simeq \Mor_\mC(X,A) \times_{\Mor_\mD(\phi(X),\phi(A))} \Mor_\mD(\phi(X),\phi(A)) \times_{\Mor_\mD(\phi(X),\phi(B))} \Mor_\mD(\phi(Y),\phi(B))$$$$ \simeq \Mor_\mC(X,A) \times_{\Mor_\mD(\phi(X),\phi(B))} \Mor_\mD(\phi(Y),\phi(B)).$$
The latter is the pullback along the morphism 
$ \Mor_\mC(X,A) \to \Mor_\mC(X,B) $ of the induced morphism
$$\Mor_\mC(Y, B) \to \Mor_\mC(X,B) \times_{\Mor_\mD(\phi(X),\phi(B))} \Mor_\mD(\phi(Y),\phi(B)).$$
\end{proof}

\begin{notation}

Let $\mV$ be a presentably monoidal category and $\phi: \mC \to \mD$
a $\mV$-enriched functor.

Let $$\mC^{\bD^1, \cocart} \subset \mC^{\bD^1}$$ be the full $\mV$-enriched subcategory of $\phi$-cocartesian morphisms.
    
\end{notation}

\begin{corollary}

Let $\mV$ be a presentably monoidal category and $\phi: \mC \to \mD$
a $\mV$-enriched functor.

The induced $\mV$-enriched functor $\mC^{\bD^1, \cocart} \to \mC \times_{\mD^{\{0\}}} \mD^{\bD^1}$ is an embedding.
    
\end{corollary}

\begin{corollary}\label{enrfibchar}

Let $\mV$ be a presentably monoidal category and $\phi: \mC \to \mD$
a $\mV$-enriched functor.

The following are equivalent:

\begin{enumerate}[\normalfont(1)]\setlength{\itemsep}{-2pt}
\item The $\mV$-enriched functor $\phi$ is a cocartesian fibration.

\item The induced $\mV$-enriched functor $\mC^{\bD^1} \to \mC \times_{\mD^{\{0\}}} \mD^{\bD^1}$ admits a fully faithful $\mV$-enriched left adjoint whose colocal objects are the $\phi$-cocartesian morphisms.

\item The induced $\mV$-enriched functor $\mC^{\bD^1, \cocart} \to \mC \times_{\mD^{\{0\}}} \mD^{\bD^1}$ is an equivalence.

\end{enumerate}
    
\end{corollary}

\begin{corollary}\label{enrcocartexp}
Let $\mV, \mW$ be presentably monoidal categories and $\phi: \mC \to \mD$ a $\mV,\mW$-enriched functor whose underlying $\mV$-enriched functor is a cocartesian fibration.
For every $\mV$-enriched category $\mB$ the induced right $\mW$-enriched functor
$${\mV\mathrm{-}\Fun}(\mB, \mC) \to {\mV\mathrm{-}\Fun}(\mB, \mD)$$ is a cocartesian fibration whose cocartesian morphisms are objectwise.
   
\end{corollary}

\begin{proposition}\label{enevolt}

Let $\mV$ be a presentably monoidal category and $\phi: \mC \to \mD$
a $\mV$-enriched functor.

\begin{enumerate}[\normalfont(1)]\setlength{\itemsep}{-2pt}
\item Evaluation at the target $\mC \times_{\mD^{\{0\}}} \mD^{\bD^1} \to \mD^{\{1\}}$ is a cocartesian fibration of $\mV$-enriched categories.

\item The $\mV$-enriched functor
$\mC \simeq \mC \times_{\mD} \mD \to \mC \times_{\mD^{\{0\}}} \mD^{\bD^1}$ induced by the diagonal functor is a $\mV$-enriched embedding and induces for every $\mV$-enriched cocartesian fibration $\mE \to \mD$ an equivalence
$$ \Fun^\cocart_\mD(\mC \times_{\mD^{\{0\}}} \mD^{\bD^1},\mE) \to \Fun_\mD(\mC, \mE).$$

\end{enumerate}

\end{proposition}

\begin{proof}

(1): By \cref{enrtargetfibr} evaluation at the target $\mD^{\bD^1} \to \mD^{\{1\}}$ is a cocartesian fibration of $\mV$-enriched categories whose cocartesian morphisms are inverted by evaluation at the source. Thus evaluation at the target $\mC \times_{\mD^{\{0\}}} \mD^{\bD^1} \to \mD^{\{1\}}$ is a cocartesian fibration of $\mV$-enriched categories.

(2): The diagonal functor $\mD \to \mD^{\bD^1} $
induces on morphism objects between $X,Y \in \mD$ the canonical equivalence
$\Mor_\mD(X,Y) \to \Mor_\mD(X,Y) \times_{\Mor_\mD(X,Y)} \Mor_\mD(X,Y) $ and so is an embedding.
The latter admits a $\mV$-enriched right adjoint that sends any morphism
$\alpha: A \to B$ to $A$ since for every $X \in \mD$ the canonical morphism
$ \Mor_{\mD^{\bD^1}}(\id_X,\beta) \to\Mor_\mD(X,A)  $
identifies with the equivalence $\Mor_\mD(X,B) \times_{\Mor_\mD(X,B)} \Mor_\mD(X,A) \to \Mor_\mD(X,A)  $.
Thus the $\mV$-enriched functor
$\bj: \mC \simeq \mC \times_{\mD} \mD \to \mC \times_{\mD^{\{0\}}} \mD^{\bD^1}$ is an embedding and admits a $\mV$-enriched right adjoint $R$ that is the projection to $\mC.$
The latter $\mV$-enriched adjunction whose left adjoint is an embedding, gives rise to an adjunction 
$ {\mV\mathrm{-}\Fun}_\mD(\mC,\mE) \leftrightarrows {\mV\mathrm{-}\Fun}_\mD(\mC \times_{\mD^{\{0\}}} \mD^{\bD^1},\mE) : \bj^* $ whose left adjoint is an embedding
that sends a $\mV$-enriched functor $\kappa: \mC \to \mE$ to
the target of the cocartesian lift $\kappa \circ R \to \kappa' $
of the canonical $\mV$-enriched natural transformation from 
$\mC \times_{\mD^{\{0\}}} \mD^{\bD^1} \to \mD^{\{0\}}$ to
$\mC \times_{\mD^{\{0\}}} \mD^{\bD^1} \to \mD^{\{1\}}$.
Since the projection 
$R: \mC \times_{\mD^{\{0\}}} \mD^{\bD^1} \to \mC$ inverts morphisms cocartesian for evaluation at the target, the adjunction
$ {\mV\mathrm{-}\Fun}_\mD(\mC,\mE) \leftrightarrows {\mV\mathrm{-}\Fun}_\mD(\mC \times_{\mD^{\{0\}}} \mD^{\bD^1},\mE) : \bj^* $ restricts to an adjunction
$ {\mV\mathrm{-}\Fun}_\mD(\mC,\mE) \leftrightarrows {\mV\mathrm{-}\Fun}^\cocart_\mD(\mC \times_{\mD^{\{0\}}} \mD^{\bD^1},\mE) : \bj^*. $
So it suffices to see that the right adjoint is conservative.

For every $Z \in \mC \times_{\mD} \mD \to \mC \times_{\mD^{\{0\}}} $ 
let $\epsilon_Z: \bj(R(Z)) \to Z $ be the counit.
A morphism in $ \mC \times_{\mD^{\{0\}}} \mD^{\bD^1}$ is cocartesian for evaluation at the target if and only if it is inverted by the projection to $\mC$, i.e. by $R$.
By the triangle identities the counit is inverted by $R$ and so cocartesian for evaluation at the target.
Hence any map $ \mC \times_{\mD^{\{0\}}} \mD^{\bD^1} \to \mE$
of cocartesian fibrations over $\mD$ sends the counit
$\epsilon_Z$ for $Z \in \mC \times_{\mD} \mD^{\bD^1} $ to a cocartesian morphism,
and so is an equivalence if its restriction to $\mC$ is an equivalence.
\end{proof}

\begin{proposition}
Let $\mV$ be a presentably monoidal category.
A $\mV$-enriched functor $\phi: \mC \to \mD$ is a cocartesian fibration if and only if the $\mV$-enriched embedding
$\mC \subset \mC \times_{\mD^{\{0\}}} \mD^{\bD^1}$ admits a $\mV$-enriched left adjoint whose local equivalences lie over equivalences of $\mD.$
    
\end{proposition}

\begin{proof}
By \cref{enevolt} the $\mV$-enriched functor $\mC \times_{\mD^{\{0\}}} \mD^{\bD^1} \to \mD^{\{1\}}$ is a cocartesian fibration and so also any $\mV$-enriched localization of it whose local equivalences lie over equivalences of $\mD.$
So it suffices to see the only-if direction.
Assume that $\phi: \mC \to \mD$ is a $\mV$-enriched cocartesian fibration. An object of $\mC \times_{\mD^{\{0\}}} \mD^{\bD^1}$ is an object of $X \in \mC$ and a morphism $\alpha: \phi(X) \to Y$ in $\mD$.
Since $\phi$ is a cocartesian fibration, there is a $\phi$-cocartesian lift $X \to X'$ in $\mC$ of the morphism $\phi(X) \to Y.$
The latter determines a morphism in $ \mC \times_{\mD^{\{0\}}} \mD^{\bD^1} $ from $(X, \alpha)$ to
$(X', \id_Y)$, where the latter is the image of $X'\in \mC$ in  
$\mC \times_{\mD^{\{0\}}} \mD^{\bD^1}$.
This morphism from $(X, \alpha)$ to
$(X', \id_Y)$ is sent by the functor $ \mC \times_{\mD^{\{0\}}} \mD^{\bD^1} \to \mD^{\{1\}}$ to the identity of $Y.$

For every $ Z\in \mC $ the induced morphism
$$ \Mor_\mC(X',Z) \simeq \Mor_{\mC \times_{\mD^{\{0\}}} \mD^{\bD^1}}((X', \id_Y),(Z,\id_{\phi(Z)})) \to \Mor_{\mC \times_{\mD^{\{0\}}} \mD^{\bD^1}}((X, \alpha),(Z,\id_{\phi(Z)})) $$
identifies with the induced morphism
$$ \Mor_\mC(X',Z) \to \Mor_{\mC}(X,Z) \times_{\Mor_{\mD}(\phi(X),\phi(Z))} \Mor_{\mD}(\phi(X),\phi(Z)) \times_{\Mor_{\mD}(\phi(X),\phi(Z))} \Mor_{\mD}(Y,\phi(Z)) $$$$ \simeq 
\Mor_{\mC}(X,Z) \times_{\Mor_{\mD}(\phi(X),\phi(Z))} \Mor_{\mD}(Y,\phi(Z)).$$
The latter is an equivalence since the morphism $X \to X'$ is $\phi$-cocartesian.
This proves that the $\mV$-enriched embedding
$\mC \to \mC \times_{\mD^{\{0\}}} \mD^{\bD^1}$ admits a $\mV$-enriched left adjoint whose local equivalences lie over equivalences of $\mD.$
\end{proof}

\begin{proposition}\label{fibcocart}
Let $\mV$ be a monoidal category, $\phi: \mB \to \mD, \psi: \mC \to \mD$ $\mV$-enriched cocartesian fibrations and 
$\rho: \mB \to \mC$ a map of $\mV$-enriched cocartesian fibrations.
Then $\rho$ is a $\mV$-enriched cocartesian fibration if and only if it induces on the fiber over every object of $\mD$ a cocartesian fibration and for every morphism $ X \to Y$ of $\mD$ the fiber transport $\mB_X \to \mB_Y $ of $\phi$ sends $\rho_X$-cocartesian morphisms to $\rho_Y$-cocartesian morphisms.
    
\end{proposition}

\begin{proof}

The only-if direction follows from the fact that cocartesian fibrations are stable under pulback and cocartesian morphisms are closed under composition.
We prove the other direction. Since $\phi, \psi$ are cocartesian fibrations and $\rho $ is a map of cocartesian fibrations that induces fiberwise cocartesian fibrations, we can find for every $X \in \mB$ and morphism $\rho(X) \to Y$ a lift $X \to Y'$ in $\mB$ that factors as a 
$\rho$-cocartesian morphism followed by a $\rho_Y$-cocartesian morphism.
Since cocartesian morphisms are closed under composition, it suffices to see that any $\rho_Y$-cocartesian morphisms $A \to B$ is $\rho$-cocartesian. For that we have to see that for any $Z \in \mB$
the following commutative square is a pullback square:
$$\begin{xy}
\xymatrix{
\Mor_{\mB}(B, Z) \ar[d]^{} \ar[r]
& \Mor_{\mB}(A, Z) \ar[d] 
\\ 
\Mor_{\mC}(\rho(B),\rho(Z)) \ar[r] & \Mor_{\mC}(\rho(A),\rho(Z)).}
\end{xy}$$
The latter is a commutative square over $\Mor_\mD(Y,\phi(Z))$
that induces on the fiber over every $\kappa: Y \to \phi(Z)$
the commutative square
$$\begin{xy}
\xymatrix{
\Mor_{\mB_{\phi(Z)}}(\kappa_!(B), Z) \ar[d]^{} \ar[r]
& \Mor_{\mB_{\phi(Z)}}(\kappa_!(A), Z) \ar[d] 
\\ 
\Mor_{\mC_{\phi(Z)}}(\kappa_!(\rho(B)), \rho(Z)) \ar[r] & \Mor_{\mC_{\phi(Z)}}(\kappa_!(\rho(A)), \rho(Z)).}
\end{xy}$$
The latter is a pullback square because the morphism $\kappa_!(A) \to \kappa_!(B) $ is $\rho_{\phi(Z)}$-cocartesian by assumption.
\end{proof}

\subsection{Cartesian cells}

\begin{notation}Let $\phi: \mC \to \mD$ be a functor and $X,Y \in \mC.$ 
We write
\begin{align*}
\phi_{X,Y}:\Mor_\mC(X,Y)\to\Mor_\mD(\phi(X),\phi(Y))
\end{align*}
for the induced functor on morphism $\infty$-categories.
\end{notation}

\begin{definition}Let $1 \leq \n \leq \infty$ and $\phi: \mC \to \mD$ a functor of $\infty$-categories.

\begin{enumerate}[\normalfont(1)]\setlength{\itemsep}{-2pt}

\item A 1-morphism $f: X \to Y$ in $\mC$ is $\phi$-cocartesian if for every $Z\in \mC$ the commutative square 
\begin{equation}\label{filler1.1}
\begin{xy}
\xymatrix{
\Mor_\mC(Y,Z) \ar[d] \ar[r]
& \Mor_\mC(X,Z) \ar[d]^\phi
\\ 
\Mor_\mD(\phi(Y), \phi(Z)) \ar[r] & \Mor_\mD(\phi(X), \phi(Z))
}
\end{xy}\end{equation}
is a pullback square.

\item An $n$-morphism $\alpha:\bD^n\to\mC$ is $\phi$-cocartesian if
for every pair of morphisms $(X \to \alpha(0), \alpha(1) \to Y)$ in $\mC$ the composite $n-1$-morphism
\[
\bD^{n-1}\to\Mor_\mC(\alpha(0),\alpha(1)) \to \Mor_\mC(X,Y)
\]
is $\phi_{X,Y}$-cocartesian.
\end{enumerate}
\end{definition}

\begin{definition}Let $\n \geq 1$ and $\phi: \mC \to \mD$ a functor.
A $n$-morphism in $\mC$ is locally $\phi$-cartesian if and only if
it is locally $\phi^{\co\op}$-cocartesian.
    
\end{definition}

\begin{example}Let $\phi: \mC \to \mD$ be a functor and $n \geq 0$.
Every $n$-morphism in $\mC$ that is an equivalence, is $\phi$-cocartesian.
An $n$-morphism in $\mC$ is an equivalence if and only if it is 
cocartesian for the functor $\mC \to \bD^0$.
    
\end{example}

\begin{remark}\label{elemen}
Let $n \geq 1$ and $\phi: \mC \to \mD$ a functor.
By pasting of pullbacks the composite of two composable $\phi$-cocartesian 1-morphisms is $\phi$-cocartesian.
Therefore for any three parallel $n-1$-morphisms $f,g,h$ in $\mC$
the composite of any two $\phi$-cocartesian $n$-morphisms $\alpha : f \to g, \beta: g \to h$ is again $\phi$-cocartesian.

Moreover the pasting law for pullbacks implies that for two functors
$\phi: \mC \to \mD, \kappa: \mD \to \mE$
an $n$-morphism of $\mC$ lying over a $\kappa$-cocartesian $n$-morphism of $\mD$, is $\phi$-cocartesian if and only if it is $\kappa \circ \phi$-cocartesian.
In particular, an $n$-morphism of $\mC$ that is inverted by $\phi$, is $\phi$-cocartesian if and only if it is an equivalence.

Also note that since forming morphism objects preserves pullbacks,
for any functors $\phi:\mC \to \mD$ and $\mB \to \mD$
an $n$-morphism in the pullback $\mB \times_\mD \mC $ is cocartesian for the projection to $\mB$ if it is $\phi$-cocartesian.
    
\end{remark}

\begin{definition} Let $0 \leq n \leq \infty.$
A functor $\phi: \mC \to \mD$ is an $n$-anticocartesian fibration if for every $1 \leq k \leq n$ 
every commutative square 
\begin{equation}\label{filler2}
\begin{xy}
\xymatrix{
\bD^{k-1} \ar[d] \ar[r]
& \mC \ar[d]^\phi
\\ 
\bD^k \ar[r] & \mD
}
\end{xy}\end{equation}
admits a filler by a $\phi$-cocartesian $k$-morphism,
where the left vertical functor is the source inclusion.

An anticocartesian fibration is an $\infty$-anticocartesian fibration.

\end{definition}

\begin{remark}
It follows immediately from the definition and \cref{elemen} that $n$-anticocartesian fibrations are stable under base change and composition.
\end{remark}

\begin{definition}Let $1 \leq n \leq \infty$ and $\sigma$ an involution of $\infty\Cat.$
A functor $\phi: \mC \to \mD$ is an $\n$-$\sigma$-fibration if $\phi^\sigma$ is an $n$-anticocartesian fibration.
A functor $\phi: \mC \to \mD$ is a $\sigma$-fibration if it is an $\infty$-$\sigma$-fibration, i.e. $\phi^\sigma$ is an anticocartesian fibration.

\end{definition}

\begin{definition}Let $1 \leq n \leq \infty. $

\begin{enumerate}[\normalfont(1)]\setlength{\itemsep}{-2pt}

\item A functor $\phi: \mC \to \mD$ is an $n$-cocartesian fibration if $\phi^\co$ is a $n$-anticocartesian fibration.

\item A functor $\phi: \mC \to \mD$ is an $n$-cartesian fibration if $\phi^{\op}$ is an $n$-anticocartesian fibration.

\item A functor $\phi: \mC \to \mD$ is an $n$-anticartesian fibration if $\phi^{\co\op}$ is an $n$-anticocartesian fibration.

For $n=\infty$ we drop the index $n.$

\end{enumerate}
    
\end{definition}

\begin{remark}
Since the $(-)^\co$ involution fixes the orientation of 1-morphisms, it follows that a functor is a 1-(co)cartesian fibration if and only if it is a 1-anti(co)cartesian fibration.

\end{remark}

\cref{enrfibchar} gives the following:

\begin{corollary}\label{fibchar}

A functor $\phi: \mC \to \mD$ is a 1-cocartesian fibration if and only if the induced functor $$\Fun(\bD^1,\mC) \to \mC \times_{\Fun(\{0\},\mD)} \Fun(\bD^1, \mD) $$ admits a fully faithful left adjoint whose colocal objects are the $\phi$-cocartesian morphisms.
    
\end{corollary}

We have the following inductive definition of anticocartesian fibrations:

\begin{proposition}\label{cocarto}
Let $1 \leq n \leq \infty$ and $\phi: \mC \to \mD$ a functor.
The following are equivalent:

\begin{enumerate}[\normalfont(1)]\setlength{\itemsep}{-2pt}
\item The functor $\phi$ is an $n$-anticocartesian fibration.

\item The functor $\phi$ is a 1-anticocartesian fibration, and for every $\X,\Y \in \mC$ the functor
\[
\phi_{X,Y}: \Mor_\mC(\X,\Y) \to \Mor_\mD(\phi(\X),\phi(\Y))
\]
is a $n-1$-anticocartesian fibration
and for every pair of morphisms $X' \to X$ and $Y \to Y'$ in $\mC$
the induced functor 
$\Mor_\mC(\X,\Y) \to \Mor_\mC(X',Y')$ sends $\phi_{X,Y}$-cocartesian morphisms to $\phi_{X',Y'}$-cocartesian morphisms.
\end{enumerate}

\end{proposition}

\begin{proof}

We procced by induction on $n \geq 1.$ For $n=1$ there is nothing to show. Let $n > 1.$ We assume the statement for $n-1$.

We prove the induction step. For every $1 < k \leq n$ every commutative square 
\begin{equation}\label{filler3}
\begin{xy}
\xymatrix{
\bD^{k-1} \ar[d] \ar[r]
& \mC \ar[d]^\phi
\\ 
\bD^k \ar[r] & \mD
}
\end{xy}\end{equation}
admits a filler by a $\phi$-cocartesian $k$-morphism,
where the left vertical functor is the source inclusion,
if and only if for every $1 < k \leq n$ and $X,Y \in \mC$ every commutative square 
\begin{equation}
\begin{xy}
\xymatrix{
\bD^{k-2} \ar[d] \ar[r]
& \Mor_\mC(X,Y) \ar[d]^{\phi_{X,Y}}
\\ 
\bD^{k-1} \ar[r] & \Mor_\mD(\phi(X),\phi(Y))
}
\end{xy}\end{equation}
admits a filler by a $\phi_{X,Y}$-cocartesian $k-1$-morphism,
where the left vertical functor is the source inclusion,
such that for every morphisms $X' \to X, Y \to Y'$ in $\mC$
the induced functor 
$\Mor_\mC(\X,\Y) \to \Mor_\mC(X',Y')$ sends $\phi_{X,Y}$-cocartesian morphisms to $\phi_{X',Y'}$-cocartesian morphisms.
So by induction (1) is equivalent to (2).
\end{proof}

\begin{definition}Let $n \geq 0$ and $\sigma$ a commutative square
\begin{equation}\label{sqco}
\begin{xy}
\xymatrix{
X \ar[d]^\psi \ar[r]^\kappa
& Y \ar[d]^\phi
\\ 
S \ar[r]^\rho & T
}
\end{xy}\end{equation}

\begin{enumerate}[\normalfont(1)]\setlength{\itemsep}{-2pt}

\item A commutative square \ref{sqco} is a map of $n$-anticocartesian fibrations if $\psi, \phi$ are $n$-anticocartesian fibrations and $\kappa$ sends $\psi$-cocartesian morphisms of dimension smaller or equal $n$ to $\phi$-cocartesian morphisms.

\item A commutative square \ref{sqco}
is a map of $n$-cocartesian fibrations if $\sigma^\co$ is a map of $n$-anticocartesian fibrations.

\item A commutative square \ref{sqco}
is a map of $n$-anticartesian fibrations if $\sigma^\op$ is a map of $n$-anticocartesian fibrations.

\item A commutative square \ref{sqco}
is a map of $n$-cartesian fibrations if $\sigma^\coop$ is a map of $n$-anticocartesian fibrations.

\item A commutative square \ref{sqco}
is a map of (anti)(co)cartesian fibrations if it is a map of
$n$-(anti)(co)cartesian fibrations for every $n \geq 1.$

\item A map of (anti) $n$-(co)cartesian fibrations over $\mD$ is a map of $n$-(anti)(co)cartesian fibrations such that $\rho$ is the identity.

\item A map of (anti)(co)cartesian fibrations over $\mD$ is a map of (anti) (co)cartesian fibrations such that $\rho$ is the identity.
\end{enumerate}

\end{definition}

\begin{notation}\label{notfib1}

Let $$\co\mathcal{C}\mathit{art}, \mathcal{C}\mathit{art}, \overline{\co\mathcal{C}\mathit{art}}, \overline{\mathcal{C}\mathit{art}} \subset \Fun(\bD^1,\infty\scat)$$
be the respective subcategories of cocartesian, cartesian, anticocartesian, anticartesian fibrations and morphisms of such. 

\end{notation}

\begin{notation}\label{notfib2} Let $S$ be an $\infty$-category.
Let $$ \infty\scat^\cocart_{/S}, \infty\scat^\cart_{/S}, \infty\scat^{\overline{\cocart}}_{/S}, \infty\scat^{\overline{\cart}}_{/S} \subset \infty\scat_{/S} $$ be the respective subcategories of cocartesian, cartesian, anticocartesian, anticartesian fibrations and maps of such.

\end{notation}

\begin{remark}\label{seqcol}
It is immediate from the definition that the subcategories of
\cref{notfib1} and \cref{notfib2} admit small filtered colimits preserved by the respective inclusions to
$\Fun(\bD^1,\infty\scat), \infty\scat_{/S}.$
\end{remark}

\subsection{Locality principles for fibrations}

\begin{lemma}Let $0 \leq m \leq \infty.$
A functor $\phi: \mC \to \mD$ is a $m$-anticocartesian fibration 
if and only if for every $n \geq 0$ the functor
$\iota_n(\phi): \iota_n(\mC) \to \iota_n(\mD) $
is a $m$-anticocartesian fibration.
\end{lemma}

\begin{proof}
For every $m$-anticocartesian fibration $\phi$ and $n \geq 0$
the functor $\iota_n(\phi)$ is a $m$-anticocartesian fibration.
On the other hand every functor $\phi$ is the sequential colimit 
$ \iota_0(\phi) \to \iota_1(\phi) \to \cdots $ in $\Fun(\bD^1,\infty\scat)$
and we observe that the subcategory of $\Fun(\bD^1,\infty\scat) $ of $m$-anticocartesian fibrations and maps of such admits small filtered colimits that are preserved by the inclusion to $\Fun(\bD^1,\infty\scat) $.
\end{proof}

\begin{proposition}\label{fiberwiseeq}Let $ n  \geq 0.$
A map $\kappa: \mB \to \mC$ of $n$-anticocartesian fibrations over $\mD$ induces an equivalence under $\iota_n$ if and only if for every $Z \in \mD$ the following induced functor is inverted by $\iota_n:$
\begin{equation}\label{eqzt}
\{Z\} \times_\mD \mB \to  \{Z\} \times_\mD \mC. 
\end{equation}
   
\end{proposition}

\begin{proof}
We proceed by induction on $n \geq 0$.
We start with $n=0.$
If for every $Z \in \mD$ the functor \ref{eqzt}
induces an equivalence under $\iota_0,$
the map $\iota_0(\mB) \to \iota_0(\mC)$
over $\iota_0(\mD)$ is fiberwise an equivalence,
and so an equivalence.

Assume next that the statement holds for $n$ and
let $\kappa: \mB \to \mC $ be a map of $n+1$-anticocartesian fibrations over $\mD$ such that for every $Z \in \mD$ 
the induced functor \begin{equation}\label{eqzzi} \iota_{n+1}(\{Z\} \times_\mD \mB) \to \iota_{n+1}(\{Z\} \times_\mD \mC)\end{equation} is an equivalence.

Then \ref{eqzt} induces an equivalence under $\iota_0$
and by the case $n=0$ the functor $\iota_0(\kappa)$ is an equivalence.
So to see that $\iota_{n+1}(\kappa)$ is an equivalence,
it suffices to see that $\iota_{n+1}(\kappa)$ is fully faithful.
The functor $\iota_{n+1}(\kappa)$ induces on morphism $\infty$-categories between $X,Y \in \mB$ lying over $W,Z \in \mD$, respectively, the functor
$$\iota_n(\Mor_\mB(X,Y)) \to \iota_n(\Mor_\mC(\kappa(X),\kappa(Y))).$$
By induction hypothesis it is enough to see that the map
$\Mor_\mB(X,Y)) \to \Mor_\mC(\kappa(X),\kappa(Y))$
of anticocartesian fibrations over $\Mor_\mD(W,Z) $  
induces on the fiber over every $\alpha: W \to Z $
a functor inverted by $\iota_n.$
Since $\kappa$ is a map of 1-anticocartesian fibrations over $\mD$,
the functor $\Mor_\mB(X,Y)) \to \Mor_\mC(\kappa(X),\kappa(Y))$ induces on the fiber over $\alpha$ the functor
$$ \Mor_{\{Z\} \times_\mD \mB}(\alpha_!(X),Y) \to \Mor_{\{Z\} \times_\mD \mC}(\kappa(\alpha_!(X)),\kappa(Y)),$$
where $\alpha_!$ denotes the target of a cocartesian lift.
The latter functor induces under $\iota_n$ the functor
$$ \Mor_{\iota_{n+1}(\{Z\} \times_\mD \mB)}(\alpha_!(X),Y) \to \Mor_{\iota_{n+1}(\{Z\} \times_\mD \mC)}(\kappa(\alpha_!(X)),\kappa(Y))$$ induced by the functor 
\ref{eqzzi}, which is an equivalence by assumption.
\end{proof}

\begin{corollary}\label{equivcocar}

A map $\kappa: \mB \to \mC$ of anticocartesian fibrations over $\mD$ is an equivalence if and only for every $Z \in \mD$ the following induced functor is an equivalence:
$$ \{Z\} \times_\mD \mB \to  \{Z\} \times_\mD \mC. $$
    
\end{corollary}

\begin{proposition} Let $ n  \geq m \geq 0$ and $\kappa: \mB \to \mC$ 
a functor over $\mD$ that induces under $\iota_n$ a map of anticocartesian fibrations over $\iota_n(\mD)$.
The following are equivalent:

\begin{enumerate}[\normalfont(1)]\setlength{\itemsep}{-2pt}
\item[(1)] The functor $\kappa$ induces a $m$-anticocartesian fibration under $\iota_n$.

\item[(2)] For every $Z \in \mD$ the induced functor 
$
\{Z\} \times_\mD \mB \to  \{Z\} \times_\mD \mC
$
induces a $m$-anticocartesian fibration under $\iota_n$ and for every morphism $\alpha:W \to Z$ in $\mD$ the induced functor
$\alpha_!:\{W\} \times_\mD \mB \to \{Z\} \times_\mD \mB$
induces a map of $m$-anticocartesian fibrations under $\iota_n$.
\end{enumerate}

Similarly, given a commutative diagram of the form 
\begin{equation}\label{sqco0}
\begin{xy}
\xymatrix{
\mB \ar[d]^\kappa \ar[r]^\rho
& \mB' \ar[d]^{\kappa'}
\\ 
\mC \ar[d]^{\lambda} \ar[r]^{\rho'} & \mC' \ar[d]^{\lambda'}
\\ 
\mD \ar[r]^\tau & \mD',
}
\end{xy}\end{equation}
such that $\kappa,\kappa'$ induce under $\iota_n$ $m$-anticocartesian fibrations and maps of anticocartesian fibrations over $\iota_n(\mD)$ and $\iota_n(\mD'),$ respectively, and $\rho,\rho'$
induce under $\iota_n$ maps of $m$-anticocartesian fibrations
over $\iota_n(\tau).$
Then the following are equivalent:

\begin{enumerate}[\normalfont(1)]\setlength{\itemsep}{-2pt}
\item[{(1')}] The top square induces under $\iota_n$ a map of $m$-anticocartesian fibrations.

\item[(2')] For every $Z \in \mD$ the induced commutative square
\begin{equation}\label{sqco1}
\begin{xy}
\xymatrix{
\{Z\} \times_\mD \mB \ar[d] \ar[r]
& \{\tau(Z)\} \times_{\mD'} \mB' \ar[d]
\\ 
\{Z\} \times_\mD \mC \ar[r] & \{\tau(Z)\} \times_{\mD'} \mC',
}
\end{xy}\end{equation}
induces under $\iota_n$ a map of $m$-anticocartesian fibrations.
\end{enumerate}
\end{proposition}

\begin{proof}
By induction on $n\geq 0$ and inspection, (1) implies (2) and (1') implies (2').
We prove by induction on $n \geq 0$ that (2) implies (1) and (2') implies (1').
For $n=0$ there is nothing to show.
Let $n \geq 0.$ We assume that (2) implies (1) for $n.$

Let $\kappa$ be a functor that induces under $\iota_{n+1}$ a map of anticocartesian fibrations over $ \iota_{n+1}(\mD)$ and that satisfies (2) for $n+1$ and $m \leq n+1.$
We prove first that (2) for $m=1$ implies (1) for $m=1.$

For every $X \in \mB$ every morphism $\kappa(X) \to Y$ in $\mC$
admits a lift $ X \to Z \to T,$ where the morphism $X \to Z$ is cocartesian over $\iota_{n+1}(\mD)$ and the morphism $Z \to T$ lies in $\{Y\} \times_{\iota_{n+1}(\mD)} \iota_{n+1}(\mB)$
and is cocartesian with respect to the functor 
$$\{Y\} \times_{\iota_{n+1}(\mD)} \iota_{n+1}(\mB) \to \{Y\} \times_{\iota_{n+1}(\mD)} \iota_{n+1}(\mC).$$
By \cref{elemen} the morphism $X \to Z$ is cocartesian over $\iota_{n+1}(\mC)$ and since cocartesian morphisms are closed under composition,
it suffices to see that every morphism $A \to B$ 
cocartesian for the functor
$\{Y\} \times_{\iota_{n+1}(\mD)} \iota_{n+1}(\mB) \to \{Y\} \times_{\iota_{n+1}(\mD)} \iota_{n+1}(\mC)$
is sent by the projection $ \{Y\} \times_{\iota_{n+1}(\mD)} \iota_{n+1}(\mB)\to \iota_{n+1}(\mB)$ to a morphism cocartesian 
for the functor
$ \iota_{n+1}(\mB) \to \iota_{n+1}(\mC).$

By assumption for every $U \in \mB $ lying over $Z \in \mD$ the induced functor
$$ \Mor_{\iota_{n+1}(\mB)}(B,U) \to \Mor_{\iota_{n+1}(\mB)}(A,U) \times_{\Mor_{\iota_{n+1}(\mC)}(\kappa(B),\kappa(U))} \Mor_{\iota_{n+1}(\mC)}(\kappa(A),\kappa(U)) $$
is a map of anticocartesian fibrations over $\Mor_{\iota_{n+1}(\mD)}(Y,Z)$ and so by \cref{fiberwiseeq} an equivalence if
it induces an equivalence on the fiber over any $\alpha \in \Mor_{\iota_{n+1}(\mC)}(Y,Z).$
The latter functor induces on the fiber over any 
$\alpha \in \Mor_{\iota_{n+1}(\mC)}(Y,Z)$
the functor 
$$ \Mor_{\{Z \}\underset{\iota_{n+1}(\mD)}{\times}\iota_{n+1}(\mB)}(\alpha_!(B), U) \to $$$$ \Mor_{\{ Z \}\underset{\iota_{n+1}(\mD)}{\times}\iota_{n+1}(\mC)}(\kappa(\alpha_!((B)), \kappa(U)) \underset{\Mor_{\{ Z \}\underset{\iota_{n+1}(\mD)}{\times}\iota_{n+1}(\mC)}(\kappa(\alpha_!(A)), \kappa(U))}{\times} \Mor_{\{ Z \}\underset{\iota_{n+1}(\mD)}{\times}\iota_{n+1}(\mB)}(\alpha_!(A), U).$$
The latter is an equivalence because by assumption the morphism $ \alpha_!(A) \to \alpha_!(B) $ is cocartesian for the functor
$\{Z\} \times_{\iota_{n+1}(\mD)} \iota_{n+1}(\mB) \to \{Z\} \times_{\iota_{n+1}(\mD)} \iota_{n+1}(\mC)$.
This proves that (2) implies (1) for $m=1$.
Moreover the description of cocartesian 1-morphisms implies that (2') implies (1') for $m=1.$
    
We prove next that (2) implies (1) for $1 < m \leq n+1.$ 
We prove first that $\iota_{n+1}(\kappa)$ induces on morphism
$\infty$-categories a $m-1$-anticocartesian fibration.
For every $X,Y \in \mB$ lying over $W,Z \in \mD$, respectively,
the functor $$\iota_{n+1}(\kappa)_{X,Y}: \Mor_{\iota_{n+1}(\mB)}(X,Y) \to \Mor_{\iota_{n+1}(\mC)}(\kappa(X),\kappa(Y)) $$
is the image under $\iota_n$ of the functor
$\kappa_{X,Y}: \Mor_\mB(X,Y) \to \Mor_\mC(\kappa(X),\kappa(Y))$
over $\Mor_\mD(W, Z)$, which by \cref{cocarto} is a map of anticocartesian fibrations over $\Mor_\mD(W, Z).$
Hence by induction hypothesis we see that $\iota_{n+1}(\kappa)_{X,Y}$ is a $m-1$-anticocartesian fibration it suffices to observe that for every $\alpha \in \Mor_\mD(W, Z)$ the induced functor 
$$\{\alpha\} \times_{\Mor_\mD(W, Z)} \Mor_\mB(X,Y) \to \{\alpha\} \times_{\Mor_\mD(W, Z)} \Mor_\mC(\kappa(X),\kappa(Y))$$
yields under $\iota_n$ the functor  
$ \iota_n(\Mor_{\{Z\} \times_\mD \mB}(\alpha_!(X),Y)) \to \iota_n(\Mor_{\{Z\} \times_\mD \mC}(\kappa(\alpha_!(X)),\kappa(Y)))$, which is a $m$-anticocartesian fibration by assumption,
and for every morphism $\alpha \to \beta$ in $\Mor_\mD(W, Z)$
the induced functor
$$\{\beta\} \times_{\Mor_\mD(W, Z)} \Mor_\mB(X,Y) \to \{\alpha\} \times_{\Mor_\mD(W, Z)} \Mor_\mB(X,Y)$$
yields under $\iota_n$ the map of anticocartesian fibrations
$$ \iota_n(\Mor_{\{Z\} \times_\mD \mB}(\alpha_!(X),Y) \to \iota_n(\Mor_{\{Z\} \times_\mD \mB}(\beta_!(X),Y) $$ induced by the canonical morphism
$\beta_!(X) \to \alpha_!(X).$

We finish the proof of $(2)\Rightarrow (1)$ by showing that for every pair of morphisms $X' \to X, Y \to Y'$ in $\mB$
the functor $\Mor_{\iota_{n+1}(\mB)}(X,Y) \to \Mor_{\iota_{n+1}(\mB)}(X',Y')$ sends $\iota_{n+1}(\kappa)_{X,Y}$-cocartesian morphisms to
$\iota_{n+1}(\kappa)_{X',Y'}$-cocartesian morphisms.
To see this we apply the induction hypothesis of the implication (2') to (1') to the following commutative diagram:
\begin{equation}\label{sqcoo}
\begin{xy}
\xymatrix{
\Mor_{\iota_{n+1}(\mB)}(X,Y) \ar[d]^{\iota_{n+1}(\kappa)_{X,Y}} \ar[r]
& \Mor_{\iota_{n+1}(\mB)}(X',Y') \ar[d]^{\iota_{n+1}(\kappa)_{X',Y'}}
\\ 
\Mor_{\iota_{n+1}(\mC)}(\kappa(X),\kappa(Y)) \ar[d] \ar[r] & \Mor_{\iota_{n+1}(\mC)}(\kappa(X'),\kappa(Y')) \ar[d]
\\ 
\Mor_{\iota_{n+1}(\mD)}(\lambda(\kappa(X)),\lambda(\kappa(Y))) \ar[r] & \Mor_{\iota_{n+1}(\mD)}(\lambda(\kappa(X')),\lambda(\kappa(Y')).
}
\end{xy}\end{equation}
This square induces the following commutative square:
\begin{equation*}
\begin{xy}
\xymatrix{
\iota_n(\Mor_{\{Z\} \times_\mD \mB}(\alpha_!(X),Y))
\ar[d] \ar[r]
& \iota_n(\Mor_{\{Z\} \times_\mD \mB}(\alpha_!(X'),Y')) \ar[d]
\\ 
\iota_n(\Mor_{\{Z\} \times_\mD \mC}(\kappa(\alpha_!(X)),\kappa(Y))) \ar[r] & \iota_n(\Mor_{\{Z\} \times_\mD \mC}(\kappa(\alpha_!(X')),\kappa(Y'))).
}
\end{xy}\end{equation*}
Therefore (1) follows from the description of the functor $\{\alpha\} \times_{\Mor_\mD(W, Z)} \iota_{n+1}(\kappa)_{X,Y}$ and the assumption that the functor $
\iota_{n+1}(\{Z\} \times_\mD \mB) \to  \iota_{n+1}(\{Z\} \times_\mD \mC)$ is a $m$-anticocartesian fibration and that the latter commutative square is a map of $m-1$-anticocartesian fibrations.

We now prove that $(2')\Rightarrow (1')$.
We have already proven this for $m=1.$
We assume $(2')\Rightarrow (1')$ for $n$
and assume that $(2')$ holds for $n+1$ and $m \leq n+1.$
By the induction hypothesis it remains to see that for every $X,Y \in \mB$ lying over $W,Z \in \mD$, respectively, and $\alpha \in \Mor_\mD(W,Z) $ the induced commutative square
\begin{equation*}
\begin{xy}
\xymatrix{
\{\alpha\} \times_{\Mor_\mD(W,Z)} \Mor_\mB(X,Y) \ar[d] \ar[r]
& \{\alpha\} \times_{\Mor_{\mD'}(\tau(W),\tau(Z))} \Mor_{\mB'}(\rho(X),\rho(Y)) \ar[d]
\\ 
\{\alpha\} \times_{\Mor_\mD(W,Z)} \Mor_\mC(\kappa(X),\kappa(Y)) \ar[r] & \{\alpha\} \times_{\Mor_{\mD'}(\tau(W),\tau(Z))} \Mor_{\mC'}(\rho'(\kappa(X)),\rho'(\kappa(Y))),
}
\end{xy}\end{equation*}
induces under $\iota_n$ a map of $m-1$-anticocartesian fibrations.
By what we have proven, the latter commutative square induces under $\iota_n$ the following one:

\begin{equation*}
\begin{xy}
\xymatrix{
\Mor_{\{Z\} \times_\mD \mB}(\alpha_!(X),Y) \ar[d] \ar[r]
& \Mor_{\{\tau(Z)\} \times_{\mD'} \mB'}(\rho(\alpha_!(X)),\rho(Y)) \ar[d]
\\ 
\Mor_{\{Z\} \times_\mD \mC}(\kappa(\alpha_!(X)),\kappa(Y)) \ar[r] & \Mor_{\{\tau(Z)\} \times_{\mD'} \mC'}(\rho'(\kappa(\alpha_!(X))),\rho'(\kappa(Y))).
}
\end{xy}\end{equation*}
By assumption the latter induces under $\iota_n$ a map of $m-1$-anticocartesian fibrations.
\end{proof}

\begin{corollary}Let $m \geq 0.$
A map $\kappa: \mB \to \mC$ of anticocartesian fibrations over $\mD$ is a $m$-anticocartesian fibration if and only if for every $Z \in \mD$ the induced functor 
$
\{Z\} \times_\mD \mB \to  \{Z\} \times_\mD \mC
$
is a $m$-anticocartesian fibration and for every morphism $Z \to W$ in $\mD$ the induced functor
$\{W\} \times_\mD \mB \to \{Z\} \times_\mD \mB$
is a map of $m$-anticocartesian fibrations.
   
\end{corollary}

\begin{corollary}\label{paracocart}
A map $\kappa: \mB \to \mC$ of cocartesian fibrations over $\mD$ is a cocartesian fibration if and only if for every $Z \in \mD$ the induced functor 
$
\{Z\} \times_\mD \mB \to  \{Z\} \times_\mD \mC
$
is a cocartesian fibration and for every morphism $Z \to W$ in $\mD$ the induced functor
$\{W\} \times_\mD \mB \to \{Z\} \times_\mD \mB$
is a map of cocartesian fibrations.
   
\end{corollary}

We have the following local definition of anticocartesian fibrations:

\begin{proposition}\label{pua} Let $0 \leq n \leq \infty$ and $\phi: \mC \to \mD$ a functor.
The following are equivalent:

\begin{enumerate}[\normalfont(1)]\setlength{\itemsep}{-2pt}
\item The functor $\iota_n(\phi)$ is an anticocartesian fibration.

\item For every $\theta \in \Theta \cap n\Cat$ the pullback of $\iota_n(\phi)$ along any functor $\theta \to \iota_n(\mD)$ is an anticocartesian fibration.

\end{enumerate}

\end{proposition}

\begin{proof}
Condition (1) implies (2) since cocartesian fibrations are stable under pullback.
We prove by induction on $n \geq 0$ that (2) implies (1).
For $n=0$ there is nothing to show.
We prove the induction step. Let $n \geq 0$ and assume that (2) implies (1) for the case of $n$ and any functor $\phi.$
Let $\phi: \mC \to \mD$ be a functor satisfying (2) for $n+1$.

We prove first that for any $X,Y \in \mC$ the functor
$\iota_{n+1}(\phi)_{X,Y}: \Mor_{\iota_{n+1}(\mC)}(X,Y) \to \Mor_{\iota_{n+1}(\mD)}(\phi(X), \phi(Y)) $
is a $n$-anticocartesian fibration and for every morphisms $X' \to X, Y \to Y'$ in $\mC$ the induced functor
\[
\Mor_{\iota_{n+1}(\mC)}(X,Y) \to \Mor_{\iota_{n+1}(\mC)}(X',Y')
\]
sends $\iota_{n+1}(\phi)_{X,Y}$-cocartesian morphims to $\iota_{n+1}(\phi)_{X',Y'}$-cocartesian morphisms.
By induction hypothesis the first is equivalent to say that
for every $X,Y \in \mC$ and functor $\sigma: \theta \to \Mor_{\iota_{n+1}(\mD)}(\phi(X), \phi(Y))$ for
$\theta \in \Theta \cap n\Cat$ the pullback $ \theta \times_{\Mor_{\iota_{n+1}(\mD)}(\phi(X), \phi(Y))} \Mor_{\iota_{n+1}(\mC)}(X,Y) \to \theta $ along $\sigma$ is an anticocartesian fibration.
By assumption the pullback $\nu: S(\theta) \times_{\iota_{n+1}(\mD)} \iota_{n+1}(\mC) \to S(\theta) $ along the functor
$S(\theta) \to \iota_{n+1}(\mD)$ corresponding to $\sigma $ is an anticocartesian fibration.
The projection $\nu $ induces on morphism $\infty$-categories between $X,Y$ the pullback $ \nu_{X,Y}: \theta \times_{\Mor_{\iota_{n+1}(\mD)}(\phi(X), \phi(Y))} \Mor_{\iota_{n+1}(\mC)}(X,Y) \to \theta $ of $\sigma$, which is therefore an anticocartesian fibration.
This implies by induction hypothesis that the functor
$\iota_{n+1}(\phi)_{X,Y}$ is an anticocartesian fibration. So the projection
$$ \theta \times_{\Mor_{\iota_{n+1}(\mD)}(\phi(X), \phi(Y))} \Mor_{\iota_{n+1}(\mC)}(X,Y) \to \Mor_{\iota_{n+1}(\mC)}(X,Y)$$ sends $\nu_{X,Y}$-cocartesian morphisms to $\iota_{n+1}(\phi)_{X,Y}$-cocartesian morphisms.
Let $X' \to X, Y \to Y' $ be morphisms in $\mC$ classified by functors
$\bD^1 \to \mC $.
The functor $S(\theta) \to \iota_{n+1}(\mD) $ and the two functors $\bD^1 \to \mC \to \mD$ give rise to a functor $\bD^1 \vee S(\theta) \vee \bD^1 \to  \iota_{n+1}(\mD) $ and the pullback $ \bD^1 \vee S(\theta) \vee \bD^1 \times_{ \iota_{n+1}(\mD)} \iota_{n+1}(\mC) \to \bD^1 \vee S(\theta) \vee \bD^1 $ is an anticocartesian fibration by assumption. The morphisms $X' \to X, Y \to Y'$ in $\mC$
determine morphisms of the pullback $ \bD^1 \vee S(\theta) \vee \bD^1 \times_{ \iota_{n+1}(\mD)} \iota_{n+1}(\mC) $ lying over the two copies of $\bD^1$.
Hence the induced functor
$\Mor_{\iota_{n+1}(\mC)}(X,Y) \to \Mor_{\iota_{n+1}(\mC)}(X',Y')$ sends $\iota_{n+1}(\phi)_{X,Y}$-cocartesian morphims to $\iota_{n+1}(\phi)_{X',Y'}$-cocartesian morphisms if the induced functor
$$\Mor_{\bD^1 \vee S(\theta) \vee \bD^1 \times_{ \iota_{n+1}(\mD)} \iota_{n+1}(\mC)}(X,Y) \to \Mor_{\bD^1 \vee S(\theta) \vee \bD^1 \times_{ \iota_{n+1}(\mD)} \iota_{n+1}(\mC)}(X',Y')$$ sends $\nu_{X,Y}$-cocartesian morphims to $\nu_{X',Y'}$-cocartesian morphisms.
The latter holds because $\nu$ is an anticocartesian fibration.

We complete the proof by showing that $\iota_{n+1}(\phi)$ is a 1-cocartesian fibration.
Condition (2) implies that the pullback of $\iota_{n+1}(\phi)$ along any functor $\bD^1 \to \iota_{n+1}(\mD) $ is an anticocartesian fibration.
Hence for every $X \in \mC $ every morphism $\phi(X) \to T$
admits a lift $f: X \to Y $ cocartesian for the projection
$\bD^1 \times_{\iota_{n+1}(\mD)} \iota_{n+1}(\mC) \to \bD^1 $.
By the description of morphism $\infty$-categories in pullbacks a morphism $f: X \to Y $ in $\mC$ is cocartesian for the latter projection
if and only if for every $Z \in \mC $ lying over $\phi(Y)$ the following induced functor is an equivalence:
$$ \Mor_{\{\phi(Y)\}\times_{\iota_{n+1}(\mD)} \iota_{n+1}(\mC)}(Y,Z) \to \{\phi(f)\} \times_{\Mor_{\iota_{n+1}(\mD)}(\phi(X),\phi(Y))} \Mor_{\iota_{n+1}(\mC)}(X,Z).$$

We would like to see that $f$ is $\phi$-cocartesian, i.e. that for every $Z \in \mC$ the induced functor
\begin{equation}\label{thefunctor}
    \Mor_{\iota_{n+1}(\mC)}(Y,Z) \to \Mor_{\iota_{n+1}(\mD)}(\phi(Y),\phi(Z)) \times_{\Mor_{\iota_{n+1}(\mD)}(\phi(X),\phi(Z))} \Mor_{\iota_{n+1}(\mC)}(X,Z)
\end{equation}
    is an equivalence.
By the first part of the proof 
the functor \ref{thefunctor} is a map of anticocartesian fibrations over 
$\Mor_{\iota_{n+1}(\mD)}(\phi(Y),\phi(Z))$.
By \cref{fiberwiseeq} this guarantees that \ref{thefunctor} is an equivalence if it induces an equivalence on the fiber over any $\alpha \in \Mor_{\iota_{n+1}(\mD)}(\phi(Y),\phi(Z))$.
Since the pullback of $\phi$ along any functor $\bD^1 \to \mD$ is an anticocartesian fibration, \ref{thefunctor} induces on the fiber over $\alpha \in \Mor_{\iota_{n+1}(\mD)}(\phi(Y),\phi(Z))$ the functor
$$\Mor_{\{\phi(Z)\}\times_{\iota_{n+1}(\mD)} \iota_{n+1}(\mC)}(\alpha_!(Y),Z) \to \Mor_{\{\phi(Z)\}\times_{\iota_{n+1}(\mD)} \iota_{n+1}(\mC)}((\alpha \circ f)_!(X),Z)$$
induced by the canonical morphism $(\alpha \circ f)_!(X) \to \alpha_!(f_!(X)) \simeq \alpha_!(Y),$
where the lifts are taken with respect to the respective anticocartesian fibrations
$\bD^1 \times_{\iota_{n+1}(\mD)} \iota_{n+1}(\mC) \to \bD^1 $.

To prove that the previous morphism is an equivalence we need to see that for every pair of morphisms $f: X \to Y, g: Y \to Z $ in $\mC$ cocartesian for the respective projection $\bD^1 \times_{\iota_{n+1}(\mD)} \iota_{n+1}(\mC) \to \bD^1$ the 
composite $g \circ f$ is also cocartesian for the respective projection $\bD^1 \times_{\iota_{n+1}(\mD)} \iota_{n+1}(\mC) \to \bD^1$.
For that we observe that the functor
$\bD^1 \vee \bD^1 \to \iota_{n+1}(\mC)$ classifying $f,g $ factors through the pullback 
$ (\bD^1 \vee \bD^1) \times_{\iota_{n+1}(\mD)} \iota_{n+1}(\mC) $
and the composition $\bD^1 \to \bD^1 \vee \bD^1 \to (\bD^1 \vee \bD^1) \times_{\iota_{n+1}(\mD)} \iota_{n+1}(\mC) $ classifying $g \circ f$ is 
cocartesian for the projection $\rho: (\bD^1 \vee \bD^1) \times_{\iota_{n+1}(\mD)} \iota_{n+1}(\mC) \to \bD^1 \vee \bD^1$, which is a cocartesian fibration by assumption. 
The latter holds since both restrictions $\bD^1 \to \bD^1 \vee \bD^1 \to (\bD^1 \vee \bD^1) \times_{\iota_{n+1}(\mD)} \iota_{n+1}(\mC) $ classifying $f,g,$ respectively, are by assumption cocartesian with respect to the respective projection $\bD^1 \times_{\iota_{n+1}(\mD)} \iota_{n+1}(\mC) \to \bD^1$ and so also $\rho$-cocartesian by the existence of $\rho$-cocartesian lifts.
Thus also the composite $g \circ f: X \to Z$ is $\rho$-cocartesian.
\end{proof}

\begin{corollary}\label{thetalocal} Let $\phi: \mC \to \mD$ be a functor.
The following are equivalent:

\begin{enumerate}[\normalfont(1)]\setlength{\itemsep}{-2pt}
\item The functor $\phi$ is a cocartesian fibration.

\item For every $\theta \in \Theta$ the pullback of $\phi$ along any functor $\theta \to \mD$ is a cocartesian fibration.

\item The pullback of $\phi$ along any functor $\bbDelta^n \to \mD$
for $n \geq 0 $ is a cocartesian fibration.

\item The pullback of $\phi$ along any functor $\cube^n \to \mD$
for $n \geq 0 $ is a cocartesian fibration.

\end{enumerate}

\end{corollary}

\begin{proof}
The statement follows from \cref{pua} and \cref{orientdense} and  \cref{cubedense}.
\end{proof}

\begin{proposition}\label{cocartexp} Let $1 \leq n \leq \infty$ and $ \phi: \mC \to \mD$ a $n$-anticocartesian fibration and $\mB$ an $\infty$-category.
The induced functor 
$$ \Fun(\mB, \mC) \to \Fun(\mB,\mD) $$
is a $n$-anticocartesian fibration whose cocartesian $k$-morphisms for
$ 1 \leq k \leq n$ are object-wise.

\end{proposition}

\begin{proof}We prove the statement by induction on $ n \geq 1.$
The case $n=1$ is \cref{enrcocartexp}.
We assume the statement holds for $n-1$.
By \cref{enrcocartexp} the induced functor 
$ \Fun(\mB, \mC) \to \Fun(\mB,\mD) $
is a 1-anticocartesian fibration.
By \cref{cocarto} it suffices to see that
for every $\X,\Y \in \Fun(\mB, \mC)$ the induced functor $\Mor_{\Fun(\mB, \mC)}(\X,\Y) \to \Mor_{\Fun(\mB, \mD)}(\phi(\X),\phi(\Y))$ is a $n-1$-anticocartesian fibration
whose cocartesian $k$-morphisms for $ 1 \leq k \leq n-1$ are objectwise.
Since $\phi$ is a $n$-anticocartesian fibration, this will imply that for every morphism $X \to X', Y' \to Y$ in $\Fun(\mB, \mC)$ the induced functor 
$\Mor_{\Fun(\mB, \mC)}(X,Y) \to \Mor_{\Fun(\mB, \mC)}(X',Y')$ sends cocartesian morphisms to cocartesian morphisms.
So we will conclude via \cref{cocarto}.

By \cite[Corollary 4.49.]{heine2024bienriched} the induced functor $$\Mor_{\Fun(\mB, \mC)}(\X,\Y) \to \Mor_{\Fun(\mB, \mD)}(\phi(\X),\phi(\Y))$$
identifies with the canonical functor
$$ \lim_{[n]\in \Delta} \prod_{(Z_0,...,Z_n) \in \iota_0(\mB)^{\times n+1}} \Fun(\Mor_\mB(Z_{n-1},Z_n)\times ... \times \Mor_\mB(Z_1,\Z_2) \times \Mor_\mB(Z_0,Z_1),\Mor_\mC(X(Z_0), Y(Z_n))) \to $$
$$ \lim_{[n]\in \Delta} \!\!\!\!\!\prod_{(Z_0,...,Z_n) \in \iota_0(\mB)^{\times n+1}} \!\!\!\!\!\Fun(\Mor_\mB(Z_{n-1},Z_n)\times ... \times \Mor_\mB(Z_1,Z_2) \times \Mor_\mB(Z_0,Z_1),\Mor_\mD(\phi(X(Z_0)), \phi(Y(Z_n)))).$$
The functor $\phi$ induces on morphism $\infty$-categories $n-1$-anticocartesian fibrations. We conclude via the induction hypothesis
since the subcategory of $n-1$-anticocartesian fibrations is closed under small limits.
\end{proof}

\subsection{Locally cocartesian fibrations}

\begin{definition}Let $1 \leq \n \leq \infty$ and $\phi: \mC \to \mD$ a functor.

\begin{enumerate}[\normalfont(1)]\setlength{\itemsep}{-2pt}

\item A 1-morphism $f: X \to Y$ in $\mC$ is locally $\phi$-cocartesian if for every $Z\in \mC$ lying over $\phi(Y) $ the induced functor
$$ \Mor_{\mC_{\phi(Y)}}(Y,Z) \to \{ \phi(f) \}\times_{\Mor_\mD(\phi(X), \phi(Z)} \Mor_\mC(X,Z) $$
is an equivalence.

\item An $n$-morphism $\alpha:\bD^n\to\mC$ is locally $\phi$-cocartesian if
for every pair of morphisms $(X \to \alpha(0), \alpha(1) \to Y)$ in $\mC$ lying over equivalences in $\mD$ the composite $n-1$-morphism
\[
\bD^{n-1}\to\Mor_\mC(\alpha(0),\alpha(1)) \to \Mor_\mC(X,Y)
\]
is locally $\phi_{X,Y}$-cocartesian.

\end{enumerate}

\end{definition}

\begin{definition}Let $\n \geq 1$ and $\phi: \mC \to \mD$ a functor.
A $n$-morphism in $\mC$ is locally $\phi$-cartesian if and only if
it is locally $\phi^{\co\op}$-cocartesian.
    
\end{definition}

\begin{remark}Let $\mA$ be an $\infty$-category and $\phi: \mM \to S(\mA)$ be a functor whose fibers over 0,1 we denote by $\mC, \mD,$ respectively.
A morphism $f: X \to Y$ of $\mM$, where $X \in \mC, Y \in \mD$, lying over $A \in \mA$, is $\phi$-cocartesian if and only if for every $Z\in \mD$ the commutative square 
\begin{equation}\label{filler4}
\begin{xy}
\xymatrix{
\Mor_\mM(Y,Z) \ar[d] \ar[r]
& \Mor_\mM(X,Z) \ar[d]^\phi
\\ 
\ast \ar[r]^{A} & \mA
}
\end{xy}\end{equation}
is a pullback square.
The latter condition is equivalent to say that the morphism $f: X \top Y$ of $\mM$ is locally $\phi$-cocartesian.
So every locally $\phi$-cocartesian is $\phi$-cocartesian in this example.

\end{remark}

\begin{proposition}\label{loccharto}
Let $1 \leq \n \leq \infty$ and $\phi: \mC \to \mD$ a functor.
An $n$-morphism $\alpha:\bD^n\to\mC$ is locally $\phi$-cocartesian if and only if it is $\phi'$-cocartesian, where $\phi': \bD^n \times_\mD \mC \to \bD^n$ is the pullback of $\phi$ along $\alpha.$ 

\end{proposition}

\begin{proof}

We proceed by induction on $n \geq 1.$
We prove the induction start.
An $1$-morphism $f:X \to Y$ is $\phi'$-cocartesian, where $\phi': \bD^1 \times_\mD \mC \to \bD^1$ if and only if for every $Z \in \mC \times_\mD \{ \phi(Y) \} $ the commutative square 
\begin{equation}\label{filler5}
\begin{xy}
\xymatrix{
\Mor_{(\bD^1 \times_\mD \mC)}(Y,Z) \ar[d] \ar[r]
& \Mor_{(\bD^1 \times_\mD \mC)}(X,Z) \ar[d]^\phi
\\ 
\ast \ar[r]^{A} & \ast
}
\end{xy}\end{equation}
is a pullback square.
The top functor in this square identfies with the canonical functor
$$ \Mor_{\mC_{\phi(Y)}}(Y,Z) \to \{ \phi(f) \}\times_{\Mor_\mD(\phi(X), \phi(Z)} \Mor_\mC(X,Z).$$

We prove the induction step. Let $n > 0$. We assume the statement holds for $n-1$.

An $n$-morphism $\alpha:\bD^n\to\mC$ is $\phi'$-cocartesian, where $\phi': \bD^n \times_\mD \mC \to \bD^n$ is the pullback of $\phi$ along $\alpha$ if and only if the corresponding $n-1$-morphism $\alpha':\bD^{n-1}\to \Mor_\mC(\alpha(0),\alpha(1))$ is cocartesian for the pullback
$$\bD^{n-1} \times_{\Mor_\mD(\phi(\alpha(0)),\phi(\alpha(1)))} \Mor_\mC(\alpha(0),\alpha(1)) \to \bD^{n-1}$$
and for every morphism
$X \to \alpha(0)$ lying over $\phi(\alpha(0)$ and $\alpha(1) \to Y$ lying over $\phi(\alpha(1)$ the induced functor
$$\bD^{n-1} \times_{\Mor_\mD(\phi(\alpha(0)),\phi(\alpha(1)))} \Mor_\mC(\alpha(0),\alpha(1)) \to \bD^{n-1} \times_{\Mor_\mD(\phi(\alpha(0)),\phi(\alpha(1)))} \Mor_\mC(X,Y)$$
sends this $n-1$-morphism cocartesian over $\bD^{n-1}$ to an $n-1$-morphism cocartesian over $\bD^{n-1}$. 

So by induction hypothesis, an $n$-morphism $\alpha:\bD^n\to\mC$ is $\phi'$-cocartesian if and only if for every morphism
$X \to \alpha(0)$ and $\alpha(1) \to Y$ lying over equivalences the induced functor
$$\bD^{n-1} \to \Mor_\mC(\alpha(0),\alpha(1)) \to \Mor_\mC(X,Y) $$
is locally $\phi_{X,Y}$-cocartesian.  
\end{proof}

\begin{definition} Let $1 \leq n \leq \infty.$
A functor $\phi: \mC \to \mD$ is a locally $n$-anticocartesian fibration if for every $1 \leq k \leq n$ 
every commutative square 
$$\begin{xy}
\xymatrix{
\bD^{k-1} \ar[d] \ar[r]
& \mC \ar[d]^\phi
\\ 
\bD^k \ar[r] & \mD
}
\end{xy}$$
admits a filler by a locally $\phi$-cocartesian morphism,
where the left vertical functor is the left inclusion.

\end{definition}

\begin{definition}Let $1 \leq n \leq \infty. $
.
\begin{enumerate}[\normalfont(1)]\setlength{\itemsep}{-2pt}

\item A functor $\phi: \mC \to \mD$ is a locally $n$-cocartesian fibration if $\phi^{\co}$ is a locally $n$-anticocartesian fibration.

\item A functor $\phi: \mC \to \mD$ is a locally $n$-cartesian fibration if $\phi^\op$ is a locally $n$-anticocartesian fibration.

\item A functor $\phi: \mC \to \mD$ is a locally $n$-anticartesian fibration if $\phi^{\co\op}$ is a locally $n$-anticocartesian fibration.

\item For $n = \infty$ we skip $n.$

\end{enumerate}
    
\end{definition}

We have the following inductive definition of locally anticocartesian fibrations:

\begin{proposition}\label{cocarto1}
Let $1 \leq n \leq \infty$ and $\phi: \mC \to \mD$ a functor.
The following are equivalent:

\begin{enumerate}[\normalfont(1)]\setlength{\itemsep}{-2pt}
\item The functor $\phi$ is a locally $n$-anticocartesian fibration.

\item The functor $\phi$ is a locally 1-anticocartesian fibration, and for every $\X,\Y \in \mC$ the functor
\[
\phi_{X,Y}: \Mor_\mC(\X,\Y) \to \Mor_\mD(\phi(\X),\phi(\Y))
\]
is a locally $n-1$-anticocartesian fibration
and for every pair of morphisms $X' \to X$ and $Y \to Y'$ in $\mC$
lying over equivalences in $\mD$ the induced functor 
$\Mor_\mC(\X,\Y) \to \Mor_\mC(X',Y')$ sends locally $\phi_{X,Y}$-cocartesian morphisms to locally $\phi_{X',Y'}$-cocartesian morphisms.
\end{enumerate}

\end{proposition}

\begin{proof}
The same proof as \cref{cocarto}.
\end{proof}

\cref{loccharto} gives the following:

\begin{corollary}Let $1 \leq n \leq \infty.$
A functor $X \to Y$ is a locally $n$-anticocartesian fibration if and only if its pullback along every functor $\bD^\bk \to Y$ for $\bk \leq n$ is an anticocartesian fibration.	
    
\end{corollary}

\begin{corollary}
A functor $X \to Y$ is a locally cocartesian fibration if and only if its pullback along every functor $\bD^n \to Y$ for $n \geq 0$ is a cocartesian fibration.	
    
\end{corollary}

\begin{corollary}Let $1 \leq n \leq \infty.$
Every locally cocartesian fibration $X \to \bD^n$ is a cocartesian fibration.	
    
\end{corollary}

\section{\mbox{Bifibrations, free fibrations, and the universal fibration}}

\subsection{Oriented pullbacks are bifibrations}

\begin{definition}

Let $ n \geq 0 $ and $\mA,\mB,\mC$ be $\infty$-categories.

\begin{enumerate}[\normalfont(1)]\setlength{\itemsep}{-2pt}
\item An $n$-bifibration 
is a functor $\mC \to \mA \times \mB $ that is a map of $n$-cartesian fibrations over $\mA$ and $n$-cocartesian fibrations over $\mB.$

\item A $n$-antibifibration 
is a functor $\mC \to \mA \times \mB $ that is a map of $n$-anticartesian fibrations over $\mA$ and $n$-anticocartesian fibrations over $\mB.$

\item A bifibration 
is a functor $\mC \to \mA \times \mB $ that is a $n$-bifibration for every $n \geq 0.$

\item An antibifibration 
is a functor $\mC \to \mA \times \mB $ that is a $n$-antibifibration for every $n \geq 0.$
\end{enumerate}

\end{definition}

\begin{remark}
Let $X$ be an $\infty$-category.
We will show in \cref{targetfibr} that the source and target projections
\[
\Fun^{\oplax}(\bD^1,X)\to X\times X\qquad\text{and}\qquad \Fun^{\lax}(\bD^1,X)\to X\times X
\]
form a bifibration and antibifibration, respectively.
\end{remark}

\begin{definition}
Let $ n \geq 0 $ and $\mA,\mB,\mC, \mA',\mB',\mC'$ be $\infty$-categories.
A map of $n$-(anti)bifibrations is a commutative square
\begin{equation}\label{sqfib}
\xymatrix{
\mC \ar[r] \ar[d] & \mC' \ar[d] \\
\mA \times \mB \ar[r] & \mA' \times \mB'
}
\end{equation}
whose vertical functors are $n$-(anti)bifibrations,
that induces a map of (anti) $n$-cartesian fibrations after projection to the first factor and induces a map of (anti) $n$-cocartesian fibrations after projection to the second factor.
A map of (anti)bifibrations is a commutative square
(\ref{sqfib}) that is a map of $n$-(anti)bifibrations for every $n \geq 0.$

\end{definition}

\begin{notation}

Let $\mC$ be an antioriented category  and let
$$
\begin{tikzcd}
\mA \ar{r}{} \ar{d}[swap]{} & \mA' \ar[double]{dl}{} \ar{d}{} \\
C  \ar{r}[swap]{} & C',
\end{tikzcd}
\qquad
\begin{tikzcd}
B \ar{r}{} \ar{d}[swap]{} & B'  \ar{d}{} \\
C \ar[double]{ur}{}  \ar{r}[swap]{} & C'
\end{tikzcd}
$$
be antioriented squares in $\mC$. 
The outer antioriented square 
$$
\begin{tikzcd}
\mA {\bar{\vec{\times}}}_{\mC} \mB \ar{r}{} \ar{d}[swap]{} & \mB \ar{d}{}\ar{r}{} & \mB' \ar{d}{} \\
\mA \ar[double]{ur}{} \ar{d}{} \ar{r}[swap]{} & \mC \ar[double]{ur}{} \ar{d}{}\ar{r}{} & \mC' \ar{d}{} \\
\mA' \ar[double]{ur}{} \ar{r}{} & \mC' \ar{r}{} & \mC'
\end{tikzcd}
$$
in $\mC$ gives a morphism in $\mC:$
$$  \mA {\bar{\vec{\times}}}_{\mC} \mB \to \mA' {\bar{\vec{\times}}}_{\mC'} \mB'.$$

\end{notation}

\begin{theorem}\label{targetfi}
Let $F: \mA \to \mC$ and  $G: \mB \to \mC$ be functors.

\begin{enumerate}[\normalfont(1)]\setlength{\itemsep}{-2pt}
\item The functor $ \mA {\bar{\vec{\times}}}_{\mC} \mB \to \mA \times \mB $
is an antibifibration.

\item Let
$$
\begin{tikzcd}
\mA \ar{r}{\alpha} \ar{d}[swap]{} & \mA' \ar[double]{dl}{} \ar{d}{} \\
\mC  \ar{r}[swap]{\gamma} & \mC',
\end{tikzcd}
\qquad
\begin{tikzcd}
\mB \ar{r}{\beta} \ar{d}[swap]{} & \mB'  \ar{d}{} \\
\mC \ar[double]{ur}{} \ar{r}[swap]{\gamma} & \mC'
\end{tikzcd}
$$
be antioriented squares of $\infty$-categories.
The induced commutative square
\begin{equation}\label{sqfibb}
\begin{tikzcd}
\mA {\bar{\vec{\times}}}_{\mC} \mB \ar{r}{} \ar{d}[swap]{} & \mA' {\bar{\vec{\times}}}_{\mC'} \mB' \ar{d}{} \\
\mA \times \mB  \ar{r}[swap]{} & \mA' \times \mB'
\end{tikzcd}
\end{equation}
is a map of antibifibrations.
\end{enumerate}
\end{theorem}

\begin{proof}

We start with proving that the functor 
$ \mA {\bar{\vec{\times}}}_{\mC} \mB \to \mA \times \mB $
is a map of 1-bifibrations, i.e. 1-cartesian fibrations over $\mA$ and 1-cocartesian fibrations over $\mB$.
We prove the second statement, the first statement is dual
by applying $(-)^{\coop}.$

Let $ f: \mA \to \mB, t: \mB \to \mB'$ be morphisms of $\mC$.
The commutative square
\[
\xymatrix{
\mA\ar[r]^\id\ar[d]^f & \mA \ar[d]^{t \circ f}\\
\mB\ar[r]^t & \mB'
}
\]
in $\mC$ determines a morphism of $\Fun^\lax(\bD^1,\mC)$ from $f$ to $t \circ f.$
We prove that this morphism is $\ev_1$-cocartesian.
This means that for every morphism $h: X \to Y $ in $\mC$
the following commutative square is a pullback square:
\[
\xymatrix{
\Mor_{\Fun^\lax(\bD^1,\mC)}(t\circ f,h) \ar[r] \ar[d] & \Mor_\mC(B',\Y) \ar[d] \\
\Mor_{\Fun^\lax(\bD^1,\mC)}(f,h) \ar[r] &\Mor_\mC(B,\Y).
}
\]

By \cref{homs2} the latter commutative square identifies with the following pullback square, where the horizontal functors are the projections:
\[
\xymatrix{
\Mor_\mC(A, X) \times_{\Mor_\mC(A, Y)} \Fun^\lax(\bD^1,\Mor_\mC(A,Y)) \times_{\Mor_\mC(A, Y)} \Mor_\mC(B',Y) \ar[r] \ar[d] & \Mor_\mC(B',\Y)\ar[d] \\
\Mor_\mC(A, X) \times_{\Mor_\mC(A, Y)} \Fun^\lax(\bD^1,\Mor_\mC(A,Y)) \times_{\Mor_\mC(A, Y)} \Mor_\mC(B,Y) \ar[r] &\Mor_\mC(B,\Y).
}
\]
Therefore for any functors $\mA \to \mC, \mB \to \mC$ the pullback
$ \mA {\bar{\vec{\times}}}_{\mC} \mB \to \mA \times \mB $
is a 1-bifibration whose cocartesian 1-morphisms
are the triples 
$ (A \to A', B \to B', \sigma),$ where $A \to A'$ is an equivalence and $\sigma$ is a commutative square in $\mC:$
$$
\begin{tikzcd}
F(A) \ar{r}{} \ar{d}[swap]{} & G(B) \ar{d}{} \\
F(A')  \ar{r}[swap]{} & G(B').
\end{tikzcd}
$$

The top horizontal functor in (\ref{sqfibb})
sends an object $(A,B,F(A) \to G(B))$ to $(\alpha(A),\beta(B),F'(\alpha(A)) \to \gamma(F(A)) \to \gamma(G(B)) \to G'(\beta(B)))$, and sends a morphism 
$ (A \to A', B \to B', \sigma),$ where $\sigma$ is a right square
\begin{equation}\label{orsqu}
\begin{tikzcd}
F(A) \ar{r}{} \ar{d}[swap]{} & G(B) \ar[double]{dl}{} \ar{d}{} \\
F(A')  \ar{r}[swap]{} & G(B')
\end{tikzcd}
\end{equation}
to the morphism $ (\alpha(A) \to \alpha(A'), \beta(B) \to \beta(B'), \sigma'),$ where $\sigma'$ is the right square
$$
\begin{tikzcd}
F'(\alpha(A)) \ar{r}{} \ar{d}[swap]{} & \gamma(F(A)) \ar{r}{} \ar{d}[swap]{} & \gamma(G(B)) \ar[double]{dl}{} \ar{d}{} \ar{r}{} & G'(\beta(B)) \ar{d}{} \\
F'(\alpha(A')) \ar{r}[swap]{} & \gamma(F(A')) \ar{r}[swap]{} & \gamma(G(B')) \ar{r}[swap]{} & G'(\beta(B')).
\end{tikzcd}
$$
This proves that the top horizontal functor in \ref{sqfibb} 
is a map of 1-bifibrations.

Next we prove the general case. For that we first reduce to the case that $\mA, \mB, \mC$ are $n$-categories for some $\n \geq 0.$
The functor $ \mA {\bar{\vec{\times}}}_{\mC} \mB \to \mA \times \mB$ is the sequential colimit
$$ {\iota_n(\mA) {\bar{\vec{\times}}}}_{\iota_n(\mC)} \iota_n(\mB) \to \iota_n(\mA) \times \iota_n(\mB) $$
since the functor $\Fun^\lax(\bD^1,-): \infty\Cat \to \infty\Cat$ preserves small filtered colimits by compact generation of the oriented cubes of \cref{cubedense}.
Since bifibrations are stable under filtered colimits, we can reduce to the case that $\mA,\mB,\mC $ are $n$-categories for some $\n \geq 0$.
In this case by \cref{dimensio} the $\infty$-category $\Fun^\lax(\bD^1,\mC) $ is an $n$-category so that the functor
$ \mA {\bar{\vec{\times}}}_{\mC} \mB \to \mA \times \mB$ is between $n$-categories.
Simimarly, the commutative \ref{sqfibb} is the sequential colimit of the sequence of commutative squares

$$
\begin{tikzcd}
{\iota_n(\mA) {{\bar{\vec{\times}}}}_{\iota_n(\mC)} \iota_n(\mB)} \ar{r}{} \ar{d}[swap]{} & {\iota_n(\mA') {{\bar{\vec{\times}}}}_{\iota_n(\mC')} \iota_n(\mB')} \ar{d}{} \\
\iota_n(\mA) \times \iota_n(\mB)  \ar{r}[swap]{} & \iota_n(\mA') \times \iota_n(\mB').
\end{tikzcd}
$$

So we can assume that $\mA,\mB,\mC,\mA',\mB',\mC'$ are $n$-categories for some $\n \geq 0$ and prove the 
statement by induction on $n \geq 0$.
The case $n=0$ is trivial because $\mA,\mB,\mC,\mA',\mB',\mC'$ are spaces. 
We assume the statement holds for $n-1$ and we prove the statement for $n.$

Let $(A, B, F(A) \to G(B)), (A', B', F(A') \to G(B')) \in \mA {\bar{\vec{\times}}}_{\mC} \mB.$ 
By \ref{homs2} the induced functor $$\Mor_{\mA {\bar{\vec{\times}}}_{\mC} \mB}((A, B, F(A) \to G(B)), (A', B', F(A') \to G(B'))) \to \Mor_\mA(A,A') \times \Mor_\mB(B,B')$$
identifies with the functor 
$$ {\Mor_\mA(A,A') {{\bar{\vec{\times}}}}_{\Mor_\mC(F(A),G(B'))}} \Mor_\mB(B,B') \to \Mor_\mA(A,A') \times \Mor_\mB(B,B').$$
The latter is an antibifibration by induction hypothesis since 
the morphism $\infty$-categories of $\mA,\mB,\mC$ are $n-1$-categories.

Moreover the commutative square \ref{sqfibb} induces on morphism $\infty$-categories the canonical commutative square

\begin{equation*}
\!\!\begin{tikzcd}
\Mor_\mA(A,A') \!\!\!\!\!\underset{\Mor_\mC(F(A),G(B'))}{{\bar{\vec{\times}}} } \!\!\!\!\!\Mor_\mB(B,B') \ar{r}{} \ar{d}[swap]{} &  \Mor_{\mA'}(\alpha(A),\alpha(A')) \!\!\!\!\!\underset{\Mor_{\mC'}(F'(\alpha(A)),G'(\beta(B'))}{{\bar{\vec{\times}}}}  \!\!\!\!\!\Mor_{\mB'}(\beta(B),\beta(B')) \ar{d}{} \\
\Mor_{\mA}(A,A') \times \Mor_{\mB}(B,B')  \ar{r}[swap]{} & \Mor_{\mA'}(\alpha(A),\alpha(A')) \times \Mor_{\mB'}(\beta(B),\beta(B')). 
\end{tikzcd}
\end{equation*}
The latter is a map of antibifibrations by induction hypothesis.
This proves (2) in view of the first part of the proof, so it remains to complete the proof of (1).

Let $(A' \to A, B' \to B, \sigma), (X \to X', Y \to Y', \rho) $ be morphisms in the $n$-category $\mA {\bar{\vec{\times}}}_{\mC} \mB, $
where $\sigma, \rho$ are antioriented squares
\begin{equation*}
\begin{tikzcd}
F(A') \ar{r}{} \ar{d}[swap]{} & G(B') \ar[double]{dl}{} \ar{d}{} \\
F(A)  \ar{r}[swap]{} & G(B),
\end{tikzcd}
\qquad
\begin{tikzcd}
F(X) \ar{r}{} \ar{d}[swap]{} & G(Y) \ar[double]{dl}{} \ar{d}{} \\
F(X')  \ar{r}[swap]{} & G(Y').
\end{tikzcd}
\end{equation*}
We must show that the induced functor $$ \Mor_{\mA {\bar{\vec{\times}}}_{\mC} \mB}((A, B, F(A) \to G(B)), (X, Y, F(X) \to G(Y))) \to $$$$ \Mor_{\mA {\bar{\vec{\times}}}_{\mC} \mB}((A', B', F(A') \to G(B')), (X', Y', F(X') \to G(Y'))) $$ preserves cocartesian morphisms.
By \cref{homs2} the latter functor identifies with the functor
$$ {\Mor_\mA(A,X) {{\bar{\vec{\times}}}}_{\Mor_\mC(F(A), G(Y))}} \Mor_\mB(B,Y)
\to {\Mor_\mA(A',X') {{\bar{\vec{\times}}}}_{\Mor_\mC(F(A'), G(Y'))}} \Mor_\mB(B',Y') $$
induced by the antioriented squares
\begin{equation*}
\begin{tikzcd}
\Mor_\mA(A,X) \ar{r}{} \ar{d}[swap]{} & \Mor_\mA(A',X'){} \ar{d}{} \\
\Mor_\mC(F(A), F(X))  \ar{r}[swap]{} \ar{d}{} & \Mor_\mC(F(A'), F(X')) \ar{d}{}  \ar[double]{dl}
\\
\Mor_\mC(F(A), G(Y))  \ar{r}[swap]{} & \Mor_\mC(F(A'), G(Y')),
\end{tikzcd}
\begin{tikzcd}
\Mor_\mB(B,Y) \ar{r}{} \ar{d}[swap]{} & \Mor_\mB(B',Y') {} \ar{d}{} \\
\Mor_\mC(G(B), G(Y)) \ar{d}{} \ar{r}[swap]{} & \Mor_\mC(G(B'), G(Y')) \ar{d}{}
\\
\Mor_\mC(F(A), G(Y))  \ar[double]{ur} \ar{r}[swap]{} & \Mor_\mC(F(A'), G(Y')).
\end{tikzcd}
\end{equation*}
The latter functor preserves cocartesian morphisms
by induction hypothesis. 
\end{proof}

\begin{corollary}\label{targetFib}Let $F: \mA \to \mC$ and  $G: \mB \to \mC$ be functors.

\begin{enumerate}[\normalfont(1)]\setlength{\itemsep}{-2pt}
\item The functor $ \mA {\vec{\times}}_\mC \mB \to \mA \times \mB $
is a bifibration.

\item Let
$$
\begin{tikzcd}
\mA \ar{r}{\alpha} \ar{d}[swap]{} & \mA' \ar[double]{dl}{} \ar{d}{} \\
\mC  \ar{r}[swap]{\gamma} & \mC',
\end{tikzcd}
\qquad
\begin{tikzcd}
\mB \ar{r}{\beta} \ar{d}[swap]{} & \mB'  \ar{d}{} \\
\mC \ar[double]{ur}{} \ar{r}[swap]{\gamma} & \mC'
\end{tikzcd}
$$
be oriented squares of $\infty$-categories.
The induced commutative square
\begin{equation}\label{sqfibb2}
\begin{tikzcd}
\mA {\vec{\times}}_\mC \mB \ar{r}{} \ar{d}[swap]{} & \mA' {\vec{\times}}_{\mC'} \mB' \ar{d}{} \\
\mA \times \mB \ar{r}[swap]{} & \mA' \times \mB'
\end{tikzcd}
\end{equation}
is a map of bifibrations.
\end{enumerate}
\end{corollary}

\begin{corollary}\label{targetfibr}
Let $\mA, \mC$ be $\infty$-categories.
The functor $$\Fun^\oplax(S(\mA),\mC) \simeq \mC {\vec{\times}}_{\Fun^\oplax(\mA,\mC)} \mC \to \mC \times \mC $$
is a bifibration.
\end{corollary}

\begin{corollary}\label{targetfi2}
The functor $\Fun^{\oplax}(\bD^1,\mC)\to\mC\times\mC$ is a bifibration whose fiber over $(A,B)$ is $\Mor_\mC(A,B)$.
\end{corollary}

\subsection{Fibrations are stable under oriented pullback}

\cref{targetfi} and \cref{targetfi2} imply the following:

\begin{corollary}\label{slicecocart}
Let $\mC$ be an $\infty$-category and $X \in \mC.$

\begin{enumerate}[\normalfont(1)]\setlength{\itemsep}{-2pt}
\item The functor $ \mC_{//^\oplax X}\to \mC$ is a cartesian fibration. 

\item The functor $ \mC_{X //^\oplax}\to \mC$ is a cocartesian fibration.

\item The functor $ \mC_{//^\lax X}\to \mC$ is an anticartesian fibration. 

\item The functor $ \mC_{X //^\lax}\to \mC$ is an anticocartesian fibration.

\end{enumerate}
\end{corollary}

\begin{corollary}
Let $\phi: \mC \to \mD$ be a functor. A morphism $f: X \to Y$ in $\mC$ is
$\phi$-cartesian if and only if the following induced functor is an equivalence:
$$ \mC_{//^\oplax X}\to \mC_{//^\oplax Y} \times_{\mD_{// ^\oplax \phi(X)}} \mD_{//^\oplax \phi(Y)}. $$
\end{corollary}

\begin{proof}

Since the subcategory $\mathcal{C}\mathit{art} \subset \Fun(\bD^1,\infty\scat)$ is closed under small limits, by \cref{slicecocart} the functor $$ \mC_{//^\oplax X}\to \mC_{// ^\oplax Y} \times_{\mD_{// ^\oplax \phi(Y)}} \mD_{//  ^\oplax \phi(X)} $$ over
$ \mC \simeq \mC \times_\mD \mD$ is a map of cartesian fibrations over
$\mD$. Therefore this functor is an equivalence by \ref{fiberwiseeq} if it is fiberwise an equivalence. It induces on the fiber over any $Z \in \mC$
the functor $$ \Mor_\mC(Z,X)\to \Mor_\mC(Z,Y) \times_{\Mor_\mC(\phi(Z),\phi(Y))} \Mor_\mC(\phi(Z),\phi(X)) $$
\end{proof}

\begin{theorem}\label{laxfib} Let $n \geq 1. $

\begin{enumerate}[\normalfont(1)]\setlength{\itemsep}{-2pt}
\item For every maps of $n+1$-cocartesiam fibrations
\begin{equation}\label{degen}
\xymatrix{
\mA \ar[r]^F \ar[d]^\rho & \mC \ar[d]^\phi \\
\mE \ar[r] & \mD,
}
\qquad
\xymatrix{
\mB \ar[r]^G \ar[d]^\kappa & \mC \ar[d]^\phi \\
\mF \ar[r] & \mD,
}
\end{equation} the induced functor 
\begin{equation}\label{sqgj}
\mA {\bar{\vec{\times}}}_{\mC} \mB \to \mE {\bar{\vec{\times}}}_{\mD} \mF \end{equation}
is an $n$-cocartesian fibration.
A 1-morphism in $\mA {\bar{\vec{\times}}}_{\mC} \mB $ is cocartesian over $\mE {\bar{\vec{\times}}}_{\mD} \mF $ if and only if its image in $\mA$ is cocartesian over $\mE$, its image in $\mB$ is cocartesian over $\mF$ and its image in $\Fun^\lax(\bD^1,\mC)$
corresponds to an antioriented square in $\mC$ whose non-invertible 2-morphism is cartesian over $\mD.$

\item For every (anti)oriented commutative squares of $n+1$-cocartesian fibrations 
%corresponding to functors $\cube^2 \to \co\Cart: $

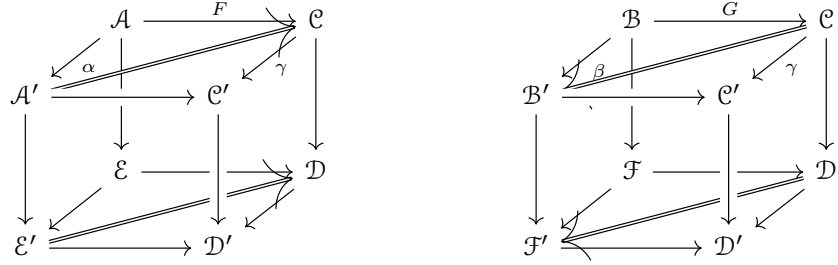
\qquad\qquad
\begin{tikzcd}[row sep=scriptsize, column sep=scriptsize]
& \mA \arrow{dl}{\alpha} \arrow{rr}{F} \arrow[dd] & & \mC \arrow{dl}{\gamma} \arrow[dd] \\ \mA' \ar[double]{rrru}{} \arrow[rr, crossing over] \arrow[dd] & & \mC' \\
& \mE \arrow[dl] \arrow[rr] & & \mD \arrow[dl] \\
\mE' \ar[double]{rrru}{} \arrow[rr] & & \mD' \arrow[from=uu, crossing over]
\end{tikzcd}
\qquad\qquad\qquad
\begin{tikzcd}[row sep=scriptsize, column sep=scriptsize]
& \mB \arrow{dl}{\beta} \arrow{rr}{G} \arrow[dd] & & \mC \ar[double]{llld}{} \arrow{dl}{\gamma} \arrow[dd] \\ \mB' \arrow[rr, crossing over] \arrow[dd] & & \mC' \\
& \mF \arrow[dl] \arrow[rr] & & \mD \ar[double]{llld}{} \arrow[dl] \\
\mF' \arrow[rr] & & \mD' \arrow[from=uu, crossing over]
\end{tikzcd}

the following induced commutative square is a map of $n$-cocartesian fibrations:
\begin{equation}\label{sqty}
\xymatrix{
\mA {\bar{\vec{\times}}}_{\mC} \mB \ar[r] \ar[d] & \mA' {\bar{\vec{\times}}}_{\mC'} \mB' \ar[d] \\
\mE {\bar{\vec{\times}}}_{\mD} \mF \ar[r] & \mE' {\bar{\vec{\times}}}_{\mD'} \mF'.
}
\end{equation}
\end{enumerate}

\end{theorem}

\begin{proof}

We prove the statements by induction on $n \geq 1.$
We start with proving (1) for $n=1.$
To prove (1) for $n=1$ it suffices to assume $\rho=\kappa=\phi$
and that the commutative squares \ref{degen} are degenerate.

This will imply that both commutative squares
\[
\begin{tikzcd}
\Fun^\lax(\bD^1,\mC) \ar{d} \ar{r}[swap]{} & \mC \times \mC \ar{d} \\
\Fun^\lax(\bD^1,\mD) \ar{r}[swap]{} & \mD \times \mD
\end{tikzcd}
\begin{tikzcd}
\mA \times \mB \ar{d} \ar{r}[swap]{} & \mC \times \mC \ar{d} \\
\mE \times \mF \ar{r}[swap]{} & \mD \times \mD
\end{tikzcd}
\]
are maps of 1-cocartesian fibrations. So also the functor \ref{sqgj} will be a 1-cocartesian fibration.

Let $f:X \to Y$ be a morphism in $\mC$ and consider the following antioriented square in $\mD:$
\[
\begin{tikzcd}
\phi(X) \ar{d}{\phi(f)} \ar{r}[swap]{} & A \ar[double]{dl}{} \ar{d}{} \\
\phi(Y) \ar{r}[swap]{} & B.
\end{tikzcd}
\]

Since $\phi$ is a 2-cocartesian fibration, the latter antioriented square lifts to an antioriented square 
\[
\begin{tikzcd}
X \ar{d}{f} \ar{r}[swap]{t} & X' \ar[double]{dl}{\tau} \ar{d}{g} \\
Y \ar{r}[swap]{r} & Y',
\end{tikzcd}
\]
where the morphisms $t: X \to X', r: Y \to Y'$ are $\phi$-cocartesian and
the 2-morphism $\tau$ is $\phi$-cartesian.

We show that for every morphism $ h: U \to V$ in $\mC$ the following commutative square is a pullback square
\[
\xymatrix{
\Mor_{\Fun^\lax(\bD^1,\mC)}(g,h) \ar[r] \ar[d] & \Mor_{\Fun^\lax(\bD^1,\mD)}(\phi(g),\phi(h)) \ar[d] \\
\Mor_{\Fun^\lax(\bD^1,\mC)}(f,h) \ar[r] & \Mor_{\Fun^\lax(\bD^1,\mD)}(\phi(f),\phi(h)).
}
\]
By \cref{homs2} the latter identifies with the following commutative square
\[
\xymatrix{
\Mor_\mC(X',U) \underset{\Mor_\mC(X', V)}{{\overset{\bar{\to}}{{\times}}}} \Mor_\mC(Y', V) \ar[r] \ar[d] & \Mor_\mD(\phi(X'), \phi(U)) {\bar{\vec{\times}}}_{\Mor_\mD(\phi(X'), \phi(V))} \Mor_\mD(\phi(Y'),\phi(V)) \ar[d] \\
\Mor_\mC(X, U) {\bar{\vec{\times}}}_{\Mor_\mC(X, V)} \Mor_\mC(Y,V) \ar[r] & \Mor_\mD(\phi(X), \phi(U)) {\bar{\vec{\times}}}_{\Mor_\mD(\phi(X), \phi(V))} \Mor_\mD(\phi(Y),\phi(V)).
}
\]
By \cref{targetfi} it suffices to see that the latter commutative square
induces a pullback square on the fiber over $\alpha \in \Mor_\mC(X',U), \beta \in \Mor_\mC(Y', V)$ and its images under the projections of the pullback.
This induced commutative square is the outer square in the following commutative diagram:
\[
\xymatrix{
\Mor_{\Mor_\mC(X', V)}(h \circ \alpha, \beta \circ g) \ar[r] \ar[d] & \Mor_{\Mor_\mD(\phi(X'),\phi(V))}(\phi(h) \circ \phi(\alpha), \phi(\beta) \circ \phi(g)) \ar[d] \\
\Mor_{\Mor_\mC(X, V)}(h \circ\alpha \circ t, \beta \circ g \circ t) \ar[r]\ar[d] & \Mor_{\Mor_\mD(\phi(X), \phi(V))}(\phi(h) \circ \phi(\alpha) \circ \phi(t), \phi(\beta) \circ\phi(g) \circ \phi(t))\ar[d]
\\
\Mor_{\Mor_\mC(X, V)}(h \circ \alpha \circ t, \beta \circ r \circ f) \ar[r] & \Mor_{\Mor_\mD(\phi(X), \phi(V))}(\phi(h) \circ \phi(\alpha)\circ \phi(t), \phi(\beta) \circ \phi(r) \circ \phi(f)).
}
\]
The upper commutative square is a pullback square since the morphism $t: X \to X'$ is $\phi$-cocartesian so that the commutative square
\[
\xymatrix{
\Mor_\mC(X', V) \ar[r] \ar[d] & \Mor_\mD(\phi(X'),\phi(V)) \ar[d] \\
\Mor_\mC(X, V)\ar[r] & \Mor_\mD(\phi(X), \phi(V))
}
\]
is a pullback square.

The $\phi$-cartesian 2-morphism $ \tau: g \circ t \to r \circ f$ is 
sent by the functor $ \beta \circ (-):  \Mor_\mC(X', Y') \to \Mor_\mC(X', V)$ to a $\phi$-cartesian 2-morphism 
$\beta \tau: \beta \circ g \circ t \to \beta \circ r \circ f$.
Therefore the lower commutative square is a pullback square, too.
This proves (1) for $n=1$.

Next we prove (2) for $n=1.$
The top horizontal functor in \ref{sqty} 
sends an object $(A,B,F(A) \to G(B))$ to $$(\alpha(A),\beta(B),F'(\alpha(A)) \to \gamma(F(A)) \to \gamma(G(B)) \to G'(\beta(B))),$$
and sends a morphism 
$ (A \to A', B \to B', \sigma),$ where $\sigma$ is an antioriented square
\begin{equation}\label{orsquo}
\begin{tikzcd}
F(A) \ar{r}{} \ar{d}[swap]{} & G(B) \ar[double]{dl}{} \ar{d}{} \\
F(A')  \ar{r}[swap]{} & G(B')
\end{tikzcd}
\end{equation}
to the 1-morphism $ (\alpha(A) \to \alpha(A'), \beta(B) \to \beta(B'), \sigma'),$ where $\sigma'$ is an antioriented square
\begin{equation}\label{sqwis}
\begin{tikzcd}
F'(\alpha(A)) \ar{r}{} \ar{d}[swap]{} & \gamma(F(A)) \ar{r}{} \ar{d}[swap]{} & \gamma(G(B)) \ar[double]{dl}{} \ar{d}{} \ar{r}{} & G'(\beta(B)) \ar{d}{} \\
F'(\alpha(A')) \ar{r}[swap]{} & \gamma(F(A')) \ar{r}[swap]{} & \gamma(G(B')) \ar{r}[swap]{} & G'(\beta(B')).
\end{tikzcd}
\end{equation}
This proves that the top horizontal functor in \ref{sqty} 
is a map of 1-cocartesian fibrations:
Let
$ (A \to A', B \to B', \sigma)$ a morphism of $\mA {\bar{\vec{\times}}}_{\mC} \mB$
cocartesian over 
$\mE {\bar{\vec{\times}}}_{\mD} \mF.$ By the description of cocartesian 1-morphisms this means that the morphism $A \to A'$ of $\mA$
is cocartesian over $\mE,$ the morphism $B \to B'$ of $\mB$
is cocartesian over $\mF$ and the non-invertible 2-cell of the antioriented square $\sigma$ is cartesian over $\mD.$
Therefore the image $\alpha(A) \to \alpha(A')$ is cocartesian over $\mE'$, the image $\beta(B) \to \beta(B')$ is cocartesian over $\mF'$
and the non-invertible 2-cell of the antioriented square $\gamma(\sigma)$ in $\mC'$ is cartesian over $\mD'.$
So by stability under whiskering the non-invertible 2-cell of the antioriented square \ref{sqwis} in $\mC'$ is cartesian over $\mD'.$

We prove the induction step. Let $n \geq 1.$ We assume that (1) and (2) hold for $n$ and prove that (1) and (2) hold for $n+1.$

We first consider (1). Let the assumptions of (1) for $n+1$ be satisfied.
By what we have proven, the functor \ref{sqgj} is a 1-cocartesian fibration.
We prove next that for every $$ (A, B, F(A) \to G(B)), (A', B', F(A') \to G(B')) \in \mA {\bar{\vec{\times}}}_{\mC} \mB$$ 
the induced functor $$\Mor_{\mA {\bar{\vec{\times}}}_{\mC} \mB}((A, B, F(A) \to G(B)), (A', B', F(A') \to G(B'))) \to $$$$ \Mor_{\mE {\bar{\vec{\times}}}_\mD \mF}((\rho(A), \kappa(B), \phi(F(A)) \to \phi(G(B))), (\rho(A'), \kappa(B'), \phi(F(A')) \to \phi(G(B'))))$$
is an $n$-cartesian fibration.

By \ref{homs2} the latter functor identifies with the functor 
\begin{equation}\label{oplk}
\Mor_\mA(A,A') \underset{\Mor_\mC(F(A),G(B'))}{{\bar{\vec{\times}}}} \Mor_\mB(B,B') \to \Mor_\mE(\rho(A),\rho(A')) \underset{{\Mor_\mD(\gamma(F(A)),\gamma(G(B')))}}{{\bar{\vec{\times}}}}\Mor_\mF(\kappa(B),\kappa(B')).\end{equation}
This functor is induced by the maps of $n+1$-cartesian fibrations

\begin{equation*}
\xymatrix{
\Mor_\mA(A,A') \ar[r] \ar[d] & \Mor_\mC(F(A),G(B')) \ar[d] \\
\Mor_\mE(\rho(A),\rho(A')) \ar[r] & \Mor_\mD(\gamma(F(A)),\gamma(G(B'))),
}
\end{equation*}
and
\begin{equation*}
\xymatrix{
\Mor_\mB(B,B') \ar[r] \ar[d] &\Mor_\mC(F(A),G(B')) \ar[d] \\
\Mor_\mF(\kappa(B),\kappa(B')) \ar[r] & \Mor_\mD(\gamma(F(A)),\gamma(G(B'))).
}
\end{equation*}

Hence the functor \ref{oplk} is an $n$-cartesian fibration because we assume (the dual of) (1) for $n.$

\vspace{1mm}

Let $$(A' \to A, B' \to B, \sigma) : (A, B, F(A) \to G(B)) \to (A', B', F(A') \to G(B'))$$
and $$(X \to X', Y \to Y', \rho)  : (X, Y, F(X) \to G(Y)) \to (X', Y', F(X') \to G(Y')), $$
where $\sigma, \rho$ are antioriented squares
\begin{equation*}
\begin{tikzcd}
F(A') \ar{r}{} \ar{d}[swap]{} & G(B') \ar[double]{dl}{} \ar{d}{} \\
F(A)  \ar{r}[swap]{} & G(B),
\end{tikzcd}
\begin{tikzcd}
F(X) \ar{r}{} \ar{d}[swap]{} & G(Y) \ar[double]{dl}{} \ar{d}{} \\
F(X')  \ar{r}[swap]{} & G(Y').
\end{tikzcd}
\end{equation*}
To complete the proof of (1) it remains to see that the induced functor $$ \Mor_{\mA {\bar{\vec{\times}}}_{\mC} \mB}((A, B, F(A) \to G(B)), (X, Y, F(X) \to G(Y))) \to $$$$ \Mor_{\mA {\bar{\vec{\times}}}_{\mC} \mB}((A', B', F(A') \to G(B')), (X', Y', F(X') \to G(Y'))) $$ is a map of $n$-cartesian fibrations.

By \cref{homs2} the latter functor identifies with the functor
\begin{equation}\label{erhm}
\Mor_\mA(A,X) {\bar{\vec{\times}}}_{\Mor_\mC(F(A), G(Y))} \Mor_\mB(B,Y)
\to \Mor_\mA(A',X') {\bar{\vec{\times}}}_{\Mor_\mC(F(A'), G(Y'))} \Mor_\mB(B',Y') \end{equation}
induced by the antioriented squares
\begin{equation*}
\begin{tikzcd}
\Mor_\mA(A,X) \ar{r}{} \ar{d}[swap]{} & \Mor_\mA(A',X'){} \ar{d}{} \\
\Mor_\mC(F(A), F(X))  \ar{r}[swap]{} \ar{d}{} & \Mor_\mC(F(A'), F(X')) \ar{d}{}  \ar[double]{dl}
\\
\Mor_\mC(F(A), G(Y))  \ar{r}[swap]{} & \Mor_\mC(F(A'), G(Y')),
\end{tikzcd}
\begin{tikzcd}
\Mor_\mB(B,Y) \ar{r}{} \ar{d}[swap]{} & \Mor_\mB(B',Y') {} \ar{d}{} \\
\Mor_\mC(G(B), G(Y)) \ar{d}{} \ar{r}[swap]{} & \Mor_\mC(G(B'), G(Y')) \ar{d}{}
\\
\Mor_\mC(F(A), G(Y))  \ar[double]{ur} \ar{r}[swap]{} & \Mor_\mC(F(A'), G(Y')).
\end{tikzcd}
\end{equation*}
The functor \ref{erhm} is a map of $n$-cartesian fibrations since we assume (the dual of) (2) for $n$.
This proves (1).

We complete the proof by showing (2) for $n+1.$
We have already proven (2) for $n=1.$
So the commutative square \ref{sqty}
is in particular a map of 1-cocartesian fibrations.
So it remains to see that the commutative square \ref{sqty}
induces on morphism $\infty$-categories a map of $n$-cartesian fibrations.
The commutative square \ref{sqty} induces on morphism $\infty$-categories between $(A, B, F(A) \to G(B)), (A', B', F(A') \to G(B')) \in \mA {\bar{\vec{\times}}}_{\mC} \mB$ the canonical functor
\begin{equation}\label{jhg}
\Mor_\mA(A,A')\!\!\!\!\!\underset{\Mor_\mC(F(A),G(B'))}{{\bar{\vec{\times}}}} \Mor_\mB(B,B') \to \Mor_{\mA'}(\alpha(A),\alpha(A'))\!\!\!\!\!\underset{\Mor_{\mC'}(F'(\alpha(A)),G'(\beta(B'))}{{\bar{\vec{\times}}}} \!\!\!\!\!\Mor_{\mB'}(\beta(B),\beta(B')) 
\end{equation}
over the functor 
$$
\Mor_\mE(\rho(A),\rho(A'))\!\!\!\!\!\underset{\Mor_\mD(\phi F(A),\phi G(B'))}{{\bar{\vec{\times}}}} \!\!\!\!\!\Mor_\mF(\kappa(B),\kappa(B')) \to $$$$ \Mor_{\mE'}(\rho'\alpha(A),\rho'\alpha(A'))\!\!\!\!\!\underset{\Mor_{\mD'}(\phi' F'(\alpha(A)),\phi' G'(\beta(B'))}{{\bar{\vec{\times}}}} \!\!\!\!\!\Mor_{\mF'}(\kappa'\beta(B),\kappa'\beta(B')). 
$$ 
The functor \ref{jhg} is a map of $n$-cartesian fibrations since we assume (the dual of) (2) for $n.$
\end{proof}

\begin{corollary}\label{slicecocar}

Let $ n \geq 1$.

\begin{enumerate}[\normalfont(1)]\setlength{\itemsep}{-2pt}
\item For every $n+1$-cocartesian ($n+1$-cartesian fibration) $\phi: \mC \to \mD$
and $X \in \mC$ the induced functors $$ \mC_{//^\lax X} \to \mD_{//^\lax \phi(X) }, \ \mC_{X //^\lax} \to \mD_{\phi(X) //^\lax} $$
are $n$-cocartesian fibrations ($n$-cartesian fibrations).

\item For every $n+1$-anticocartesian fibration ($n+1$-anticartesian fibration) $\phi: \mC \to \mD$
and $X \in \mC$ the induced functors $$ \mC_{//^\oplax X} \to \mD_{//^\oplax \phi(X) }, \ \mC_{X //^\oplax} \to \mD_{\phi(X) //^\oplax} $$
are $n$-anticocartesian fibrations ($n$-anticartesian fibrations).

\end{enumerate}
    
\end{corollary}

\begin{theorem}\label{indufibr}

Let $\mB$ be an $\infty$-category and $\phi: \mC \to \mD$ an anticocartesian fibration. The induced functor
$$\psi: \Fun^\oplax(\mB, \mC) \to \Fun^\oplax(\mB, \mD) $$
is an anticocartesian fibration.
An $\n$-morphism of $\Fun^\oplax(\mB, \mC)$ is $\psi$-cocartesian if and only if for every $m$-morphism in $\mB$ for $m \geq 0$ the associated
$n +m$-morphism is $\phi$-cocartesian. 
    
\end{theorem}

\begin{proof}

By cubical density every $\infty$-category $\mB$ is the colimit
$\colim_{\cube^\ell \to \mB} \cube^\ell$ so that 
$\psi: \Fun^\oplax(\mB, \mC) \to \Fun^\oplax(\mB, \mD) $
is the limit 
$$ \lim_{\cube^\ell \to \mB} \Fun^\oplax(\cube^\ell, \mC) \to \lim_{\cube^\ell \to \mB} \Fun^\oplax(\cube^\ell, \mD) $$
in $\Fun(\bD^1,\infty\Cat)$. Since the subcategory $\coCart \subset \Fun(\bD^1,\infty\Cat) $ is closed under small limits,
by the description of $\psi$-cocartesian morphisms we can reduce to the case that $\mB$ is an oriented $\ell$-cube for $\ell \geq 0.$

We prove the statement by induction on $\ell \geq 0.$

For $\ell=0$ there is nothing to show.
For $\ell=1$ the statement follows from the dual of \cref{laxfib}.
Let $\ell \geq 1$ and we assume the statement holds for $\ell.$
Since $\cube^\ell $ is an $\ell$-category, the statement for $\ell $ is equivalent to the following one: The induced functor
$$\psi': \Fun^\oplax(\cube^{\ell}, \mC) \to \Fun^\oplax(\cube^{\ell}, \mD) $$
is an anticocartesian fibration.
Moreover an $\n$-morphism of $\Fun^\oplax(\cube^{\ell}, \mC)$ is $\psi'$-cocartesian if and only if for every $m$-morphism in $\cube^{\ell}$ for $m \leq \ell$ the associated $n +m$-morphism is $\phi$-cocartesian. 
We simplify this statement.
We say that a $m$-morphism in $\cube^{\ell}$ for $m \leq \ell$ 
is atomic if it factors as $\bD^m \to \cube^m \to \cube^\ell$,
where the first functor takes the unique $m$-cell and the second functor is a tensor product of functors of the form $\bD^i \to \bD^1$
for $i \leq 1.$
Every $m$-morphism in $\cube^{\ell}$ for $m \leq \ell$ is a composite of atomic morphisms of dimension smaller or equal $m$.
Since cocartesian morphisms are closed under composition, the statement for $\ell$ is equivalent to the following one:
The induced functor
$$\psi': \Fun^\oplax(\cube^{\ell}, \mC) \to \Fun^\oplax(\cube^{\ell}, \mD) $$
is an anticocartesian fibration.
Moreover an $\n$-morphism of $\Fun^\oplax(\cube^{\ell}, \mC)$ is $\psi'$-cocartesian if and only if for every atomic $m$-morphism in $\cube^{\ell}$ for $m \leq \ell$ the associated $n +m$-morphism is $\phi$-cocartesian. 

We observe that for every atomic $m$-morphism in $\cube^{\ell}$ for $m \leq \ell$ and $i \leq 1$ the induced functor
$\bD^{i+m} \to \bD^i \boxtimes \bD^m \to \bD^1 \boxtimes \cube^\ell = \cube^{\ell+1} $ is again atomic, where the first functor takes the unique $i+m$-morphism.
Conversely, by definition of atomic morphism, every atomic $m+1$-morphism of $\cube^{\ell+1}$ arises this way from some atomic $m$-morphism in $\cube^{\ell}$.

The statement for $\ell=1$ and the induction hypothesis imply that the induced functor
$$\psi: \Fun^\oplax(\cube^{\ell+1}, \mC) \simeq \Fun^\oplax(\cube^1, \Fun^\oplax(\cube^{\ell}, \mC)) \to \Fun^\oplax(\cube^1,\Fun^\oplax(\cube^{\ell}, \mD)) \simeq \Fun^\oplax(\cube^{\ell+1}, \mD) $$
is an anticocartesian fibration.
Moreover an $\n$-morphism of $$\Fun^\oplax(\cube^{\ell+1}, \mC) \simeq \Fun^\oplax(\cube^1, \Fun^\oplax(\cube^{\ell}, \mC)) $$ is $\psi$-cocartesian if and only if for every $i$ morphism of $\bD^1$ for $ i \leq 1 $ the associated $n + i$-morphism is $\psi' $-cocartesian.
So by induction hypothesis an $\n$-morphism of $\Fun^\oplax(\cube^{\ell+1}, \mC) \simeq \Fun^\oplax(\cube^1, \Fun^\oplax(\cube^{\ell}, \mC)) $ is $\psi$-cocartesian if and only if for every $i$-morphism of $\bD^1$ for $ i \leq 1 $ and atomic $m$-morphism in $\cube^{\ell}$ for $m \leq \ell$ the associated $(n +i) + m $-morphism is $\phi$-cocartesian. 
By the remark above about atomic morphisms,
the latter condition is equivalent to say that an $\n$-morphism of $\Fun^\oplax(\cube^{\ell+1}, \mC) \simeq \Fun^\oplax(\cube^1, \Fun^\oplax(\cube^{\ell}, \mC)) $ is $\psi$-cocartesian if and only if for every atomic $m$ morphism of $\cube^{\ell+1}$ for $m \leq \ell+1$ the associated $n + m $-morphism is $\phi$-cocartesian. 
\end{proof}

\begin{corollary}\label{induoplax}
Let $\mB$ be an $\infty$-category and $\phi: \mC \to \mD$ an anticartesian fibration. The induced functor
$$\psi: \Fun^\oplax(\mB, \mC) \to \Fun^\oplax(\mB, \mD) $$
is an anticartesian fibration.

\end{corollary}

\begin{corollary}\label{indulax}

Let $\mB$ be an $\infty$-category.

\begin{enumerate}[\normalfont(1)]\setlength{\itemsep}{-2pt}
   
\item Let $\mC \to \mD$ be a cocartesian fibration. The induced functor
$$\Fun^\lax(\mB, \mC) \to \Fun^\lax(\mB, \mD) $$
is a cocartesian fibration.
  
\item Let $\mC \to \mD$ be a cartesian fibration. The induced functor
$$\Fun^\lax(\mB, \mC) \to \Fun^\lax(\mB, \mD) $$
is a cartesian fibration.

\end{enumerate}
  
\end{corollary}

\subsection{Higher dimensional cartesian factorization}

We now consider factorizations into (co)cartesian and fiberwise morphisms.

\begin{notation}Let $X$ be an $\infty$-category and $\mE$ a collection of positive dimensional cells.
For every $A,B \in X$ let $\mE_{A,B} $ be the collection of positive dimensional cells of $\Mor_X(A,B)$ such that for every $n \geq 0$ an $n$-morphism of 
$\Mor_X(A,B)$ belongs to $\mE_{A,B} $ if and only if the corresponding $n+1$-morphism of $X$ belongs to $\mE.$
\end{notation}

\begin{definition}
Let $X$ be an $\infty$-category and $\mE$ a collection of positive dimensional cells. For every $n > 1$ we define by induction that $\mE$ is closed under the composition below dimension $n$:

\begin{enumerate}[\normalfont(1)]\setlength{\itemsep}{-2pt}
\item The collection $\mE$ is closed under the composition below dimension $2$ if $\mE$ contains all identity 1-morphisms, and 
the composite of every composable 1-morphisms in $\mE$ belongs to $\mE$.

\item The collection $\mE$ is closed under the composition below dimension $n+1$ if it is closed under the composition below dimension $2$, for every $A,B \in X$ the collection
$\mE_{A,B} $ of $\Mor_X(A,B) $ is closed under the composition below dimension $n $ and for every morphisms $f: A' \to A, g: B \to B' $ in $X$ the induced functor $\Mor_X(A,B) \to \Mor_X(A',B') $ sends $\mE_{A,B}$ to $\mE_{A',B'}$.

\item The collection $\mE$ is closed under the composition if it is closed under the composition below dimension $n $ for every $n > 1.$

\end{enumerate}

\end{definition}

\begin{example}\label{cocartex}

Let $\phi: X \to Y$ be a functor.
The collection of $\phi$-cocartesian cells is closed under the composition.
    
\end{example}

\begin{remark}\label{premark}
Let $X$ be an $\infty$-category and $\mE$ a collection of positive dimensional cells. Let $\phi: Y \to X$ be a functor and $\phi^{-1}(\mE)$ the collection of cells sent by $\phi$ to cells of $\mE.$ The collection $\phi^{-1}(\mE)$ is closed under the composition if $\mE$ is closed under the composition.
    
\end{remark}

\begin{proposition}\label{subcarz}
Let $1 \leq n \leq \infty. $
Let $X$ be an $n$-category and $\mE$ a collection of positive dimensional cells that is closed under the composition below dimension $n+1$.
There is a subcategory inclusion $Y \to X$ such that for every
$k \geq 0$ a $k$-morphism of $X$ belongs to $Y$ if and only if 
it belongs to $\mE.$
    
\end{proposition}

\begin{proof}

We first reduce to the case $n < \infty.$
Let $X$ be an $\infty$-category. If we have proven the statement for $m$-categories for every $m \geq 1,$ then for every $m \geq 1$ there is a subcategory $Y^m \to \iota_m(X)$ whose $k$-cells are precisely the $k$-cells of $\mE.$
Thus the functor $Y^m \to X $ factors through $Y^{m+1} \to X.$
So we obtain a subcategory $Y:= \colim_{m \geq 1} Y^m \to \colim_{m \geq 1} \iota_m(X) \simeq X $ whose $k$-cells are precisely the $k$-cells of $\mE.$

So we can assume that $n < \infty.$
We prove the statement by induction on $n \geq 1$.
For $n=1$ the statement is clear.

Let $n > 1$. We assume the statement holds for $n-1$.
Let $X$ be an $n$-category and $\mE$ a collection of positive dimensional cells that is closed under the composition below dimension $n+1$.
Then for every $A,B \in X$ the collection
$\mE_{A,B} $ of the $n-1$-category $ \Mor_X(A,B) $ is closed under the composition below dimension $n$ and so by induction hypothesis there is a subcategory 
$Y_{A,B} \to \Mor_X(A,B)$ whose $k$-cells are precisely the $k$-cells of 
$\mE_{A,B} $. So by \cref{subchar} there is a subcategory $Y \to X$
whose $k$-cells are precisely the $k$-cells of $\mE.$
\end{proof}

\begin{proposition}\label{inheri} Let $X$ be an $\infty$-category and $\mE$ a collection of positive dimensional cells.
For $\ell \geq 0$ let $\mE^\ell$ be the collection of cells of $\Fun^\lax(\cube^\ell, X)$ whose $n$-cells for $n > 0$ are the functors $\bD^n \to \Fun^\lax(\cube^\ell, X)$ corresponding to functors $\cube^\ell \boxtimes \bD^n \to X$ such that the composition 
$\bD^{\ell +n} \to \cube^\ell \boxtimes \bD^n \to X$ with the unique non-invertible $\ell+n$-cell of $\cube^\ell \boxtimes \bD^n $, belongs to $\mE.$

If $\mE$ is closed under the composition, then 
$\mE^\ell $ is closed under the composition.

\end{proposition}

\begin{proof}

We proceed by induction on $\ell \geq 0.$ 
For $\ell=0 $ there is nothing to show. Let $\ell \geq 0. $ We assume the statement holds for $\ell.$
We have $\mE^{\ell+1}= (\mE^\ell)^1$ because for every $n \geq 0$ the composition
$$ \bD^{\ell+n+1} \to \cube^\ell \boxtimes \bD^{n+1} \to \cube^\ell \boxtimes \bD^1 \boxtimes \bD^n $$ is the unique non-invertible cell.
Consequently, we can reduce to the case $\ell=1,$ which we prove in the following.
We set $ \mE':= \mE^1. $ 

The collection $\mE'$ is closed under the composition if and only if it is closed under the composition below dimension $n$ for every $n > 1.$
So it suffices to prove by induction on $n > 1$ that for every $\infty$-category $X$ and collection $\mE$ of positive dimensional cells, the collection $\mE'$ is closed under the composition below dimension $n$ if the collection $\mE$ is closed under the composition below dimension $n$.

We start with $n=2$. We assume that $\mE$ is closed under the composition below dimension $2.$ We have to verify that $\mE'$ contains all identity 1-morphisms of $\Fun^\lax(\bD^1, X) $ and that the composite of composable 1-morphisms in $\mE'$ belongs to $\mE'.$
By definition of $\mE'$ we find that every identity 1-morphism of $\Fun^\lax(\bD^1, X) $ belongs to $\mE'$. 
We prove that for every 1-morphisms $\alpha: f \to g, \beta: g \to h $ in $\mE'$ the composite $\beta \circ \alpha: f \to h$ in $\Fun^\lax(\bD^1, X)$ belongs to $\mE'$.
So we have to see that the 2-morphism in the outer lax commutative square
\[
\begin{tikzcd}
A \ar{d}{f} \ar{r} & U \ar[double]{dl}{}  \ar{r}{} \ar{d}[swap]{g} & T \ar[double]{dl}{} \ar{d}{h} \\
B \ar{r} & V  \ar{r}[swap]{} & Z
\end{tikzcd}
\]
belongs to $\mE$.
This follows immediately from the conditions in the definition of closedness under composition.

Let $n > 2.$ We assume the statement holds for $n-1.$ We prove the statement for $n.$
We have to verify that for every $f,f' \in \Fun^\lax(\bD^1, X)$ the collection $\mE'_{f,f'} $ is closed under the composition
below dimension $n-1$ and for every $ f' \to g', g \to f \in \Fun^\lax(\bD^1, X)$ the induced functor  
$$ \Mor_{\Fun^\lax(\bD^1, X)}(f,f') \to \Mor_{\Fun^\lax(\bD^1, X)}(g,g') $$ sends $\mE'_{f,f'} $ to $\mE'_{g,g'} $.

Let $f: A \to B ,f': A' \to B'$. The collection $\mE'_{f,f'} $ is the preimage of $(\mE_{A,B'})'$ under the functor 
$$ \Mor_{\Fun^\lax(\bD^1, X)}(f,f') \simeq \Mor_X(A,A') \times \Mor_X(B,B') \times_{\Mor_X(A,B') \times \Mor_X(A,B')} \Fun^\lax(\bD^1, \Mor_X(A,B')) \to $$$$ \Fun^\lax(\bD^1, \Mor_X(A,B')).$$

Hence by \cref{premark} the collection $\mE'_{f,f'} $ is closed under the composition 
below dimension $n-1$ 
if $ (\mE_{A,B'})' $ is closed under the composition of 
below dimension $n-1$.
Since $\mE$ is closed under the composition below dimension $n$, also $\mE_{A,B'} $ is closed under the composition below dimension $n-1$.
So by induction hypothesis, $ (\mE_{A,B'})' $ is closed under the composition below dimension $n-1$.

It remains to see that for every morphisms $\alpha: f' \to g', \beta: g \to f $ in $ \Fun^\lax(\bD^1, X)$ the induced functor   
$$ \Mor_{\Fun^\lax(\bD^1, X)}(f,f') \to \Mor_{\Fun^\lax(\bD^1, X)}(g,g') $$ sends $\mE'_{f,f'} $ to $\mE'_{g,g'} $.

This functor identifies with the induced functor
$$ \Mor_X(A,A') \times \Mor_X(B,B') \times_{\Mor_X(A,B') \times \Mor_X(A,B')} \Fun^\lax(\bD^1, \Mor_X(A,B')) \to $$$$ \Mor_X(U,U') \times \Mor_X(V,V') \times_{\Mor_X(U,V') \times \Mor_X(U,V')} \Fun^\lax(\bD^1, \Mor_X(U,V')).$$

This functor sends $ (p: A \to A', q:B \to B', \sigma: f' p \to q f) $ to $$ (\alpha_0 p \beta_0 : U \to U', \alpha_1 q \beta_1 :V \to V', g' \alpha_0 p \beta_0 \to \alpha_1 f' p \beta_0 \xrightarrow{\alpha_1 \sigma \beta_0} \alpha_1 q f \beta_0 \to \alpha_1 q \beta_1 g)$$
and so sends $\mE'_{f,f'} $ to $\mE'_{g,g'} $.

Hence it suffices to see that the induced functor   
$$ \Mor_{\Fun^\lax(\bD^1, X)}(f,f') \to \Mor_{\Fun^\lax(\bD^1, X)}(g,g') $$ induces a functor on morphism $\infty$-categories that preserves positive dimensional distinguished cells.
For every 1-morphisms 
$t, s: f \to f'$ in $\Fun^\lax(\bD^1, X) $ the induced functor
$$ \Mor_{\Mor_{\Fun^\lax(\bD^1, X)}(f,f')}(t,s) \to \Mor_{\Mor_{\Fun^\lax(\bD^1, X)}(g,g')}(\alpha t \beta , \alpha s \beta) $$
identifies with the induced functor

$$ \Mor_{\Mor_X(A,A')}(t_0,s_0) \times \Mor_{\Mor_X(B,B')}(t_1,s_1) \times_{\Mor_{\Mor_X(A,B')}(f' t_0,s_1 f)^{\times 2}} \Fun^\lax(\bD^1, \Mor_{\Mor_X(A,B')}(f' t_0,s_1 f)) $$
$$ \xrightarrow{\alpha_1 \circ (-) \circ \beta_0} \Mor_{\Mor_X(U,U')}(\alpha_0 t_0 \beta_0,\alpha_0 s_0 \beta_0) \times \Mor_{\Mor_X(V,V')}(\alpha_1 t_1 \beta_1,\alpha_1 s_1 \beta_1) \times_{\Mor_{\Mor_X(U,V')}(g' \alpha_0 t_0 \beta_0,\alpha_1 s_1 \beta_1 g)^{\times 2}} $$$$ \Fun^\lax(\bD^1, \Mor_{\Mor_X(U,V')}(g' \alpha_0 t_0 \beta_0,\alpha_1 s_1 \beta_1 g)). $$
This functor covers the functor
$$ \Fun^\lax(\bD^1, \Mor_{\Mor_X(A,B')}(f' t_0,s_1 f)) \to  \Fun^\lax(\bD^1, \Mor_{\Mor_X(U,V')}(g' \alpha_0 t_0 \beta_0,\alpha_1 s_1 \beta_1 g))$$
induced by the functor 
$$ \Mor_{\Mor_X(A,B')}(f' t_0,s_1 f) \xrightarrow{\alpha_1 \circ (-) \circ \beta_0} \Mor_{\Mor_X(U,V')}(\alpha_1 f' t_0 \beta_0, \alpha_1 s_1 f \beta_0) \to \Mor_{\Mor_X(U,V')}(g' \alpha_0 t_0 \beta_0,\alpha_1 s_1 \beta_1 g),$$
where the last functor is induced by the morphisms
$$ g' \alpha_0 t_0 \beta_0 \to \alpha_1 f' t_0 \beta_0, \alpha_1 s_1 f \beta_0 \to \alpha_1 s_1 \beta_1 g $$ in $\Mor_X(U,V').$
This functor preserves positive dimensional distinguished cells because the functor
$\alpha_1 \circ (-) \circ \beta_0: \Mor_X(A,B') \to \Mor_X(U,V')$ sends $\mE_{A,B'}$ to $\mE_{U,V'}$ and $\mE_{U,V'}$ is closed under the composition.
\end{proof}

\cref{cocartex}, \cref{inheri} and \cref{subcarz} imply the following:

\begin{proposition}

Let $\phi: \mC\to \mD$ be a functor.
There is a subcategory $$ \Fun^{\lax, \cocart}(\bD^1,\mC) \subset \Fun^{\lax}(\bD^1,\mC)$$ whose $n$-morphisms for $n \geq 0$
correspond to functors $\bD^1 \boxtimes \bD^n \to \mC $
such that the composition $$ \bD^{n+1} \to \bD^1 \boxtimes \bD^n \to \mC $$
with the functor assigning the unique non-invertible $n+1$-morphism of $\bD^1 \boxtimes \bD^n$
is a $\phi$-cocartesian $n+1$-morphism.

\end{proposition}

\begin{proposition}\label{recur}
Let $\phi: \mC \to \mD$ be a functor and $f: X \to Y, g: A \to B$ morphisms such that $f$ is $\phi$-cocartesian.
The induced functor $$ \kappa_{\phi}: \Fun^{\lax}(\bD^1,\mC) \to \mC \times_{\Fun^{\lax}(\{0\},\mD)} \Fun^{\lax}(\bD^1,\mD)$$
induces on morphism $\infty$-categories a functor
$$\Mor_{\Fun^{\lax}(\bD^1,\mC)}(f,g) \to 
\Mor_\mC(X,A) \times_{\Mor_\mD(\phi(X),\phi(A))} \Mor_{\Fun^\lax(\bD^1,\mD)}(\phi(f), \phi(g))$$
that is the pullback along the functors
$$ \Mor_\mC(X,A) \to \Mor_\mC(X,B), \
\Mor_\mD(\phi(Y),\phi(B)) \to \Mor_\mD(\phi(X),\phi(B)) $$ of the induced functor
$$ \kappa_{\phi_{X,B}}: \Fun^{\lax}(\bD^1,\Mor_\mC(X,B)) \to \Mor_\mC(X,B ) \times_{\Fun^\lax(\{0\},\Mor_\mD(\phi(X),\phi(B)))} \Fun^\lax(\bD^1,\Mor_\mD(\phi(X),\phi(B))). $$

\end{proposition}

\begin{proof}
By \cref{homs} the functor $\kappa_\phi $
identifies with the following functor:
$$\Mor_\mC(X,A) \times_{\Mor_\mC(X,B)} \Fun^{\lax,\cocart}(\bD^1,\Mor_\mC(X,B)) \times_{\Mor_\mC(X,B)} \Mor_\mC(Y,B) \to $$$$ \Mor_\mC(X,A) \times_{\Mor_\mD(\phi(X),\phi(A))} \Mor_\mD(\phi(X),\phi(A)) \times_{\Mor_\mD(\phi(X),\phi(B))} \Fun^\lax(\bD^1,\Mor_\mD(\phi(X),\phi(B))) $$$$\times_{\Mor_\mD(\phi(X),\phi(B))} \Mor_\mD(\phi(Y),\phi(B)) \simeq $$
$$\Mor_\mC(X,A) \times_{\Mor_\mD(\phi(X),\phi(B))} \Fun^\lax(\bD^1,\Mor_\mD(\phi(X),\phi(B))) \times_{\Mor_\mD(\phi(X),\phi(B))} \Mor_\mD(\phi(Y),\phi(B)). $$
The latter is the pullback along the functor
$ \Mor_\mC(X,A) \to \Mor_\mC(X,B) $ of the induced functor
$$ \theta: \Fun^{\lax,\cocart}(\bD^1,\Mor_\mC(X,B)) \times_{\Mor_\mC(X,B)} \Mor_\mC(Y,B) \to $$$$ \Mor_\mC(X,B ) \times_{\Mor_\mD(\phi(X),\phi(B))} \Fun^\lax(\bD^1,\Mor_\mD(\phi(X),\phi(B))) \times_{\Mor_\mD(\phi(X),\phi(B))} \Mor_\mD(\phi(Y),\phi(B)). $$

Since $f: X \to Y $ is $\phi$-cocartesian, the following induced functor
is an equivalence: $$\Mor_\mC(Y,B) \to \Mor_\mD(\phi(Y),\phi(B)) \times_{\Mor_\mD(\phi(X),\phi(B))} \Mor_\mC(X,B).$$
Hence $\theta$ is the pullback along the functor
$\Mor_\mD(\phi(Y),\phi(B)) \to \Mor_\mD(\phi(X),\phi(B)) $ of the functor
$$ \kappa_{\phi_{X,B}}: \Fun^{\lax, \cocart}(\bD^1,\Mor_\mC(X,B)) \to \Mor_\mC(X,B ) \times_{\Mor_\mD(\phi(X),\phi(B))} \Fun^\lax(\bD^1,\Mor_\mD(\phi(X),\phi(B))). $$
\end{proof}

For the next notation observe that for every $n \geq 0$ the $\infty$-category $\bD^1 \boxtimes \bD^n $ has a unique non-invertible $n+1$-morphism.

\begin{proposition}\label{laxcocar}
Let $\phi: \mC \to \mD$ be a functor. The following induced functor is an inclusion: $$\Fun^{\lax, \cocart}(\bD^1,\mC) \to \Fun^\lax(\bD^1,\mD) \times_{\Fun^\lax(\{0\},\mD)} \mC.$$ 
    
\end{proposition}

\begin{proof}
Since the oriented cubes are compact objects in $\infty\Cat$ that generate $\infty\Cat$ under small colimits, the category $\infty\Cat$ is compactly generated and the functor
$\Fun^\lax(\bD^1,-): \infty\Cat \to \infty\Cat $ preserves small filtered colimits. Compact generation implies that pullbacks commute with small filtered colimits in $\infty\Cat$.
Thus the induced functor of the statement factors as $$\Fun^{\lax, \cocart}(\bD^1,\mC) \simeq \colim_{n\geq 0}\Fun^{\lax, \cocart}(\bD^1,\iota_n(\mC)) \to $$$$ \colim_{n\geq 0} (\Fun^\lax(\bD^1,\iota_n(\mD)) \times_{\Fun^\lax(\{0\},\iota_n(\mD))} \iota_n(\mC)) \simeq 
\Fun^\lax(\bD^1,\mD) \times_{\Fun^\lax(\{0\},\mD)} \mC.$$ 

Since inclusions are stable under filtered colimits, 
we can reduce to show that for every $n \geq 0$
the statement holds for any functor $\phi: \mC \to \mD$ between
$n$-categories. We prove the latter by induction on $n \geq 0.$
For $n =0$ the functor of the statement identifies with the identity of $\mC.$ 
Let $\phi: \mC \to \mD$ be a functor between $n$-categories.
The functor of the statement induces on underlying spaces the induced map
$$\iota_0(\Fun^{\cocart}(\bD^1,\mC)) \to \iota_0(\Fun(\bD^1,\mD)) \times_{\iota_0(\Fun(\{0\},\mD))} \iota_0(\mC),$$
which agrees with the map on underlying spaces induced by the functor
$$\Fun^{\cocart}(\bD^1,\mC) \to \Fun(\bD^1,\mD) \times_{\Fun(\{0\},\mD)} \mC.$$
By \cref{enrfibchar} the latter functor is fully faithful
so that its induced map on underlying spaces is an embedding.
So it suffices to see that for every $\phi$-cocartesian morphisms 
$f: X \to Y, g: A \to B$ the following induced functor on morphism $\infty$-categories is an inclusion:
$$\Mor_{\Fun^{\lax, \cocart}(\bD^1,\mC)}(f,g) \to \Mor_\mC(X,A) \times_{\Mor_\mD(\phi(X),\phi(A))} \Mor_{\Fun^\lax(\bD^1,\mD)}(\phi(f), \phi(g)).$$
By \cref{recur} the latter functor is the pullback along the functors
$ \Mor_\mC(X,A) \to \Mor_\mC(X,B) $ and
$\Mor_\mD(\phi(Y),\phi(B)) \to \Mor_\mD(\phi(X),\phi(B)) $ of the induced functor
$$ \Fun^{\lax, \cocart}(\bD^1,\Mor_\mC(X,B)) \to \Mor_\mC(X,B ) \times_{\Mor_\mD(\phi(X),\phi(B))} \Fun^\lax(\bD^1,\Mor_\mD(\phi(X),\phi(B))).$$
The latter is an inclusion by induction hypothesis and the fact that
$\Mor_\mC(X,B) \to \Mor_\mD(\phi(X),\phi(B))$ is a functor between $n-1$-categories.
\end{proof}

\begin{definition}
Let $n \geq 0.$ We define by induction on $n \geq 0$ when a functor
$\phi:\mC \to \mD$ is $n$-full.

\begin{enumerate}[\normalfont(1)]\setlength{\itemsep}{-2pt}
\item A functor $\phi:\mC \to \mD$ is 0-full if it is essentially surjective.

\item A functor $\phi:\mC \to \mD$ is $n$-full if it is 0-full and for every $X,Y\in \mC$
the induced functor $\Mor_{\mC}(X,Y) \to \Mor_\mD(\phi(X),\phi(Y))$
is $n-1$-full.

\item A functor is full if it is $n$-full for every $n \geq 0.$

\end{enumerate}  

\end{definition}

\begin{lemma}Let $n \geq 0.$ A functor $\phi:\mC \to \mD$ is $n$-full if and only if for every $0 \leq m \leq n$ every commutative square
$$\begin{xy}
\xymatrix{
\partial\bD^m \ar[d]^{} \ar[r]
& \mC \ar[d]^\phi
\\ 
\bD^m \ar[r]& \mD}
\end{xy}$$
has a filler.
    
\end{lemma}

\begin{proof}

This follows immediately by induction on $n \geq 0$ from the fact that
every commutative square as in the statement for $1 \leq m \leq n+1$
admits a filler if and only if for every $X,Y \in \mC$ and $0 \leq m \leq n$ every commutative square
$$\begin{xy}
\xymatrix{
\partial\bD^m \ar[d]^{} \ar[r]
& \Mor_\mC(X,Y) \ar[d]^{\phi_{X,Y}}
\\ 
\bD^m \ar[r]& \Mor_\mD(\phi(X),\phi(Y))}
\end{xy}$$
has a filler.
\end{proof}

\begin{lemma}\label{eqil} 
A functor $\phi :\mC \to \mD$ is an equivalence if and only if
it is a full inclusion.   
\end{lemma}

\begin{proof}
Every equivalence is a full inclusion. We prove the converse.
A functor $\phi:\mC \to \mD$ is an equivalence if and only if for every
$n \geq 0$ the induced functor $\iota_n(\phi): \iota_n(\mC) \to \iota_n(\mD) $ is an equivalence.
We prove by induction on $n \geq 0 $ that for every full inclusion
$\phi: \mC \to \mD$ the functor 
$\iota_n(\phi): \iota_n(\mC) \to \iota_n(\mD) $ is an equivalence.
For every full inclusion
$\phi: \mC \to \mD$ the induced map 
$\iota_0(\phi): \iota_0(\mC) \to \iota_0(\mD) $ is an essentially surjective embedding and so an equivalence.
We assume the statement for $n$ holds.
Let
$\phi: \mC \to \mD$ be a full inclusion.
Since $\iota_0(\phi)$ is an equivalence, to see that the functor
$$\iota_{n+1}(\phi): \iota_{n+1}(\mC) \to \iota_{n+1}(\mD) $$ is an equivalence, it suffices to show that for every $X,Y \in \mC$
the induced functor $$ \Mor_{\iota_{n+1}(\mC)}(X,Y) \to \Mor_{\iota_{n+1}(\mD)}(\phi(X),\phi(Y)) $$ is an equivalence.
The latter identifies with the functor
$$ \iota_n(\Mor_{\mC}(X,Y)) \to \iota_n(\Mor_{\mD}(\phi(X),\phi(Y))),$$
which is an equivalence by induction hypothesis since the induced functor $$\Mor_{\mC}(X,Y) \to \Mor_{\mD}(\phi(X),\phi(Y)) $$
is a full inclusion.
\end{proof}

\begin{proposition}\label{cocfull}
Let $n \geq 0.$
A functor $\phi: \mC \to \mD$ is a $n+1$-anticocartesian fibration
if and only if the following induced functor is $n$-full:
$$\kappa_\phi: \Fun^{\lax, \cocart}(\bD^1,\mC) \to \Fun^\lax(\bD^1,\mD) \times_{\Fun^\lax(\{0\},\mD)} \mC.$$ 
\end{proposition}

\begin{proof}
We prove the statement by induction on $n \geq 0.$
For $n=0$ the statement is tautological.
We assume the statement holds for $n.$
A functor $\phi:\mC \to \mD$ is a $n+1$-anticocartesian fibration if and only if it is a 1-anticocartesian fibration and for every $X,Y \in \mC$
the induced functor $$\phi_{X,Y}: \Mor_\mC(X,Y) \to \Mor_\mD(\phi(X),\phi(Y))$$ is an $n$-anticocartesian fibration.
So by induction hypothesis a functor $\phi:\mC \to \mD$ is a $n+1$-anticocartesian fibration if and only if $\kappa_\phi$ is essentially surjective and for any $X,Y \in \mC$ the functor $\kappa_{\phi_{X,Y}}$ is $n$-full.
Since $n$-full functors are stable under base change,
\cref{recur} implies that for every $X,Y \in \mC$
the functor $\kappa_{\phi_{X,Y}}$ is $n$-full if and only if
the functor $\kappa_\phi$ induces on morphism $\infty$-categories $n$-full functors.  
\end{proof}

\begin{theorem}\label{precocart}
A functor $\phi: \mC \to \mD$ is an anticocartesian fibration if and only if 
the induced functor $$\Fun^{\lax, \cocart}(\bD^1,\mC) \to \Fun^\lax(\bD^1,\mD) \times_{\Fun^\lax(\{0\},\mD)} \mC $$ 
is an equivalence.
    
\end{theorem}

\begin{proof}
By \cref{laxcocar} the functor of (2) is an inclusion.
By Lemma \cref{eqil} a functor is an equivalence if and only if it is a full inclusion.
So it suffices to see that the functor of (2) is full if and only if $\phi$ is a cocartesian fibration.
This follows from \cref{cocfull}.
\end{proof}

\begin{corollary}\label{precocart2}
A functor $\phi: \mC \to \mD$ is a cocartesian fibration if and only if the induced functor $$\Fun^{\oplax, \cocart}(\bD^1,\mC) \to \Fun^\oplax(\bD^1,\mD) \times_{\Fun^\oplax(\{0\},\mD)} \mC $$ 
is an equivalence.

\end{corollary}

\subsection{Cartesian oriented simplices}

In the following we define by induction on $n \geq 0$ when an oriented $n$-simplex is cocartesian.

\begin{definition}Let $n \geq 1$ and $\phi: \mC \to \mD$ a functor.

\begin{enumerate}[\normalfont(1)]\setlength{\itemsep}{-2pt}
\item By convention any oriented 0-simplex of $\mC$ is $\phi$-cocartesian.

\item An oriented 1-simplex is $\phi$-cocartesian if it is a 
$\phi$-cocartesian morphism.

\item An oriented $n$-simplex $\alpha$ in $\mC$ is $\phi$-cocartesian if for every morphism $\alpha(n) \to Y$ in $\mC$ the image of the corresponding oriented $n-1$-simplex $\bDelta^{n-1}\to \mC_{//^\oplax \alpha(n)} $ 
in $ \mC_{//^\oplax Y}$ is cocartesian with respect to the functor
$ \mC_{//^\oplax Y} \to \mD_{//^\oplax Y}$
and in $ \mC$ is $\phi$-cocartesian.
\end{enumerate}

\end{definition}

\begin{definition}Let $n \geq 1$ and $\phi: \mC \to \mD$ a functor.

\begin{enumerate}[\normalfont(1)]\setlength{\itemsep}{-2pt}
\item By convention any antioriented 0-simplex of $\mC$ is $\phi$-cocartesian.
\item An antioriented 1-simplex is $\phi$-cocartesian if it is a 
$\phi$-cocartesian morphism.
\item An antioriented $n$-simplex $\alpha$ in $\mC$ is $\phi$-cocartesian if for every morphism $Y \to \alpha(0) $ the image of the corresponding antioriented $n-1$-simplex $(\bDelta^{n-1})^\co \to \mC_{\alpha(0)//^\oplax} $
in $\mC_{Y//^\oplax }$ is cocartesian with respect to the functor
$ \mC_{Y //^\oplax}  \to \mD_{Y //^\oplax}$
and in $ \mC$ is $\phi$-cocartesian.

\end{enumerate}

\end{definition}

\begin{theorem}\label{T1}
Let $1 \leq k \leq \infty.$
\begin{enumerate}[\normalfont(1)]\setlength{\itemsep}{-2pt}
\item A functor $\phi:\mC \to \mD$ is a $k$-anticocartesian fibration if and only if the following condition holds:
for every $1 \leq n \leq k$ every commutative square
\begin{equation}\label{filler0}
\begin{xy}
\xymatrix{
\bDelta^0 \star \bDelta^{n-2} \ar[d]_{\{0\}\star \bDelta^{n-2}} \ar[r]
& \mC \ar[d]^\phi
\\ 
\bDelta^1 \star \bDelta^{n-2} \ar[r] & \mD
}
\end{xy}\end{equation}
admits a filler by a $\phi$-cocartesian oriented $n$-simplex and every commutative square
\begin{equation}\label{filler1.2}
\begin{xy}
\xymatrix{(\bDelta^{n-2})^\co \bar{\star} (\bDelta^0)^\co \ar[d]_{(\bDelta^{n-2})^\co \bar{\star} \{0\}} \ar[r]
& \mC \ar[d]^\phi
\\ 
(\bDelta^{n-2})^\co \bar{\star} (\bDelta^1)^\co \ar[r] & \mD
}
\end{xy}\end{equation}
admits a filler by a $\phi$-cocartesian antioriented $n$-simplex.

\vspace{1mm}

\item A commutative square 
\begin{equation}\label{sqq}
\begin{xy}
\xymatrix{
\mC \ar[d]^\phi \ar[r]
& \mB \ar[d]^\psi
\\ 
\mD \ar[r] & \mE,
}
\end{xy}\end{equation}
in which $\phi$ and $\psi$ are $k$-anticocartesian fibrations, is a morphism of $k$-anticocartesian fibrations if and only if it preserves cocartesian oriented and antioriented simplices of dimension smaller or equal $k.$
\end{enumerate}

\end{theorem}

\begin{proof}
We prove the statement by induction on $ k \geq 1.$
For $\bk=1$ the statement is tautological and we prove the statement for $k >1$.

For any $k > 1 $, a 1-cocartesian fibration $\phi: \mC \to \mD$ 
satisfies the condition of the statement for $1 \leq n \leq k$ if and only if it satisfies the condition of the statement for every $1 < n \leq k$. The latter is equivalent to the following condition:
for every $X \in \mC$ every commutative square
\begin{equation}
\begin{xy}
\xymatrix{
\bDelta^0 \star \bDelta^{n-3} \ar[d]_{\{0\}\star \bDelta^{n-3}} \ar[r]
& \mC_{//^\oplax X} \ar[d]
\\ 
\bDelta^1 \star \bDelta^{n-3} \ar[r] & \mD_{//^\oplax \phi(X)}
}
\end{xy}\end{equation}
admits a filler by a $\phi$-cocartesian oriented $n-1$-simplex such that for every morphism $X \to X'$ in $\mC$
the functor $\mC_{//^\oplax X } \to \mC_{//^\oplax X'}$ sends the filler to a cocartesian oriented simplex, and for every $Y \in \mC$ every commutative square
\begin{equation}
\begin{xy}
\xymatrix{
(\bDelta^{n-3})^\co \overline{\star} (\bDelta^0)^\co \ar[d]_{(\bDelta^{n-3})^\co \overline{\star} \{0\}} \ar[r]
& \mC_{Y//^\oplax } \ar[d]
\\ 
(\bDelta^{n-3})^\co \overline{\star} (\bDelta^1)^\co \ar[r] & \mD_{\phi(Y)//^\oplax }
}
\end{xy}\end{equation}
admits a filler by a $\phi$-cocartesian antioriented $n-1$-simplex such that for every morphism $Y' \to Y$ in $\mC$
the functor $\mC_{Y //^\oplax  } \to \mC_{Y' //^\oplax }$ sends the filler to a cocartesian antioriented simplex.

By induction on $k \geq 1$ the latter condition is equivalent to the following one, which we refer to as (*): 
for every $X,Y \in \mC$ the functors $$\mC_{//^\oplax X} \to \mD_{//^\oplax \phi(X)}, \ \mC_{Y//^\oplax } \to \mD_{\phi(Y)//^\oplax } $$ are $k-1$-anticocartesian fibrations such that for all morphisms $X \to X', Y' \to Y $ the functors
$$\mC_{//^\oplax X } \to \mC_{//^\oplax X'}, \ \mC_{Y //^\oplax } \to \mC_{Y' //^\oplax } $$ are maps of $k-1$-anticocartesian fibrations.
Hence it suffices to show that for every $k \geq 1$ a 1-cocartesian fibration $\phi:\mC \to \mD$ is a $k$-anticocartesian fibration if and only if it satisfies condition (*).
The `only if' direction follows from \cref{slicecocar}. We prove the other direction.
Let $\phi:\mC \to \mD$ be a 1-cocartesian fibration that satisfies condition (*).
Then, in particular, for every $X,Y \in \mC$ the functor 
$$\Mor_\mC(Y,X) \to \Mor_\mD(\phi(Y), \phi(X)) $$ is a $k-1$-anticocartesian fibration and for every morphisms $X \to X', Y' \to Y $ in $\mC$ the induced functors
$$\Mor_\mC(Y,X) \to \Mor_\mC(Y,X'), \ \Mor_\mC(Y,X) \to \Mor_\mC(Y',X) $$ 
are maps of $k-1$-anticocartesian fibrations.
Using \cref{cocarto1}, we conclude that $\phi$ is a $k$-anticocartesian fibration.

It remains to see (2).
The `only if' direction of (2) follows immediately from \cref{slicecocart}
We prove the other direction.

A commutative square (\ref{sqq}) where $\phi$ and $\psi$ are $k$-anticocartesian fibrations, which preserves cocartesian oriented and antioriented simplices of dimension smaller or equal $k$, is a morphism of 1-cocartesian fibrations, and yields for every $X \in \mC $ lying over $Y \in \mB$ a commutative square \begin{equation}
\begin{xy}
\xymatrix{
\mC_{//^\oplax X} \ar[d] \ar[r]
& \mB_{//^\oplax Y} \ar[d]
\\ 
\mD_{//^\oplax \phi(X)} \ar[r] & \mE_{// ^\oplax \psi(Y)},
}
\end{xy}\end{equation}
which preserves cocartesian oriented and antioriented simplices of dimension smaller or equal $k-1.$
So by induction hypothesis the latter square is a morphism of 
$k-1$-anticocartesian fibrations.
Hence the latter square induces for every $Z \in \mC $ lying over $T \in \mB$ a commutative square 
\begin{equation}
\begin{xy}
\xymatrix{
\Mor_\mC(Z,X) \ar[d] \ar[r]
& \Mor_\mB(T,Y) \ar[d]
\\ 
\Mor_\mD(\phi(Z), \phi(X)) \ar[r] & \Mor_\mE(\psi(T),\psi(Y)),
}
\end{xy}\end{equation}
which is a map of $k-1$-anticocartesian fibrations.
It follows that the commutative square (\ref{sqq}) is a morphism of $k$-anticocartesian fibrations.
\end{proof}

\subsection{The free fibration}

\begin{definition}
Let $\phi: \mC \to \mD $ be a functor.
The enveloping cocartesian fibration or free cocartesian fibration
is the cocartesian fibration $$\Env(\mC):=\Fun^\oplax(\bD^1,\mD) \times_{\Fun^\oplax(\{0\},\mD)} \mC \to \Fun^\oplax(\{1\},\mD)\simeq \mD.$$

\end{definition}

\begin{remark}

The unique functor $\bD^1 \to \bD^0$ induces an inclusion
$ \mD \simeq \Fun^\oplax(\bD^0,\mD) \to \Fun^\oplax(\bD^1,\mD)$.
The latter gives rise to an inclusion
$$ \mC \simeq \mD \times_{\mD} \mC \to \Fun^\oplax(\bD^1,\mD) \times_{\Fun^\oplax(\{0\},\mD)} \mC = \Env_\mD(\mC) $$
over $\Fun^\oplax(\{1\},\mD).$

\end{remark}

\begin{remark}
For every $X \in \mD$ the fiber $\{X\} \times_\mD \Env(\mC)$
is canonically equivalent to $$ \mC \times_{\mD} \Fun^\oplax(\bD^1,\mD) \times_\mD \{X\} \simeq \mC \underset{\mD}{\vec{\times}} \{X\} .$$

\end{remark}

\begin{proposition}\label{envo}

\begin{enumerate}[\normalfont(1)]\setlength{\itemsep}{-2pt}

\item Let $\phi: \mC \to \mD$ be a functor and $\mE \to \mD $ a cocartesian fibration.
The inclusion $\iota: \mC \times_\mD \mE \subset \mC \underset{\mD}{\vec{\times}} \mE$ over $\mC \times \mE$ admits a left inverse
$\rho: \mC \underset{\mD}{\vec{\times}} \mE \to \mC \times_\mD \mE$ over $\mC \times \mE$ that is a map of cocartesian fibrations over $\mE.$

\item 
There is a commutative square
\[
\xymatrix{
\mC \underset{\mD}{\vec{\times}} \mE \ar[rr]^{} \ar[d]^{\rho} && \mC' \underset{\mD'}{\vec{\times}} \mE' \ar[d]^{\rho} \\
\mC \times_\mD \mE \ar[rr] && \mC' \times_{\mD'} \mE'.}
\]

\item For every functor $\mC \to \mD $ the following composition is the identity: 
$$ \mC \underset{\mD}{\vec{\times}} \mD \simeq (\mC \times_\mD \mD) \underset{\mD}{\vec{\times}} \mD \to (\mC \underset{\mD}{\vec{\times}} \mD) \underset{\mD}{\vec{\times}} \mD \to (\mC \underset{\mD}{\vec{\times}} \mD) \times_\mD \mD \simeq  \mC \underset{\mD}{\vec{\times}} \mD. $$
\end{enumerate}
    
\end{proposition}

\begin{proof}

(1): By \Cref{precocart} the following induced functor is an equivalence: $$\Fun^{\oplax, \cocart}(\bD^1,\mC) \to \mC \times_{\Fun^\lax(\{0\},\mD)} \Fun^\oplax(\bD^1,\mD). $$ 

Hence the following induced functor is an equivalence: $$\Fun^{\oplax, \cocart}(\bD^1,\mC) \times_{\Fun^\lax(\{1\},\mD)} \mE \to \mC \underset{\mD}{\vec{\times}} \mE = \mC \times_{\Fun^\lax(\{0\},\mD)} \Fun^\oplax(\bD^1,\mD) \times_{\Fun^\lax(\{1\},\mD)} \mE.$$

Let $\rho$ be the composition
$$ \mC \underset{\mD}{\vec{\times}} \mE \simeq \Fun^{\oplax, \cocart}(\bD^1,\mC) \times_{\Fun^\lax(\{1\},\mD)} \mE \to \mC \times_{\mD} \mE, $$
where the last functor evaluates at the target.

The composition $$\mC \times_{\mD} \mE \to \Fun^{\oplax, \cocart}(\bD^1,\mC) \times_{\Fun^\lax(\{1\},\mD)} \mE \simeq \mC \times_{\Fun^\lax(\{0\},\mD)} \Fun^\oplax(\bD^1,\mD) \times_{\Fun^\lax(\{1\},\mD)} \mE $$
induced by the diagonal functor $\mC \to \Fun^{\oplax, \cocart}(\bD^1,\mC)$, is the canonical inclusion
$$\mC \times_{\mD} \mE \subset \mC \times_{\Fun^\lax(\{0\},\mD)} \Fun^\oplax(\bD^1,\mD) \times_{\Fun^\lax(\{1\},\mD)} \mE.$$

Hence $$\mC \times_{\mD} \mE \subset \mC \underset{\mD}{\vec{\times}} \mE \simeq \Fun^{\oplax, \cocart}(\bD^1,\mC) \times_{\Fun^\lax(\{1\},\mD)} \mE $$ 
is the canonical inclusion so that
$$\rho \circ \iota: \mC \times_{\mD} \mE \subset \mC \underset{\mD}{\vec{\times}} \mE \simeq \Fun^{\oplax, \cocart}(\bD^1,\mC) \times_{\Fun^\lax(\{1\},\mD)} \mE \to \mC \times_{\mD} \mE $$ 
is the identity.

(2) follows from the construction of $\rho.$

We prove next that $\rho$ is a map of cocartesian fibrations over $\mE.$
By naturality of $\rho$ by (2) and since oriented pullbacks commute with small filtered colimits and every $\infty$-category $\mC$
is the sequential colimit of $\iota_0(\mC) \to \iota_1(\mC) \to ...,$
we can assume that $\mA,\mB,\mC$ are $n$-categories for some $n \geq 0.$
We prove by induction on $n \geq 0$ that $\rho$ is a map of cocartesian fibrations for every $n$-categories $\mA,\mB,\mC.$
The functor $\rho$ induces on morphism $\infty$-categories 

(3): The functor in (3) factors as 
$$ \Env(\mC) \xrightarrow{\Env(\iota)} \Env(\Env(\mC)) \simeq \Fun^{\oplax, \cocart}(\bD^1,\Env(\mC)) \xrightarrow{\ev_1} \Env(\mC). $$

For every $\infty$-category $\mB$ there is a canonical equivalence
$$ \Fun^\lax(\mB, \Env(\mC)) \simeq \Env(\Fun^\lax(\mB,\mC)) $$ over 
$\Fun^\lax(\mB,\mC)$.
Similarly, there is a canonical equivalence
$$ \Fun^\lax(\mB, \Fun^{\oplax}(\bD^1,\Env(\mC))) \simeq
\Fun^{\oplax}(\bD^1, \Fun^\lax(\mB, \Env(\mC))) \simeq 
\Fun^{\oplax}(\bD^1, \Env(\Fun^\lax(\mB,\mC))).$$

Let $\mE:=\Fun^\lax(\mB,\mC)$ and $ \mF:=\Fun^\lax(\mB,\mD) $
and $\psi: \mE \to \mF$ the functor induced by $\phi: \mC \to \mD.$ 
Let $\xi: \Env(\mC) \to \Fun^{\oplax}(\bD^1,\Env(\mC)) $
be the functor represented by the induced map of sets  
$$ \sigma: \pi_0(\iota_0(\Fun^\lax(\mB, \Env(\mC)))) \simeq \pi_0(\iota_0(\Env(\mE))) \to $$$$ \pi_0(\iota_0(\Fun^\lax(\mB, \Fun^{\oplax}(\bD^1,\Env(\mC))))) \simeq 
\pi_0(\iota_0(\Fun^{\oplax}(\bD^1, \Env(\mE)))) $$
that sends $(X, \psi(X) \to Y) $ to the morphism in 
$\Env(\mE)$ from $(X, \id: \psi(X) \to \psi(X)) $ to $(X, \psi(X) \to Y) $ given by the identity of $X$ and the commutative square in $\mF:$
\[
\xymatrix{
\psi(X) \ar[r]^\id \ar[d]^{\id} & \psi(X) \ar[d] \\
\psi(X) \ar[r] & Y.}
\]

Then $\Env(\mC) \xrightarrow{\xi} \Fun^{\oplax}(\bD^1,\Env(\mC)) \xrightarrow{\ev_1} \Env(\mC)$
is the identity.

We prove next that the composition
$$ \Env(\mC) \xrightarrow{\xi} \Fun^{\oplax}(\bD^1,\Env(\mC)) \xrightarrow{} \Env(\Env(\mC)) = \Fun^\oplax(\bD^1,\mD) \times_{\Fun^\lax(\{0\},\mD)} \Env(\mC) $$
is the functor $\Env(\iota).$
By the Yoneda-lemma it suffices to see that for every $\infty$-category $\mB$ the composition
$$ \pi_0(\iota_0(\Fun^\lax(\mB, \Env(\mC)))) \xrightarrow{\pi_0(\iota_0(\Fun^\lax(\mB, \xi)))}  
\pi_0(\iota_0(\Fun^\lax(\mB, \Fun^{\oplax}(\bD^1,\Env(\mC)))))$$$$
\xrightarrow{} \pi_0(\iota_0(\Fun^\lax(\mB,\Env(\Env(\mC))))) $$
is the map $ \pi_0(\iota_0(\Fun^\lax(\mB, \Env(\iota)))).$
The first map identifies with the map
$$\pi_0(\iota_0(\Env(\mE))) \xrightarrow{\sigma}
\pi_0(\iota_0(\Fun^{\oplax}(\bD^1, \Env(\mE)))) 
\to \pi_0(\iota_0(\Env(\Env(\mE))))$$
and the second map identifies with the map
$$\pi_0(\iota_0(\Env(\iota))): \pi_0(\iota_0(\Env(\mE))) \to \pi_0(\iota_0(\Env(\Env(\mE)))).$$
Both latter maps agree by definition of $\sigma.$
By construction the functor $\xi:\Env(\mC) \to \Fun^{\oplax}(\bD^1,\Env(\mC)) $ 
lands in the full subcategory $\Fun^{\oplax, \cocart}(\bD^1,\Env(\mC))$.
By what we have proven, the composition
$$ \Env(\mC) \xrightarrow{\xi} \Fun^{\oplax}(\bD^1,\Env(\mC)) \simeq \Env(\Env(\mC)) $$
is the functor $\Env(\iota).$
Hence the functor $$ \rho_{\Env(\mC)} \circ \Env(\iota) :\Env(\mC) \xrightarrow{\Env(\iota)} \Env(\Env(\mC)) \simeq \Fun^{\oplax}(\bD^1,\Env(\mC)) \xrightarrow{\ev_1} \Env(\mC) $$
is the functor $ \Env(\mC) \xrightarrow{\xi} \Fun^{\oplax}(\bD^1,\Env(\mC)) \xrightarrow{\ev_1} \Env(\mC), $
which we have proven to be the identity.
\end{proof}

\begin{notation}

Let $\mB \to \mD, \mC \to \mD$ be functors.
\begin{enumerate}[\normalfont(1)]\setlength{\itemsep}{-2pt}

\item Let $\Fun_\mD^\cocart(\mB,\mC) \subset \Fun_\mD(\mB,\mC) $
be the full subcategory of maps of cocartesian fibrations over $\mD.$

\item Let $\Fun_\mD^\cart(\mB,\mC) \subset \Fun_\mD(\mB,\mC) $
be the full subcategory of maps of cartesian fibrations over $\mD.$

\end{enumerate}

\end{notation}

\begin{theorem}\label{envelo}

Let $\phi: \mC \to \mD $ be a functor and
$\rho: \mB \to \mD $ a cocartesian fibration.
The induced functor
$$ \Fun^\cocart_{\Fun^\oplax(\{1\},\mD)}(\Env(\mC),\mB) \to \Fun_\mD(\mC,\mB) $$
is an equivalence.

\end{theorem}

\begin{proof}

For every $\infty$-category $\mA$ the induced map
$$ \Map_{\infty\Cat}(\mA,\Fun^{\cocart}_\mD(\Env(\mC),\mB)) \to \Map_{\infty\Cat}(\mA, \Fun_\mD(\mC,\mB))$$ identifies with the map
$$ \iota_0(\Fun^{\cocart}_\mD(\Env(\mC), \mD \times_{\Fun(\mA,\mD)}\Fun(\mA,\mB))) \to \iota_0(\Fun_\mD(\mC,\mD \times_{\Fun(\mA,\mD)}\Fun(\mA,\mB))).$$
By \cref{cocartexp} the functor $\mD \times_{\Fun(\mA,\mD)}\Fun(\mA,\mB) \to \mD$ is a cocartesian fibration.
So by the Yoneda-lemma it suffices to show that 
the functor $ \Fun^\cocart_\mD(\Env(\mC),\mB) \to \Fun_\mD(\mC,\mB) $ induces a bijection on equivalence classes.
We first prove surjectivity.
Let $F: \mC \to \mB$ be a functor over $\mD.$
Since $\rho: \mB \to \mD$ is a cocartesian fibration,
by \cref{envo} there is a map $\xi_\mB: \Env(\mB) \to \mB$ of cocartesian fibrations over $\mD$ such that the composition
$\mB \subset \Env(\mB) \xrightarrow{\xi_\mB} \mB$ is the identity.
The composition $\Env(\mC) \xrightarrow{\Env(F)} \Env(\mB) \xrightarrow{\xi_\mB} \mB$ of maps of cocartesian fibrations over $\mD$
restricts on $\mC$ to the composition 
$\mC \xrightarrow{F} \mB \subset \Env(\mB) \xrightarrow{\xi_\mB} \mB$, which is $F.$

We continue with proving injectivity.
Let $H: \Env(\mC) \to \mB$ be a map of cocartesian fibrations over $\mD.$
The composition $\Env(\mC) \xrightarrow{\Env(H_{\mid \mC})} \Env(\mB) \to \mB $ 
factors as $$\Env(\mC) \xrightarrow{\Env(\iota)} \Env(\Env(\mC)) \xrightarrow{\Env(H} \Env(\mB) \xrightarrow{\xi_\mB} \mB.$$
Since $H$ is a map of cocartesian fibrations over $\mD$, by \cref{envo} the latter functor over $\mD$ factors as 
$\Env(\mC) \xrightarrow{\Env(\iota)} \Env(\Env(\mC)) \xrightarrow{\xi_{\Env(\mC)}} \Env(\mC) \xrightarrow{H} \mB $.
Hence the statement follows from \cref{envo}.
\end{proof}

\begin{corollary}\label{envelo2}
Let $\phi: \mC \to \mD $ be a functor and
$\rho: \mB \to \mD $ a cartesian fibration.
The induced functor
$$ \Fun^\cart_{\Fun^\oplax(\{0\},\mD)}(\Env(\mC^\coop)^\coop,\mB) \to \Fun_\mD(\mC,\mB) $$
is an equivalence.
\end{corollary}

\begin{definition}Let $\mC$ be a small $\infty$-category.
Let $$ \infty\scat^\cocart_{/\mD}\subset \infty\scat_{/\mD} $$
be the subcategory of cocartesian fibrations over $\mC$ and maps of cocartesian fibrations over $\mD$.
\end{definition}

\begin{corollary}\label{huip}
The inclusion of $\infty$-categories $$\infty\scat^\cocart_{/\mD}\subset \infty\scat_{/\mD}$$
admits a left adjoint that sends $\mC \to \mD$ 
to $\Env(\mC) \to \mD$.
\end{corollary}

\subsection{The universal fibration}

\begin{proposition}

Evaluation at the target $\co\mathcal{C}\mathit{art} \to \infty\scat$
is a cartesian fibration.
  
\end{proposition}

\begin{proof}

By \cref{enrtargetfibr} evaluation at the target $\Fun(\bD^1,\infty\scat) \to \infty\scat$
is a 1-cartesian fibration.
By the definition of morphisms in $\co\mathcal{C}\mathit{art}$ this implies that
also the restriction $\co\mathcal{C}\mathit{art} \to \infty\scat$ is a 1-cartesian fibration.
\cref{cocartexp} implies that $\co\mathcal{C}\mathit{art} \to \infty\scat$ is enriched in $\co\Cart$
and so a cartesian fibration. 
\end{proof}

\begin{corollary}

Evaluation at the target $\mathcal{C}\mathit{art} \to \infty\scat$
is an anticartesian fibration.
  
\end{corollary}

\begin{proof}

The canonical equivalence of $\infty$-categories $$ (-)^{\coop} : (-)^{\coop}_!(\infty\scat) \simeq \infty\scat$$
gives rise to an equivalence
$$ (-)^{\coop}_!(\co\mathcal{C}\mathit{art}) \simeq \mathcal{C}\mathit{art}$$
over $ (-)^{\coop} : (-)^{\coop}_!(\infty\scat) \simeq \infty\scat.$
\end{proof}

For the following notation we use \cref{slicecocart}.
\begin{notation}

Let $$\mathcal{R}\mathit{ep}\mathcal{C}\mathit{art} \subset \mathcal{C}\mathit{art}$$ be the full subcategory of cartesian fibrations of the form
$\mC_{//^\oplax X} \to \mC $ for some small $\infty$-category $\mC.$

\end{notation}

\begin{proposition}\label{univ2}
Evaluation at the target $\mathcal{R}\mathit{ep}\mathcal{C}\mathit{art} \to \infty\scat$
is a cocartesian fibration.
The unique map $$ \infty\scat_{*//^\oplax} \to \mathcal{R}\mathit{ep}\mathcal{C}\mathit{art} $$
of cocartesian fibrations over $ \infty\scat$ sending $* $ to $*$ is an equivalence.

\end{proposition}

\begin{proof}

As a consequence of \ref{cocartexp} evaluation at the target $\mathcal{C}\mathit{art} \to \infty\scat$
is enriched in $\Cart. $
Hence also the restriction $\mathcal{R}\mathit{ep}\mathcal{C}\mathit{art} \to \infty\scat$ is enriched in $\Cart. $ So it remains to see that evaluation at the target $\mathcal{R}\mathit{ep}\mathcal{C}\mathit{art} \to \infty\scat$ is a 1-cocartesian fibration.
Let $\phi: \mC \to \mD$ be a functor and $X \in \mC$.
We obtain a map $\mC_{//^\oplax X} \to \mC \times_\mD \mD_{//^\oplax \phi(X)} \to $ of cartesian fibrations over $\mC$.
We need to see that for every cartesian fibration $\mF \to \mE$
the following commutative square is a pullback square:
$$
\begin{xy}
\xymatrix{
\Mor_{\mathcal{C}\mathit{art}}(\mD_{//^\oplax \phi(X)} \to \mD,\mF \to \mE) \ar[d] \ar[r]
& \Mor_{\mathcal{C}\mathit{art}}(\mC_{//^\oplax X} \to \mC,\mF \to \mE) \ar[d]
\\ 
\Fun(\mD,\mE) \ar[r] & \Fun(\mC,\mE)
}
\end{xy}$$

By what we have shown, the latter square is a morphism of cartesian fibrations. So by \cref{fiberwiseeq} this square is a pullback square if and only if it induces on the fiber over every functor $\mD \to \mE$ an equivalence.
It induces on the fiber over every functor $\rho: \mD \to \mE$ the induced functor 
$$ \Fun^\cart_\mD(\mD_{//^\oplax \phi(X)},\mD \times_\mE \mF) \simeq 
\{ \phi(X) \} \times_\mD \mD \times_\mE \mF \to \Fun^\cart_\mC(\mC_{//^\oplax X},\mC \times_\mE \mF) \simeq \{X\} \times_\mC \mC \times_\mE \mF, $$
which identifies with the identity of $\{ \rho(\phi(X)) \} \times_\mE \mF $.
So we have proven that evaluation at the target $\mathcal{R}\mathit{ep}\mathcal{C}\mathit{art} \to \infty\scat$ is a cocartesian fibration.
This implies by \cref{envelo} that there is a unique map $\infty\scat_{*//^\oplax}\to \mathcal{R}\mathit{ep}\mathcal{C}\mathit{art}$ of cocartesian fibrations
over $\infty\scat$ sending $* $ to $*$.
This map is an equivalence since it induces on the fiber over every small $\infty$-category $\mC$ the canonical equivalence
$$ \{\mC\} \times_{\infty\scat} \infty\scat_{*//^\oplax } \simeq \Fun(*,\mC) \simeq \mC \simeq \infty\scat^\cart_\mC = \{\mC\} \times_{\infty\scat} \mathcal{R}\mathit{ep}\mathcal{C}\mathit{art}, $$
where the last equivalence is induced by the canonical embedding
$ \mC \to \infty\scat^\cart_\mC, X \mapsto \mC_{//^\oplax X}.$
\end{proof}

Next we define the Grothendieck construction.
For the next definition we use \cref{envelo2}:

\begin{definition}

The Grothendieck construction is the unique map
$$\int: \infty\widehat{\scat}_{//^\oplax \infty\scat} \to \co\widehat{\mathcal{C}\mathit{art}}$$
of cartesian fibrations over
$\infty\widehat{\scat}$ sending $\infty\scat$ to the universal cocartesian fibration $ \infty\scat_{*//^\oplax} \to \infty\scat.$

The map $\int: \infty\widehat{\scat}_{//^\oplax \infty\scat} \to \co\widehat{\mathcal{C}\mathit{art}}$ 
restricts to a map 
\begin{equation}\label{smallGroCon}
\int: \infty\scat \times_{\infty\widehat{\scat}} \infty\widehat{\scat}_{//^\oplax \infty\scat} \to \co\mathcal{C}\mathit{art}
\end{equation} of cartesian fibrations over $\infty\scat$,
which we also call the Grothendieck construction.

\end{definition}

\begin{remark}
The map $$\int: \infty\widehat{\scat}_{//^\oplax \infty\scat} \to \co\widehat{\mathcal{C}\mathit{art}}$$ of cartesian fibrations over
$\infty\widehat{\scat}$ induces on the fiber over every $\mC \in \infty\widehat{\scat}$ the functor
$$\int_\mC: \Fun(\mC, \infty\scat) \to \infty\widehat{\scat}^\cocart_{/\mC}, $$
which takes the pullback along the universal cocartesian fibration.
The map 
\begin{equation}
\int: \infty\scat \times_{\infty\widehat{\scat}} \infty\widehat{\scat}_{//^\oplax \infty\scat} \to \co\mathcal{C}\mathit{art}
\end{equation} of cartesian fibrations over $\infty\scat$ induces on the fiber over every $\mC \in \infty\scat$ the functor
$$\int_\mC : \Fun(\mC, \infty\scat) \to \infty\scat^\cocart_{/\mC}$$
that takes the pullback along the universal cocartesian fibration.

\end{remark}

\begin{definition}
Let $\mC \in \infty\scat$.
\begin{enumerate}[\normalfont(1)]\setlength{\itemsep}{-2pt}
\item The Grothendieck construction for cartesian fibrations is the composition
$$\int_\mC: \Fun(\mC^\circ, \infty\scat)^\cop \simeq \Fun(\mC, \infty\scat^\circ)^{\co\op} \xrightarrow{(-)^{\co\op}_!} \Fun(\mC, \infty\scat^{\co\op})^{\co\op} \simeq $$$$ \Fun(\mC^{\co\op}, \infty\scat) \xrightarrow{\int_{\mC^{\co\op}}} {\infty\scat}^\cocart_{/\mC^{\co\op}} \xrightarrow{(-)^{\co\op}} {\infty\scat}^\cart_{/\mC}$$
\item The Grothendieck construction for anticocartesian fibrations is the composition $$\int_\mC: \Fun(\mC^{\co\op}, \infty\scat)^\co \simeq \Fun(\mC, \infty\scat^{\co\op})^\op \xrightarrow{(-)^{\op}_!} \Fun(\mC, \infty\scat^\op)^\op \simeq $$$$ \Fun(\mC^\op, \infty\scat) \xrightarrow{\int_{\mC^\op}} {\infty\scat}^\cocart_{/\mC^\op} \xrightarrow{(-)^{\op}} {\infty\scat}^{\overline{\co\cart}}_{/\mC}$$
\item The Grothendieck construction for anticartesian fibrations is the composition $$\int_\mC: (\Fun(\mC^\cop, \infty\scat)^\op)^\circ \simeq \Fun(\mC, \infty\scat^\cop)^\co \xrightarrow{(-)^{\co}_!} \Fun(\mC, \infty\scat^\co)^\co \simeq $$$$ \Fun(\mC^\co, \infty\scat) \xrightarrow{\int_{\mC^\co}} {\infty\scat}^\cocart_{/\mC^\co} \xrightarrow{(-)^{\co}} {\infty\scat}^{\overline{\cart}}_{/\mC}.$$
\end{enumerate}
\end{definition}

\begin{lemma}\label{Grorep}
Let $X$ be an $\infty$-category and $Z \in X.$
The following commutative square is a pullback square:
$$
\begin{xy}
\xymatrix{
X_{Z //^\oplax} \ar[d] \ar[r]
& \infty\Cat_{*//^\oplax } \ar[d]
\\ 
X  \ar[r]^{\Mor_\X(Z,-)} & \infty\Cat
}
\end{xy}$$

\end{lemma}

\begin{proof}
The square is a map of cocartesian fibrations by \cref{targetfi} and induces on the fiber over every $Y \in X$ the canonical equivalence
$ \Mor_X(Z,Y) \simeq \Mor_{\infty\Cat}(\bD^0, \Mor_\X(Z,Y)).$
Hence the claim follows from \cref{fiberwiseeq}.
\end{proof}

\begin{corollary}\label{Grorep2}

Let $X$ be an $\infty$-category and $Z \in X.$
The Grothendieck construction for cartesian fibrations sends 
$$ \Mor_\X(-,Z): X^\circ \to \infty\Cat $$ 
to the cartesian fibration $$ X_{//^\oplax Z} \to X. $$

\end{corollary}

\begin{proof}

The Grothendieck construction for cartesian fibrations sends 
$ \Mor_\X(-,Z): X^\circ \to \infty\Cat $ 
to the image under $(-)^\coop$ of the cocartesian fibration $Y \to X^\coop$ associated to the functor $$\Mor_{\X^\coop}(Z,-) \simeq (-)^\coop \circ \Mor_\X(-,Z) : X^\coop \to \infty\Cat.$$ 

By \cref{Grorep} the Grothendieck construction for cocartesian fibrations sends 
$ \Mor_\X(Z,-): X \to \infty\Cat $
to the cocartesian fibration $ X_{Z //^\oplax} \to X.$ 
Hence the cocartesian fibration $Y \to X^\coop$ is $ (X^\coop)_{Z //^\oplax} \to X^\coop.$ 
Hence the cartesian fibration $Y^\coop \to X$ is $X_{//^\oplax Z} \simeq ((X^\coop)_{Z //^\oplax})^\coop \to X.$ 
\end{proof}

\section{\mbox{Fibrations of oriented categories}}

\subsection{Oriented categories of fibrations}

\begin{lemma}\label{comple}

Let $A \to B$ be an embedding of $\Theta$-Segal spaces and let $B$ be complete. Then $A$ is complete.
    
\end{lemma}

\begin{proof}

Let $J$ be the groupoid with two objects and one isomorphism and $\N(J) = \Map_{\infty\Cat^\strict}((-)_{|\Theta},J)$ the strict nerve of $J$.
Then for every $n \geq 0$ the $n$-fold suspension of the map of $\Theta$-spaces
$\N(J) \to *$ induces an equivalence $$ \Map_{\mP(\Theta)}(S^n(*), B) \to \Map_{\mP(\Theta)}(S^n(\N(J)), B),$$
which restricts to an embedding 
$$ \Map_{\mP(\Theta)}(S^n(*), A) \to \Map_{\mP(\Theta)}(S^n(\N(J)), A).$$
We prove that the latter map is also essentially surjective. There is a map $* \to \N(\J)$ of $\Theta$-spaces since $\mJ$ is non-empty. Thus the map $$ \Map_{\mP(\Theta)}(S^n(*), B) \to \Map_{\mP(\Theta)}(S^n(\N(J)), B)$$ admits a left inverse $\xi.$

Every object $T$ of $$\Map_{\mP(\Theta)}(S^n(\N(J)), A) \subset \Map_{\mP(\Theta)}(S^n(\N(J)), B)$$ is the image of an object $Z$ in $\Map_{\mP(\Theta)}(S^n(*), B)$ by completeness of $B.$ So $Z \simeq \xi(T)$. Since $T \in \Map_{\mP(\Theta)}(S^n(\N(J)), A),$
also $Z \simeq \xi(T) \in \Map_{\mP(\Theta)}(S^n(*), A).$
\end{proof}

\begin{notation}

Let $\alpha: X \to S, \beta: Y \to T$ be functors.
Let $\Fun^\oplax(X\to S,Y \to T)$ be the pullback 
$$ \Fun^\oplax(X,Y) \times_{\Fun^\oplax(X,T)} \Fun^\oplax(S,T).$$

\end{notation}

\begin{notation}
Let $\alpha: X \to S, \beta: Y \to S$ be functors.
Let $\Fun^\oplax_S(X,Y)$ be the fiber of the induced functor
$$ \Fun^\oplax(X,Y) \to \Fun^\oplax(X,S)$$ over $\alpha,$
which is the fiber of $\Fun^\oplax(X\to S,Y \to S) \to \Fun^\oplax(S,S)$
over the identity.
    
\end{notation}

\begin{proposition}\label{repren}
Let $X \to S, Y \to T$ be cocartesian fibrations.
The $\Theta$-space 
$$\theta \mapsto \Map_{\coCart}(X \to S , \Fun^\lax(\theta, Y) \to \Fun^\lax(\theta, T)) $$ is a complete Segal $\Theta$-space and so represented by some $\infty$-category
$\Fun^{\oplax, \cocart}(X \to S,Y \to T)$.

The $\infty$-category $$\Fun^{\oplax, \cocart}(X \to S,Y \to T)$$ is the subcategory of $ \Fun^{\oplax}(X \to S,Y \to T)$ whose $n$-morphisms for $n \geq 0$ are the $n$-morphisms of $\Fun^{\oplax}(X \to S,Y \to T)$ corresponding to a map from $X \to S$ to $\Fun^\lax(\theta, Y) \to \Fun^\lax(\theta, T) $ of cocartesian fibrations.

\end{proposition}

\begin{proof}

The functor $ \infty\Cat^\op \to \Fun(\bD^1 ,\infty\Cat), Z \mapsto \Fun^\lax(Z,Y) \to \Fun^\lax(Z,T) $
preserves small limits and by \cref{indulax} induces a functor 
$\infty\Cat^\op \to \co\Cart,$
which also preserves small limits since the inclusion $\coCart \subset \Fun(\bD^1 ,\infty\Cat)$ preserves small limits. 
Thus the functor
$$\infty\Cat^\op \to \mS, \ Z \mapsto \Map_{\coCart}(X \to S, \Fun^\lax(Z, Y) \to \Fun^\lax(Z, T)) $$ preserves small limits. 
In particular, the $\Theta$-space 
$$\theta \mapsto \Map_{\coCart}(X \to S , \Fun^\lax(\theta, Y) \to \Fun^\lax(\theta, T)) $$ is a Segal $\Theta$-space.
There is a canonical embedding of $\Theta$-spaces
$$\theta \mapsto \Map_{\coCart}(X \to S , \Fun^\lax(\theta, Y) \to \Fun^\lax(\theta, T)) \subset \Map_{\Fun(\bD^1,\infty\Cat)}(X \to S, \Fun^\lax(\theta, Y) \to \Fun^\lax(\theta, T))$$
into the complete $\Theta$-Segal space represented by
$\Fun^{\oplax}(X\to S,Y\to T)$. 
\cref{comple} implies that also the presheaf $$\theta \mapsto \Map_{\coCart}(X \to S , \Fun^\lax(\theta, Y) \to \Fun^\lax(\theta, T))$$ is a complete $\Theta$-Segal space represented by a subcategory $$\Fun^{\oplax, \cocart}(X \to S,Y \to T) \subset \Fun^{\oplax}(X \to S,Y \to T)$$ with the desired properties.
\end{proof}

\begin{lemma}\label{cotensol}

Let $X \to S, Y \to T$ be cocartesian fibrations and $K$ an $\infty$-category. The canonical equivalence
$$ \Fun^\oplax(X \to S, \Fun^\lax(K,Y) \to  \Fun^\lax(K,T)) \simeq \Fun^\lax(K,\Fun^\oplax(X \to S,Y \to T)) $$
restricts to an equivalence
$$ \Fun^{\oplax, \cocart}(X \to S, \Fun^\lax(K,Y) \to \Fun^\lax(K,T)) \simeq \Fun^\lax(K,\Fun^{\oplax, \cocart}(X \to S,Y \to T)). $$
    
\end{lemma}

\begin{proof}

Let $L $ be an $\infty$-category.
By density of $\Theta$ in $\infty\Cat$ of \cref{theta}, 
there are a canonical equivalences
\begin{align*}
&\Map_{\infty\Cat}(L, \Fun^{\oplax, \cocart}(X \to S, Y \to T))\\
\simeq &\Map_{\infty\Cat}(\colim_{\theta \to L, \theta \in \Theta},  \Fun^{\oplax, \cocart}(X \to S, Y \to T))\\
\simeq &\lim_{\theta \to L, \theta \in \Theta}\Map_{\infty\Cat}(\theta, \Fun^{\oplax, \cocart}(X \to S, Y \to T))\\
\simeq &\lim_{\theta \to L, \theta \in \Theta}\Map_{\coCart}(X \to S, \Fun^\lax(\theta,Y) \to \Fun^\lax(\theta,T))\\
\simeq &\Map_{\coCart}(X \to S, \Fun^\lax(\colim_{\theta \to L, \theta \in \Theta},Y) \to \Fun^\lax(\colim_{\theta \to L, \theta \in \Theta},T))\\
\simeq &\Map_{\coCart}(X \to S, \Fun^\lax(L,Y) \to \Fun^\lax(L,T)).
\end{align*}
The equivalence of the statement is represented by the following canonical equivalence
\begin{align*}
&\iota_0(\Fun^\lax(L, \Fun^{\oplax, \cocart}(X \to S, \Fun^\lax(K,Y) \to \Fun^\lax(K,T))))\\
\simeq &\iota_0(\Fun^{\oplax, \cocart}(X \to S, \Fun^\lax(L,\Fun^\lax(K,Y)) \to \Fun^\lax(L, \Fun^\lax(K,T))  ))\\
\simeq  &\iota_0(\Fun^{\oplax, \cocart}(X \to S, \Fun^\lax(K \boxtimes L, Y) \to \Fun^\lax(K \boxtimes L, T) ))\\
\simeq &\iota_0(\Fun^\lax(K \boxtimes L,\Fun^{\oplax, \cocart}(X \to S,Y \to T)))\\
\simeq &\iota_0(\Fun^\lax(L,\Fun^\lax(K,\Fun^{\oplax, \cocart}(X \to S,Y\to T)))).
\end{align*}
\end{proof}

\begin{lemma}\label{simplo}

Let $X,Y$ be $\infty$-categories and $Z' \to Z$ an inclusion.
A functor $X \boxtimes Y \to Z$ lands in $Z'$
if and only if for every $n,m \geq 0$ and functors $\bD^n \to X, \bD^m \to Y$ the composition $\bD^{n+m} \to \bD^n \boxtimes \bD^m \to X \boxtimes Y \to Z $
lands in $Z',$ where the functor $\bD^{n+m} \to \bD^n \boxtimes \bD^m$ takes the unique non-degenerate $n+m$-morphism.
    
\end{lemma}

\begin{proof}

The inclusion $Z' \to Z$  induces an inclusion $\Fun^\oplax(Y, Z') \to \Fun^\oplax(Y, Z)$ since right adjoints preserve monomorphisms.
A functor $X \boxtimes Y \to Z$ lands in $Z'$
if and only the corresponding functor $X \to \Fun^\oplax(Y, Z)$
lands in $\Fun^\oplax(Y, Z')$.
The canonical equivalence $ \colim_{\theta \to Y} \theta \simeq Y$
induces equivalences
$$ \Fun^\oplax(Y, Z) \simeq \lim_{\theta \to Y} \Fun^\oplax(\theta, Z),$$
$$ \Fun^\oplax(Y, Z') \simeq \lim_{\theta \to Y} \Fun^\oplax(\theta, Z').$$
A functor $X \to \Fun^\oplax(Y, Z)$
lands in $\Fun^\oplax(Y, Z')$ if and only if for every $\theta \in \Theta$ and functor $\theta \to Y$ the induced functor $X \to \Fun^\oplax(Y, Z) \to \Fun^\oplax(\theta, Z)$
lands in $\Fun^\oplax(\theta, Z'). $

The canonical equivalence $\theta \simeq \colim_{i \in I} \bD^{n_i}$
induces equivalences
$$ \Fun^\oplax(\theta, Z) \simeq \lim_{i \in I} \Fun^\oplax(\bD^{n_i}, Z),$$
$$\Fun^\oplax(\theta, Z') \simeq \lim_{i \in I} \Fun^\oplax(\bD^{n_i}, Z').$$
A functor $X \to \Fun^\oplax(\theta, Z)$
lands in $\Fun^\oplax(\theta, Z')$ if and only if for every $i \in I$ the induced functor $X \to \Fun^\oplax(\theta, Z) \to \Fun^\oplax(\bD^{n_i}, Z)$
lands in $\Fun^\oplax(\bD^{n_i}, Z'). $
Hence a functor $X \boxtimes Y \to Z$ lands in $Z'$ if and only if for every functors $\bD^m \to X, \theta \to Y$
the induced functor $\bD^m \to X \to \Fun^\oplax(\theta, Z) \to \Fun^\oplax(\bD^{n_i}, Z)$
lands in $\Fun^\oplax(\bD^{n_i}, Z').$ 
\end{proof}

\begin{proposition}\label{comprest}

For every cocartesian fibrations $X \to S, Y \to T, Z \to R $
the canonical functor 
$$\Fun^{\oplax,\cocart}(Y \to T, Z \to R) \boxtimes \Fun^{\oplax, \cocart}(X \to S, Y \to T) \subset $$$$ \Fun^\oplax(Y \to T, Z \to R) \boxtimes \Fun^\oplax(X\to S, Y\to T) \to \Fun^\oplax(X \to S ,Z \to R) $$
lands in $$\Fun^{\oplax,\cocart}(X \to S, Z\to R).$$

\end{proposition}

\begin{proof}

By \cref{simplo} it suffices to show that for every $n,m \geq 0$
and functors $$\bD^n \to \Fun^{\oplax,\cocart}(Y\to T,Z \to R), \bD^m \to \Fun^{\oplax,\cocart}(X \to S,Y \to T)$$
the induced functor
$$\rho: \bD^n \boxtimes \bD^m \to \Fun^{\oplax,\cocart}(Y \to T, Z \to R) \boxtimes \Fun^{\oplax,\cocart}(X \to S, Y \to T) \to \Fun^{\oplax}(X \to S, Z \to R) $$
lands in $\Fun^{\oplax,\cocart}(X \to S, Z \to R).$

The functor $\bD^n \to \Fun^{\oplax,\cocart}(Y \to T, Z \to R)$ corresponds to a map of cocartesian fibrations from $Y \to T $ to $ \Fun^\lax(\bD^n,Z) \to \Fun^\lax(\bD^n,R) $,
and the functor $\bD^m \to \Fun^{\oplax,\cocart}(X \to S,Y \to T)$ corresponds to a map of cocartesian fibrations from $X \to S $ to $ \Fun^\lax(\bD^m,Y) \to \Fun^\lax(\bD^m,T) $

The functor $\rho$ corresponds to the map
$$\gamma: X \to \Fun^\lax(\bD^m,Y) \to\Fun^\lax(\bD^m,\Fun^\lax(\bD^n,Z)) \simeq \Fun^\lax(\bD^n \boxtimes \bD^m, Z) \to \Fun^\lax(\bD^{n+m},Z) $$
in $\Fun(\bD^1, \infty\Cat)$
induced by the functor $\bD^{n+m} \to \bD^n \boxtimes \bD^m$ taking the unique non-degenerate $n+m$-morphism.

We have to see that the functor $\rho$ lands in $\Fun^{\oplax, \cocart}(X \to S, Z\to R).$ This holds by definition if and only if 
$\gamma$ is a map of cocartesian fibrations.
The first map in the composition of $\gamma$ is a map of cocartesian fibrations since the functor $\bD^m \to \Fun^{\oplax}(X \to S,Y \to T)$ lands in $\Fun^{\oplax,\cocart}(X \to S,Y \to T)$.
The second map in the composition of $\gamma$ is a map of cocartesian fibrations since the functor $\bD^n \to \Fun^{\oplax}(Y \to T, Z\to R)$ lands in $\Fun^{\oplax,\cocart}(Y \to T, Z \to R)$
and so the map $Y \to \Fun^\lax(\bD^n,Z) $ and so also the map $$\Fun^\lax(\bD^m,Y) \to \Fun^\lax(\bD^m,\Fun^\lax(\bD^n,Z)) \simeq \Fun^\lax(\bD^n \boxtimes \bD^m, Z) $$ are maps of cocartesian fibrations.

The last functor in the composition of $\gamma$ is a map of cocartesian fibrations by the characterization of cocartesian morphisms of \cref{indufibr}.
\end{proof}

We obtain the following:

\begin{theorem}\label{orienfib}

There is antioriented subcategory 
$$ \mathfrak{coCart} \subset \Fun(\bD^1, \infty\mathfrak{Cat})$$
such that for every pair of cocartesian fibrations $X \to S$ and $Y \to T$
there is a canonical equivalence
$$ \L\Mor_{\mathfrak{coCart}}(X \to S, Y\to T) \simeq \Fun^{\oplax,\cocart}(X \to S, Y\to T).$$

\end{theorem}

\begin{proof}
This follows from \cref{enrsub} and \cref{comprest}.
\end{proof}

\begin{notation}

Let $S$ be an $\infty$-category.
Let $\infty\fcat^\cocart_{/S} \subset \infty\fcat_{/S}$ be the fiber over $S$ of the antioriented functor
$$\mathfrak{coCart} \to \infty\fcat $$ evaluating at the target.
    
\end{notation}

\begin{notation}

Let $S$ be an $\infty$-category.
The monoidal equivalence $(-)^\coop: (\infty\Cat, \boxtimes) \to (\infty\Cat, \boxtimes) $
induces an equivalence of bioriented categories
$$ \infty\fcat \simeq \coop_!(\infty\fcat) = \infty\fcat^\cop.$$

Let $$\mathfrak{Cart} \subset \Fun(\bD^1, \infty\mathfrak{Cat})$$
be the antoriented subcategory corresponding to
$$\mathfrak{coCart}^\cop \subset \Fun(\bD^1, \infty\mathfrak{Cat})^\cop $$
under the antioriented equivalence
$$ \Fun(\bD^1, \infty\mathfrak{Cat}) \simeq \Fun(\bD^1, \infty\mathfrak{Cat}^\cop) \simeq \Fun(\bD^1, \infty\mathfrak{Cat})^\cop.$$
    
\end{notation}

\begin{remark}

For every pair of cartesian fibrations $X \to S$ and $Y \to T$
there is a canonical equivalence
$$ \Fun^{\oplax,\cart}(X \to S, Y\to T):= \L\Mor_{\mathfrak{Cart}}(X \to S, Y \to T) \simeq \Fun^{\oplax,\cocart}(X^\coop \to S^\coop, Y^\coop\to T^\coop)^\coop.$$

\end{remark}

\begin{notation}
Let $X \to S$ be a cocartesian fibration
and $K$ an $\infty$-category.
Let $X^K_\lax \to S $ be the pullback of the induced functor
$$ \Fun^\lax(K,X) \to \Fun^\lax(K,S) $$
along the diagonal functor $S \to \Fun^\lax(K,S). $

\end{notation}

\cref{indulax} gives the following:

\begin{corollary}\label{restcoca}

Let $X \to S$ be a cocartesian fibration
and $K$ an $\infty$-category.
The functor $X^K_\lax \to S $ is a cocartesian fibration.
   
\end{corollary}

\begin{remark}For every functors $X \to S, Y \to S$ there is a canonical equivalence
$$ \Fun^\oplax_S(X,Y^K_\lax) \simeq \Fun^\lax(K,\Fun^\oplax_S(X,Y)).$$

Hence  $Y^K_\lax \to S$ is the left cotensor of $\K$ and $Y \to S$ in the antioriented category $\infty\fcat_{/S}.$

\end{remark}

\begin{notation}\label{univeqr}
    
For every cocartesian fibrations $X \to S, Y \to S$ 
we set $$\Fun^{\oplax, \cocart}_S(X,Y) := \LMor_{\infty\fcat^\cocart_{/S}}(\X,\Y).$$

\end{notation}

\begin{remark}

The $\infty$-category $\Fun^{\oplax, \cocart}_S(X,Y)$ is the subcategory of $\Fun^\oplax_\S(\X,\Y)$ such that for every $\infty$-category $K$  the canonical equivalence
$$\Map_{\infty\Cat}(K,\Fun^{\oplax}_S(X,Y)) \simeq \Map_{\infty\Cat_{/\rS}}(X,Y^K_\lax)$$
restricts to an equivalence
$$\Map_{\infty\Cat}(K,\Fun^{\oplax, \cocart}_S(X,Y)) \simeq \Map_{\infty\Cat^\cocart_{/\rS}}(X,Y^K_\lax).$$

The latter equivalence implies that the canonical equivalence
$$ \Fun^\oplax_S(X,Y^K_\lax) \simeq \Fun^\lax(K,\Fun^\oplax_S(X,Y))$$
restricts to an equivalence 
$$ \Fun^{\oplax, \cocart}_S(X,Y^K_\lax) \simeq \Fun^\lax(K,\Fun^{\oplax, \cocart}_S(X,Y)).$$

Hence  $Y^K_\lax \to S$ is the left cotensor of $\K$ and $Y \to S$ in the antioriented category $\infty\fcat^\cocart_{/S}.$
    
\end{remark}

% \begin{remark}\label{transfou}

% For every $\infty$-categories $L, K$ and functor $Y \to S$ there is a canonical equivalence over $S:$
% $$ (Y^K_\lax)^L_\lax \simeq Y^{K \boxtimes L} $$
% represented by the following canonical equivalence
% $$ \Fun^\oplax_S(X,(Y^K_\lax)^L_\lax) \simeq \Fun^\lax(L,\Fun^\oplax_S(X,Y^K_\lax)) \simeq $$$$ \Fun^\lax(L,\Fun^\lax(K,\Fun^\oplax_S(X,Y)))$$
% $$ \simeq \Fun^\lax(K \boxtimes L,\Fun^\oplax_S(X,Y)) \simeq  \Fun^\oplax_S(X,Y^{K \boxtimes L}_\lax),$$
% where $X \to S$ is any functor.
    
% \end{remark}

\begin{theorem}\label{leftadjointtheorem}
Let $S$ be an $\infty$-category.
The inclusion of antioriented categories $$\infty\fcat^\cocart_{/S}\subset \infty\fcat_{/S}$$
admits a left adjoint that sends $X \to S$ 
to $\Env(X) \to S$.
\end{theorem}

\begin{proof}

This follows from \cref{huip} and \cref{adj} (2) because the antioriented inclusion $$\infty\fcat^\cocart_{/S}\subset \infty\fcat_{/S}$$ preserves left cotensors by \cref{restcoca} and \cref{univeqr}.  
\end{proof}

\subsection{Fibrations of oriented categories}

\begin{definition}
An oriented functor $\phi: \mC \to \mD$ is a cartesian fibration
if it is a $(\infty\Cat, \boxtimes^\rev)$-enriched cartesian fibration
in the sense of \cref{enrfibr} and, for every $X,Y \in \mC$
the induced functor
$$\Mor_\mC(X,Y) \to \Mor_\mD(\phi(X),\phi(Y)) $$ is a cocartesian fibration and for every $X \in \mC$ and morphism $Y \to Z $ in $\mC$ the induced commutative square
$$\begin{xy}
\xymatrix{
\Mor_\mC(X,Y) \ar[d]^{} \ar[r]
& \Mor_\mC(X,Z) \ar[d]
\\ 
\Mor_\mD(\phi(X),\phi(Y)) \ar[r] & \Mor_\mD(\phi(X),\phi(Z))}
\end{xy}$$
is a map of cocartesian fibrations. 
    
\end{definition}

\begin{definition}
Let $\phi: \mC \to \mD, \phi': \mC' \to \mD'$ be oriented cartesian fibrations.
A map of oriented cartesian fibrations $\phi \to \phi'$
is a commutative square
$$\begin{xy}
\xymatrix{
\mC \ar[d]^{\phi} \ar[r]
& \mC' \ar[d]^{\phi'}
\\ 
\mD \ar[r] & \mD'}
\end{xy}$$
that is a morphism of $(\infty\Cat, \boxtimes^\rev)$-enriched fibrations
in the sense of \cref{enrfibr} and such that for every $X,Y \in \mC$
lying over $X', Y' \in \mC' $ the induced commutative square
$$\begin{xy}
\xymatrix{
\Mor_\mC(X,Y) \ar[d]^{} \ar[r]
& \Mor_{\mC'}(X',Y') \ar[d]
\\ 
\Mor_\mD(\phi(X),\phi(Y)) \ar[r] & \Mor_{\mD'}(\phi'(X'),\phi'(Y'))}
\end{xy}$$
is a map of cocartesian fibrations. 
    
\end{definition}

% \begin{remark}Let $K$ be a category and let $X,Y$ be $\infty$-categories.

% By ... oriented fibrations over $S(\K)$ whose fiber over $0$ is $X$ and whose fiber over 1 is $Y$ are classified by functors $K \to \boxtimes\Fun(X,Y).$
    
% \end{remark}

\begin{definition}We have the following dual notions:

\begin{itemize}

\item An antioriented functor $\phi: \mC \to \mD$ is a cocartesian fibration
if the oriented functor $\phi^\coop: \mC^\coop \to \mD^\coop$ is a cartesian fibration.

\item An antioriented functor $\phi: \mC \to \mD$ is an anticartesian fibration
if the oriented functor $\phi^\co: \mC^\co \to \mD^\co$ is a cartesian fibration.

\item An oriented functor $\phi: \mC \to \mD$ is an anticocartesian fibration
if the oriented functor $\phi^\op: \mC^\op \to \mD^\op$ is a cartesian fibration.

\end{itemize}
    
\end{definition}

Next we consider examples of cartesian fibrations of oriented categories.

\begin{notation}
Let $$ \mathfrak{COCART} \subset \Fun(\bD^1, \fcat)$$ be the full oriented subcategory of cocartesian fibrations.
    
\end{notation}

\begin{proposition}\label{orientfibru}
The oriented functor $$\mathfrak{COCART} \to \infty\fcat $$
evaluating at the target is a cartesian fibration.

\end{proposition}

\begin{proof}

By \cref{enrtargetfibr} the oriented functor $$ \Fun(\bD^1,\infty\fcat) \to \infty\fcat $$
evaluating at the target is a cartesian fibration of right $(\infty\Cat, \boxtimes)$-enriched categories, where the fiber transport assigns the pullback.
Since cocartesian fibrations are stable under pullback, the
oriented functor $$\mathfrak{COCART} \to \infty\fcat $$
evaluating at the target is also a cocartesian fibration of right $(\infty\Cat, \boxtimes)$-enriched categories.

The latter oriented functor induces on morphism $\infty$-categories between any cocartesian fibrations $X \to S$ and $Y \to T$ the functor 
$$ \Fun^\lax(X, Y) \times_{\Fun^\lax(X, T)} \Fun^\lax(S, T) \to \Fun^\lax(S,T), $$
which is an cocartesian fibration by \cref{indulax} since $Y \to T$ is a cocartesian fibration.

For every functor cocartesian fibrations $X \to S, Y \to T, Z \to R$ and
map of cocartesian fibrations $(Y \to T) \to (Z \to R) $ the induced 
commutative square
$$\begin{xy}
\xymatrix{
\R\Mor_{\mathfrak{COCART}}(X \to S, Y \to \T) \ar[d]^{} \ar[r]
& \R\Mor_{\mathfrak{COCART}}(X \to S, Z \to R) \ar[d]
\\ 
\Fun^\lax(S, T) \ar[r] & \Fun^\lax(S, R),}
\end{xy}$$
which identifies with the commutative square
$$\begin{xy}
\xymatrix{
 \Fun^\lax(X, Y) \times_{\Fun^\lax(X, T)} \Fun^\lax(S, T) \ar[d]^{} \ar[r]
& \Fun^\lax(X, Z) \times_{\Fun^\lax(X, R)} \Fun^\lax(S, R) \ar[d]
\\ 
\Fun^\lax(S, T) \ar[r] & \Fun^\lax(S, R).}
\end{xy}$$

The latter commutative square is a map of cocartesian fibrations
since the commutative square 
$$\begin{xy}
\xymatrix{
\Fun^\lax(X, Y) \ar[d]^{} \ar[r]
& \Fun^\lax(X, Z) \ar[d]
\\ 
\Fun^\lax(X, T) \ar[r] & \Fun^\lax(X, R).}
\end{xy}$$
is a map of cocartesian fibrations by \cref{indulax}.
\end{proof}

\begin{proposition}\label{orientfibru0}
The functor $$\mathcal{COCART} \to \infty\scat $$
evaluating at the target is a cartesian fibration.

\end{proposition}

\begin{proof}

The proof is similar to the one of \cref{orientfibru}, where we use
\cref{enrcocartexp} instead of \cref{indulax}.   
\end{proof}

\subsection{A Grothendieck construction for oriented fibrations}

\begin{notation}
Let $\infty\fcat_{*//} $ be the full subcategory of $\infty\fcat_{/\bD^1}$ consisting of the cartesian fibrations $\mM \to [1]$ whose fiber over 1 is the final $\infty$-category.
\end{notation}

\begin{remark}
The bioriented functors $\infty\fcat_{/\bD^1} \to \infty\fcat$ taking the fibers over $0$ and $1$ give rise to bioriented functors $\infty\fcat_{*//} \to \infty\fcat$.
We view $\infty\fcat_{*//}$ as a bioriented category over $\infty\fcat$ by taking the fiber over 0.
Objects of $\infty\fcat_{*//}$ are cartesian fibrations $\mM \to [1]$,
which are classified by functors
$\bD^0 \to \mC$, where $\mC$ is the fiber over 0.
The latter correspond to pairs $(\mC,X),$ where $\mC $ is a small $\infty$-category and $X \in \mC.$  
\end{remark}

\begin{notation}
Let $\alpha: \mA \to \mC, \beta: \mB \to \mC$ be functors.
Let $$\Fun^\oplax_\mC(\mA,\mB):= \{\alpha\}\times_{\Fun^\oplax(\mA,\mC)} \Fun^\oplax(\mA,\mB).$$

\end{notation}

\begin{remark}\label{remaros}
Let $K$ be an $\infty$-category and $\alpha: \mA \to \mC, \beta: \mB \to \mC$ be functors. Let $p_\mA: K \boxtimes \mA \to \mA$ be the projection.
There is a canonical equivalence
$$ \Fun^\oplax(K, \Fun^\oplax_\mC(\mA,\mB)) \simeq \{\underline{\alpha}\}\times_{\Fun^\oplax(K,\Fun^\oplax(\mA,\mC))} \Fun^\oplax(K,\Fun^\oplax(\mA,\mB)) \simeq $$$$\{\alpha \circ p_\mA \}\times_{\Fun^\oplax(K \boxtimes \mA,\mC))} \Fun^\oplax(K \boxtimes \mA,\mB)) \simeq \Fun^\oplax_\mC(K \boxtimes \mA,\mB).$$
There is a canonical equivalence
$$ \Fun^\lax(K, \Fun^\oplax_\mC(\mA,\mB)) \simeq \{\underline{\alpha}\}\times_{\Fun^\lax(K,\Fun^\oplax(\mA,\mC))} \Fun^\lax(K,\Fun^\oplax(\mA,\mB)) \simeq $$$$
\{\underline{\delta \circ \alpha}\}\times_{\Fun^\oplax(\mA,\Fun^\lax(K, \mC))} \Fun^\oplax(\mA,\Fun^\lax(K,\mB)) \simeq \Fun_\mC^\oplax(\mA,\mC \times_{\Fun^\lax(K, \mC)} \Fun^\lax(K, \mB)).$$

Let $\sigma : \mC \to \mD, \gamma: \mE \to \mD$ be functors.
There is a canonical equivalence
$$ \Fun^\oplax_\mD(\mA,\mE) \simeq \{\sigma \circ \alpha\}\times_{\Fun^\oplax(\mA,\mD)} \Fun^\oplax(\mA,\mE) \simeq \{\alpha\}\times_{\Fun^\oplax(\mA,\mC)} \Fun^\oplax(\mA,\mC \times_\mD \mE)$$$$ \simeq \Fun^\oplax_\mC(\mA,\mC \times_\mD \mE).$$
    
\end{remark}

\begin{lemma}\label{sectio}
Let $\mM \to [1]$ be a cartesian fibration.
There is a canonical equivalence
$$\Fun^\oplax_{\bD^1}(\bD^1, \mM) \simeq \mM_1 \times_{\Fun^\oplax(\{1 \},\mM_0)} \Fun^\oplax(\bD^1,\mM_0).$$

\end{lemma}

\begin{proof}

Let $G: \mM_1 \to \mM_0$ be the functor classified by the cocartesian fibration $\mM \to \bD^1.$
We first observe that there is a canonical bijection of equivalence classes between the $\infty$-categories 
$\Fun^\oplax_{\bD^1}(\bD^1, \mM) $ and $\mM_1 \times_{\Fun^\oplax(\{1 \},\mM_0)} \Fun^\oplax(\bD^1,\mM_0):$
an object of $\Fun^\oplax_{\bD^1}(\bD^1, \mM)$
is a section of $\mM \to \bD^1$ which corresponds to an object $X$ in $\mM_1$ and a morphism $Y \to G(X) $ in $\mM_0.$
The latter is precisely an object of $\mM_1 \times_{\Fun^\oplax(\{1 \},\mM_0)} \Fun^\oplax(\bD^1,\mM_0).$

To prove the result it suffices by the Yoneda-lemma to see that for every $\infty$-category $\mB$ there is a bijection of equivalence classes between the $\infty$-categories 
$ \Fun^\lax(\mB, \Fun^\oplax_{\bD^1}(\bD^1, \mM)) $ and $$\Fun^\lax(\mB, \mM_1 \times_{\Fun^\oplax(\{1 \},\mM_0)} \Fun^\oplax(\bD^1,\mM_0)).$$
Let $\mN:= \bD^1 \times_{\Fun^\lax(\mB, \bD^1)} \Fun^\lax(\mB, \mM) $. By \cref{remaros} there are canonical equivalences
$$ \Fun^\lax(\mB, \Fun^\oplax_{\bD^1}(\bD^1, \mM)) \simeq \Fun_{\bD^1}^\oplax(\bD^1, \mN),$$
$$ \Fun^\lax(\mB, \mM_1 \times_{\Fun^\oplax(\{1 \},\mM_0)} \Fun^\oplax(\bD^1,\mM_0)) \simeq $$$$ \Fun^\lax(\mB, \mM_1) \times_{\Fun^\oplax(\{1 \},\Fun^\lax(\mB,\mM_0))} \Fun^\oplax(\bD^1,\Fun^\lax(\mB,\mM_0)) \simeq $$$$ \mN_1 \times_{\Fun^\oplax(\{1 \}, \mN_0)} \Fun^\oplax(\bD^1,\mN_0).$$
By \cref{laxfib} the functor $\mN \to [1]$ is a cartesian fibration.
This proves the result by the first part of the proof.
\end{proof}

\begin{proposition}\label{homuniv}
Let $(\mC,X), (\mD,Y) \in \infty\fcat_{*//}$.
There is a canonical equivalence
$$ \L\Mor_{\infty\fcat_{*//}}((\mC,X), (\mD,Y)) \simeq \Fun^\oplax(\mC,\mD) \times_{\Fun^\oplax(\{0\},\mD)} \Fun^\oplax(\bD^1,\mD) \times_{\Fun^\oplax(\{1\},\mD)} \{Y\}.$$
\end{proposition}

\begin{proof}

The canonical map
$\mC \coprod_{\{0\}} \bD^1 \to \mM $  of cocartesian fibrations over over $\{0\} \coprod_{\{0\}} \bD^1$, where $\mC = \mM_0,$ induces on the fiber over 0 the identity of $\mC$ and on the fiber over 1 the identity of $*$ and so is an equivalence.

There is a canonical equivalence
$$
\L\Mor_{\infty\fcat_{*//}}((\mC,X), (\mD,Y)) \simeq \Fun^\oplax_{\bD^1}(\mC \coprod_{\{0\}} \bD^1, (\mD,Y))
\simeq $$$$
\Fun^\oplax_{\bD^1}(\mC \times \{0\}, (\mD,Y))
\underset{\Fun^\oplax_{\bD^1}(\{0\}, (\mD,Y))}{\times} \Fun^\oplax_{\bD^1}(\bD^1, (\mD,Y)) \simeq $$$$ \Fun^\oplax(\mC, \mD)
\times_{\Fun^\oplax(\{0\},\mD)} \Fun^\oplax_{\bD^1}(\bD^1, (\mD,Y)) \simeq $$$$ \Fun^\oplax(\mC,\mD) \times_{\Fun^\oplax(\{0\},\mD)} \Fun^\oplax(\bD^1,\mD) \times_{\Fun^\oplax(\{1\},\mD)} \{Y\},
$$
where the third equivalence is by \cref{remaros} and the last equivalence is by \cref{sectio}.
\end{proof}

\begin{theorem}
The bioriented functor $\infty\fcat_{*//} \to \infty\fcat$ given by the target projection is a cocartesian fibration whose fiber over every small $\infty$-category $\mC$ is $\mC.$
Moreover, the pullback $\infty\scat \times_{\infty\fcat} \infty\fcat_{*//} \to \infty\scat$
is the universal cocartesian fibration $ \infty\scat_{*//^\oplax} \to \infty\scat.$

\end{theorem}

\begin{proof}
We start with showing that the bioriented functor $\infty\fcat_{*//} \to \infty\fcat$ is a cocartesian fibration.
Let $\mM \to \bD^1$ be a cartesian fibration whose fiber over 1 is the final $\infty$-category classifying a functor $* \xrightarrow{X} \mM_0$
and $\phi: \mM_0 \to \mC$ a functor.
The canonical map
$$ \mM_0 \coprod_{\{0\}} \bD^1 \to \mM$$ of cartesian fibrations over $\{0\} \coprod_{\{0\}} \bD^1$ induces on each fiber an equivalence and so is an equivalence.
Let $$ \mN := \mC \coprod_{\{0\}} \bD^1 \to \{0\} \coprod_{\{0\}} \bD^1.$$
Then $\mN \to \bD^1$ is a cartesian fibration and the functor $\mM_0 \to \mC$ gives rise to a map
$$\mM \simeq \mM_0 \coprod_{\{0\}} \bD^1 \to \mC \coprod_{\{0\}} \bD^1 \simeq \mN$$ of cartesian fibrations over $\bD^1.$
For every cartesian fibration $\mO \to \bD^1$
classifying a functor $ * \xrightarrow{Y} \mO_0. $
the following commutative square is a pullback square:
$$
\begin{xy}
\xymatrix{
\Fun_{\bD^1}(\mC \coprod_{\{0\}} \bD^1,\mO) \ar[d] \ar[r]
& \Fun(\mC,\mO_0) \simeq  \Fun_{\bD^1}(\mC \times\{0\},\mO)\ar[d]
\\ 
\Fun_{\bD^1}(\mM_0 \coprod_{\{0\}} \bD^1,\mO)  \ar[r] & \Fun(\mM_0,\mO_0) \simeq  \Fun_{\bD^1}(\mM_0 \times\{0\},\mO).
}
\end{xy}$$
This proves that the bioriented functor $\infty\fcat_{*//} \to \infty\fcat$ is a 1-cocartesian fibration.
The latter induces on morphism $\infty$-categories between $\mM \to \bD^1, \mO\to \bD^1$ a functor
$$ \L\Mor_{\infty\fcat_{*//}}((\mM_0,X),(\mO_0,Y)) \to \Fun^\oplax(\mM_0,\mO_0) ,$$ which by \cref{homuniv} identifies with the pullback $$\Fun^\oplax(\mM_0,\mO_0) \times_{\Fun^\oplax(\{0\},\mO_0)} \Fun^\oplax(\bD^1,\mO_0) \times_{\Fun^\oplax(\{1\},\mO_0)} \{Y\} \to \Fun^\oplax(\mM_0,\mO_0)$$
of the cartesian fibration
$$  \Fun^\oplax(\bD^1,\mO_0) \times_{\Fun^\oplax(\{1\},\mO_0)} \{Y\} \to \Fun^\oplax(\{0\},\mO_0). $$

So it suffices to see that for every functor $\psi: \mO \to \mU $ over $\bD^1$
between cartesian fibrations over $\bD^1$ classifying functors $* \xrightarrow{Y} \mO_0, * \xrightarrow{Z} \mU_0 $ the induced functor
$$\L\Mor_{\infty\fcat_{*//}}((\mM_0,X),(\mO_0,Y)) \to \L\Mor_{\infty\fcat_{*//}}((\mM_0,X),(\mU_0,Z))
$$ preserves cartesian morphisms. 
This functor identifies with a functor
$$\Fun^\oplax(\mM_0,\mO_0) \times_{\Fun^\oplax(\{0\},\mO_0)} \Fun^\oplax(\bD^1,\mO_0) \times_{\Fun^\oplax(\{1\},\mO_0)} \{Y\} \to $$$$ \Fun^\oplax(\mM_0,\mU_0) \times_{\Fun^\oplax(\{0\},\mU_0)} \Fun^\oplax(\bD^1,\mU_0) \times_{\Fun^\oplax(\{1\},\mU_0)} \{Z\}, $$
which is the pullback of the functor
$$ \Fun^\oplax(\bD^1,\mO_0) \underset{\Fun^\oplax(\{1\},\mO_0)}{\times} \{Y\} \to \Fun^\oplax(\{0\},\mO_0) \times_{\Fun^\oplax(\{0\},\mU_0)} \Fun^\oplax(\bD^1,\mU_0) \times_{\Fun^\oplax(\{1\},\mU_0)} \{Z\} $$
induced by the functor $\psi_0: \mO_0 \to \mU_0 $ and the canonical morphism $ \psi(Y) \to Z$ in 
$\mU_0$ and by \cref{targetFib} preserves cartesian morphisms. 
This proves that the bioriented functor $\infty\fcat_{*//} \to \infty\fcat$ is a cocartesian fibration.

Next we observe that the pullback $ \infty\scat \times_{\infty\fcat} \infty\fcat_{*//} \to \infty\scat$ is a cocartesian fibration of $\infty$-categories: for every morphism $(\mC,X) \to (\mD,Y)$ in
$\infty\scat \times_{\infty\fcat} \infty\fcat_{*//} $ 
corresponding to a functor $F:\mC \to \mD$ and a morphism
$\kappa: F(X) \to Y$ in $\mD$, and every $(\mE,\Z) \in \infty\scat \times_{\infty\fcat} \infty\fcat_{*//} $
the induced functor
$$\LMor_{\infty\scat \times_{\infty\fcat} \infty\fcat_{*//}}((\mD,Y),(\mE,Z)) \to \LMor_{\infty\scat \times_{\infty\fcat} \infty\fcat_{*//}}((\mC,X),(\mE,Z)) $$ 
identifies with the functor
$$ \Fun(\mD,\mE) \times_{\Fun^\oplax(\{0\},\mE)} \mE_{//^\oplax Z} \to \Fun(\mC,\mE) \times_{\Fun^\oplax(\{0\},\mE)} \mE_{//^\oplax Z} $$
induced by the morphism $\kappa: F(X) \to Y$ in $\mD,$ which is a map of cartesian fibrations.

% The latter functor is a map of cartesian fibrations:
% the map $\kappa: F(X) \to Y$ in $\mD$ gives rise to a natural transformation from $\Fun(\mD,\mE) \to \Fun(\mC,\mE) \to \Fun^\oplax(\{0\},\mE)$ to $\Fun(\mD,\mE) \to \Fun^\oplax(\{0\},\mE)$ corresponding to a functor $\bD^1 \times \Fun(\mD,\mE) \to \mE$.
% The pullback $\bD^1 \times \Fun(\mD,\mE) \times_{\mE} \mE_{//^\oplax Z} \to \bD^1 \times \Fun(\mD,\mE)$
% is a cartesian fibration and so a map of cartesian fibrations over $\bD^1,$ which classifies the map 

Moreover for every $(\mC,X), (\mD,Y), (\mE,\Z)$ in
$\infty\scat \times_{\infty\fcat} \infty\fcat_{*//} $
the composition functor
$$ \LMor_{\infty\scat \times_{\infty\fcat} \infty\fcat_{*//}}((\mD,Y),(\mE,Z)) \boxtimes \LMor_{\infty\scat \times_{\infty\fcat} \infty\fcat_{*//}}((\mC,X),(\mD,Y)) \to $$$$ \LMor_{\infty\scat \times_{\infty\fcat} \infty\fcat_{*//}}((\mC,X),(\mE,Z)) $$
is the canonical functor
$$ (\Fun(\mD,\mE) \underset{\Fun^\oplax(\{0\},\mE)}{\times} \mE_{//^\oplax Z}) \boxtimes (\Fun(\mC,\mD) \times_{\Fun^\oplax(\{0\},\mD)} \mD_{//^\oplax Y}) \to \Fun(\mC,\mE) \underset{\Fun^\oplax(\{0\},\mE)}{\times} \mE_{//^\oplax Z}, $$
which canonically factors as
$$ (\Fun(\mD,\mE) \underset{\Fun^\oplax(\{0\},\mE)}{\times} \mE_{//^\oplax Z}) \boxtimes (\Fun(\mC,\mD) \underset{\Fun^\oplax(\{0\},\mD)}{\times} \mD_{//^\oplax Y}) \to $$$$ (\Fun(\mD,\mE) \underset{\Fun^\oplax(\{0\},\mE)}{\times}\mE_{//^\oplax Z}) \times (\Fun(\mC,\mD) \underset{\Fun^\oplax(\{0\},\mD)}{\times} \mD_{//^\oplax Y}) \to \Fun(\mC,\mE) \underset{\Fun^\oplax(\{0\},\mE)}{\times} \mE_{//^\oplax Z}.$$

\cite[Theorem 4.2.7.]{oriented} implies that the oriented category $ \infty\scat \times_{\infty\fcat} \infty\fcat_{*//} $ is an $\infty$-category.

By \cref{envelo} there is a unique map $\rho: \infty\scat_{*//^\oplax} \to \infty\scat \times_{\infty\fcat} \infty\fcat_{*//} $ of cocartesian fibrations
over $\infty\scat$ that sends the identity of $\bD^0$ to
itself. The latter sends an $\infty$-category $\mC$ equipped with an object $X \in \mC$ to the cartesian fibration 
$$ \mC \coprod_{\{0\} } \bD^1 \to \{0\} \coprod_{\{0\}} \bD^1 $$
and so is essentially surjective.
The functor $\rho$ induces on the fiber over every small $\infty$-category $\mC$ a functor 
$$\{\mC\} \times_{\infty\scat} \infty\scat_{*//^\oplax} \simeq \mC \to \{\mC\} \times_{\infty\fcat} \infty\fcat_{*//},$$
which by \cref{homuniv} induces on morphism objects between $X,Y \in \mC$
the canonical equivalence $$ \L\Mor_\mC(X,Y) \simeq \{X\} \times_{\Fun^\oplax(\{0\},\mC)} \Fun^\oplax(\bD^1,\mC) \times_{\Fun^\oplax(\{1\},\mC)} \{Y\}.$$
This implies that $\rho$ is an equivalence.    
\end{proof}

\begin{theorem}\label{T7}

The Grothendieck construction for oriented cocartesian fibrations is the unique map 
$$\int: {\boxtimes\widehat{\Cat}}_{//^\oplax\infty\fcat} \to {\boxtimes\co\widehat{\Cart}}$$
of cartesian fibrations over $\boxtimes\widehat{\Cat}$ sending $\infty\fcat$ to the universal oriented cocartesian fibration $$ \infty\fcat_{*//^\oplax} \to \infty\fcat.$$
The map $\int: {\boxtimes\widehat{\Cat}}_{//^\oplax\infty\fcat} \to {\boxtimes\co\widehat{\Cart}}$ of cartesian fibrations over
$\boxtimes\widehat{\Cat}$
restricts to a map 
\begin{equation}\label{GroCon2}
\int: {\boxtimes\Cat}_{//^\oplax\infty\fcat} \to {\boxtimes\coCart}
\end{equation} of cartesian fibrations over $\boxtimes\Cat$.

\end{theorem}

\begin{remark}

The map (\ref{GroCon2}) of cartesian fibrations over $\boxtimes\Cat$
induces on the fiber over every $\mC \in \boxtimes\Cat$ the functor
$$\int_\mC : {\boxtimes\Fun}(\mC, \infty\fcat) \to \boxtimes\Cat^\cocart_{/\mC}$$
that takes the pullback along the universal oriented cocartesian fibration.
    
\end{remark}

\cref{Grorep} motivates the following construction of oplax slice antioriented categories, which in the classical setting were constructed by Ara-Guetta \cite{ARA2026110762}:

\begin{definition}
Let $\mC$ be an antioriented category  and $X \in \mC$.
Let $$\mC_{X //^\oplax} := \int \L\Mor_\mC(X,-)\to \mC $$ be the pullback of the antioriented functor $\L\Mor_\mC(X,-): \mC \to \infty\fcat$
along $\infty\fcat_{*//} \to \infty\fcat$.  
\end{definition}

\begin{remark}
The latter definition extends the notation $\infty\fcat_{*//}.$
\end{remark}

\begin{corollary}
Let $\mC$ be an antioriented category  and $X \in \mC$.
Let $\alpha: X \to Y$ and $\beta: X \to Z $ be morphisms in $\mC$.
There is a canonical equivalence
$$ \L\Mor_{\mC_{X//^\oplax}}(Y, Z) \simeq $$$$ \L\Mor_\mC(Y,Z) \underset{\Fun^\oplax(\{0\},\L\Mor_\mC(X,Z))}{\times} \Fun^\oplax(\bD^1,\L\Mor_\mC(X,Z)) \underset{\Fun^\oplax(\{1\},\L\Mor_\mC(X,Z))}{\times} \{\beta\}.$$
\end{corollary}

\begin{proof}
By \cref{homuniv} there is a canonical equivalence
$$ \L\Mor_{\mC_{X//^\oplax}}(Y, Z) \simeq $$
$$\L\Mor_\mC(Y,Z) \times_{\L\Mor_{\infty\fcat}(\L\Mor_\mC(X,Y),\L\Mor_\mC(X,Z))}
\L\Mor_{\infty\fcat_{*//}}((\L\Mor_\mC(X,Y),\alpha),(\L\Mor_\mC(X,Z),\beta)) $$
$$ \simeq \L\Mor_\mC(Y,Z) \times_{\Fun^\oplax(\L\Mor_\mC(X,Y),\L\Mor_\mC(X,Z))}
\L\Mor_{\infty\fcat_{*//}}((\L\Mor_\mC(X,Y),\alpha),(\L\Mor_\mC(X,Z),\beta)) $$
$$\simeq \L\Mor_\mC(Y,Z) \times_{\Fun^\oplax(\L\Mor_\mC(X,Y),\L\Mor_\mC(X,Z))}
\Fun^\oplax(\L\Mor_\mC(X,Y),\L\Mor_\mC(X,Z)) $$$$ \times_{\Fun^\oplax(\{0\},\L\Mor_\mC(X,Z))} \Fun^\oplax(\bD^1,\L\Mor_\mC(X,Z)) \times_{\Fun^\oplax(\{1\},\L\Mor_\mC(X,Z))} \{\beta\}.$$
\end{proof}

\bibliographystyle{plain}
\bibliography{mainbib}

\end{document}